\documentclass[twoside]{article}
\usepackage[letterpaper, left=3.5cm, right=3.5cm, top=4cm, bottom=2.5cm]{geometry}
\usepackage{graphicx}
\usepackage{amsmath,amssymb,amsthm}
\usepackage{authblk}
\usepackage[utf8]{inputenc}
\usepackage{microtype}
\usepackage{mathrsfs} 
\usepackage{enumitem}
\usepackage{calc}
\usepackage{esint}
\usepackage{fancyhdr}
\usepackage{xcolor}
\usepackage[labelfont = bf]{caption}
\usepackage{doi}
\usepackage{subfigure}
\usepackage[numbers]{natbib}
\usepackage{float}
\usepackage{stmaryrd}
\usepackage{mathtools}
\usepackage{booktabs}
\usepackage{colortbl}
\usepackage{tabulary}
\usepackage{diagbox}
\usepackage{caption}
\usepackage{orcidlink}

\usepackage{physics}
\usepackage{amsmath}
\usepackage{tikz}
\usepackage{mathdots}
\usepackage{yhmath}
\usepackage{cancel}
\usepackage{color}
\usepackage{siunitx}
\usepackage{array}
\usepackage{multirow}
\usepackage{amssymb}
\usepackage{gensymb}
\usepackage{tabularx}
\usepackage{extarrows}
\usepackage{booktabs}
\usetikzlibrary{fadings}
\usetikzlibrary{patterns}
\usetikzlibrary{shadows.blur}
\usetikzlibrary{shapes}

\newtheorem{theorem}{Theorem}[section]
\numberwithin{equation}{section}
\newtheorem{proposition}[theorem]{Proposition}

\newtheorem{remark}[theorem]{Remark}
\newtheorem{lemma}[theorem]{Lemma}

\newtheorem{algorithm}[theorem]{Algorithm}

\usepackage{titlesec}
\titleformat{\section}{\normalfont\scshape\centering}{\thesection.}{0.5em}{}
\titleformat*{\subsection}{\itshape}
\titleformat*{\subsubsection}{\itshape}

\providecommand{\keywords}[1]
{
	{\small\emph{Keywords:} #1}
}

\providecommand{\MSC}[1]
{
	{\small\emph{AMS MSC (2020):~~} #1}
}

\definecolor{denim}{rgb}{0.08, 0.38, 0.74}
\definecolor{byzantium}{rgb}{0.44, 0.16, 0.39} 
\definecolor{shamrockgreen}{rgb}{0.0, 0.62, 0.38} 

\usepackage{hyperref}
\hypersetup{
	colorlinks=true,
	linkcolor=denim,
	citecolor = shamrockgreen,
	filecolor=magenta, 
	urlcolor=byzantium,
} 

\providecommand{\jumptmp}[2]{#1\llbracket{#2}#1\rrbracket}
\providecommand{\jump}[1]{\jumptmp{}{#1}}

\begin{document}
	\setlength{\abovedisplayskip}{5.5pt}
	\setlength{\belowdisplayskip}{5.5pt}
	\setlength{\abovedisplayshortskip}{5.5pt}
	\setlength{\belowdisplayshortskip}{5.5pt}

	\title{\vspace*{-20mm}Variational problems with gradient constraints:\\  
 \textit{A priori} and \textit{a posteriori} error identities\thanks{HA is partially supported by the Office of Naval Research (ONR) under Award NO: N00014-24-1-2147, NSF grant DMS-2408877, the Air Force Office of Scientific Research (AFOSR) under Award NO: FA9550-22-1-0248.}}
	\author[1]{Harbir Antil\,\orcidlink{0000-0002-6641-1449}\thanks{Email: \texttt{\url{hantil@gmu.edu}}}}
	\author[2]{Sören Bartels\thanks{Email: \texttt{\url{bartels@mathematik.uni-freiburg.de}}}}
	\author[3]{Alex Kaltenbach\,\orcidlink{0000-0001-6478-7963}\thanks{Email: \texttt{\url{kaltenbach@math.tu-berlin.de}}}}
	\author[4]{Rohit Khandelwal\,\orcidlink{0000-0001-7063-0508}\thanks{Email: \texttt{\url{rkhandel@sau.int}}}}
	\date{October 25, 2024}
	\affil[2]{\small{Department of Applied Mathematics, University of Freiburg, Hermann--Herder--Str. 10,  79104  Freiburg, GERMANY}}
	\affil[1]{\small{Department of Mathematical Sciences and the Center for Mathematics and Artificial Intelligence (CMAI), George Mason University, Fairfax, VA 22030, USA.}}
	\affil[3]{\small{Institute of Mathematics, Technical University of Berlin, Stra\ss e des 17.\ Juni 135, 10623 Berlin, GERMANY}}
    \affil[4]{\small{Department of Mathematics, South Asian University, New Delhi, India 110068}}
	\maketitle

	\pagestyle{fancy}
	\fancyhf{}
	\fancyheadoffset{0cm}
	\addtolength{\headheight}{-0.25cm}
	\renewcommand{\headrulewidth}{0pt} 
	\renewcommand{\footrulewidth}{0pt}
	\fancyhead[CO]{\textsc{variational problems with gradient constraints: error identities}}
	\fancyhead[CE]{\textsc{H. Antil, S. Bartels, A. Kaltenbach, and R. Khandelwal}}
	\fancyhead[R]{\thepage}
	\fancyfoot[R]{}
	
	\begin{abstract}
		In this paper, on the basis of a (Fenchel) duality theory on the continuous level, we derive an \textit{a posteriori} error identity for arbitrary conforming approximations of a primal formulation and a dual formulation of  variational problems involving gradient constraints.
		In addition, on the basis of a (Fenchel) duality theory on the discrete level, we derive an \textit{a priori} error identity that applies
		to the approximation of the primal formulation using the Crouzeix--Raviart element and  to  the approximation of the dual formulation using the Raviart--Thomas element, and leads to error decay rates that are optimal with respect to the regularity of a dual solution.
	\end{abstract}
	
	
	\MSC{49J40; 49M29; 65N30; 65N15; 65N50.}
	
	\section{Introduction}\thispagestyle{empty}
	
    \hspace{5mm}The present article aims to study the following variational problem with gradient constraints:\enlargethispage{20mm} 
\begin{equation}\label{eq:var_model}
\begin{aligned} 
		& \min_{v \in K} \big\{ I(v) \coloneqq  \tfrac{1}{2}\| \nabla v\|_{2,\Omega}^2-(f,v)_{\Omega}-\langle g,v\rangle_{\Gamma_N} \big\}
        & \mbox{}
\end{aligned}   
\end{equation}
over the convex set 
\begin{align*}
    K\coloneqq \big\{v\in W^{1,\infty}(\Omega)\mid \vert \nabla v\vert\leq \zeta \text{ a.e.\ in }\Omega\,,\; v=u_D\text{ a.e.\ on }\Gamma_D\big\}\,.
\end{align*}
Here, $\Omega \subseteq \mathbb{R}^d$, $d \in \mathbb{N}$, is a bounded simplicial Lipschitz domain with (topological) boundary~$\partial\Omega$ that is split into the Dirichlet (\textit{i.e.}, $\Gamma_D$) and Neumann (\textit{i.e.}, $\Gamma_N$) parts.~The~\mbox{functions}~${f\hspace*{-0.1em}\in \hspace*{-0.1em}L^1(\Omega)}$,~$g\hspace*{-0.1em}\in\hspace*{-0.1em} W^{-1,1}(\Gamma_N)$, $u_D\in W^{1,\infty}(\Gamma_D)$ represent the load, the Neumann and the Dirichlet boundary~data, respectively. Apart from that, the function $\zeta\in L^\infty(\Omega)$ represents the gradient obstacle function. Under generic assumptions on the data (\textit{cf}.~Section~\ref{sec:continuous}), by standard arguments, one can establish that \eqref{eq:var_model} admits a unique solution (\textit{cf}.\ \cite{MR0464857,MR0567696}).
 
This model problem arises in various applications, \textit{e.g.}, (i) \emph{elasto-plastic torsion in structural  engineering}  (\textit{cf}.\ \cite[Chap.\ 2]{MR0597520}), 
(ii) \emph{stochastic optimal control} (\textit{cf}.\ \cite{soner1991free}), 
(iii)~\emph{mathematical~finance} (\textit{cf}.~\cite{MR1284980}). Despite it being of significant interest, there exist (a) no provably convergent solvers, which also take discretization into account, for this problem  and (b) no \textit{a posteriori} or \textit{a priori} error estimates in full generality. The goal of this paper is to close both these~gaps.~During~this~process, several new tools, which are of interest on their own, have been developed.

A key challenge in developing a solver for \eqref{eq:var_model} is how to handle the projection onto the constraint set $K$. This is highly non-trivial no matter what  discretization is used.~To~overcome~this, motivated by \cite{BartelsKaltenbachOverview24}, we consider a primal-dual formulation. The main contributions are stated next:\vspace{-2.5mm}

\subsection{Main contributions}\vspace{-1mm}

\begin{enumerate}[noitemsep,topsep=2pt,leftmargin=!,labelwidth=\widthof{\quad 6.},font=\itshape]
    \item This article works under {\it full generality} of problem \eqref{eq:var_model}: more precisely, for the data~in~\eqref{eq:var_model}, we only assume that
    $f\in L^1(\Omega)$, $g\in W^{-1,1}(\Gamma_N)$, $\zeta\in L^\infty(\Omega)$, and $u_D\in W^{1,\infty}(\Gamma_D)$ such that there exists a trace lift $\widehat{u}_D\in W^{1,\infty}(\Omega)$ satisfying $\smash{\|\frac{\nabla\widehat{u}_D}{\zeta}\|_{\infty,\Omega}}<1$. In this context,~the Sobolev space $W^{1,\infty}(\Omega)$ turns out to provide the correct norm topology~on~the~convex~set~$K$, so that well-established convex duality methods (\textit{e.g.}, the celebrated Fenchel duality theorem (\emph{cf}.\ \cite[Rem.\ 4.2, (4.21), p.\ \hspace*{-0.1mm}61]{ET99})) can be applied.
    
    \item A {\it thorough characterization} of the (Fenchel) dual problem (in the sense of \cite[Rem.\ 4.2, p.\ 60/61]{ET99}) as well as convex optimality relations are provided in Theorem~\ref{thm:duality}, including a strong duality relation. The derived (Fenchel) dual problem is defined on $(L^\infty(\Omega))^*$ (\textit{i.e.}, the dual space of essentially bounded Lebesgue measurable functions $L^\infty(\Omega)$), which~is isometrically isomorphic to $(\mathrm{ba}(\Omega))^d$ (\textit{i.e.}, the space of bounded and finitely additive~vector~measures). 

    \item At the continuous level, for arbitrary conforming approximations of the primal and dual problem, an {\it a posteriori} error estimator (\textit{i.e.}, \emph{primal-dual gap estimator}) is derived in~Lemma~\ref{lem:primal_dual_gap_estimator}. This is followed by the \emph{primal-dual total error} in Lemma~\ref{lem:strong_convexity_measures}. Both these error quantities are shown to be exactly equal by the \emph{primal-dual gap identity} in Theorem~\ref{thm:prager_synge_identity}.

    \item The results from the second and third bullets above also hold at the {\it discrete level for appropriate finite element discretizations}. In particular, Theorem~\ref{thm:discrete_duality} establishes a strong discrete duality relation and convex optimality relations. Theorem~\ref{thm:discrete_reconstruction} shows the reconstruction of the discrete primal solution from the discrete dual solution via an {\it inverse generalized~Marini~formula}.

    \item Discrete primal-dual gap estimator and discrete primal-dual total errors are respectively derived in Lemmas~\ref{lem:discrete_primal_dual_gap_estimator} and \ref{lem:discrete_strong_convexity_measures}. They are again shown to be equal in Theorem~\ref{thm:discrete_prager_synge_identity}. This is followed by \emph{a priori error estimates} in Theorem~\ref{thm:apriori_identity}.

    \item 
    The outer loop of the main algorithm drives the primal-dual gap estimator to zero~(\textit{e.g.}, via mesh refinement). The inner loop consists of solving a dual sub-problem. An example of an inner loop via gradient flow is provided in Algorithm~\ref{algorithm}. Other alternatives include Newton or quasi-Newton methods. The convergence of the inner loop is established in Proposition~\ref{prop:stability}.\vspace{-2.5mm}\enlargethispage{5mm}
\end{enumerate}

\subsection{Related contributions}\vspace{-1mm}
\hspace{5mm}The existing body of literature largely either focuses on the theoretical analysis or on the numerical analysis of simplified problems: 
Under smooth assumptions on the displacement corresponding to \eqref{eq:var_model}, the article \cite{Falk1977} initially derived the \textit{a priori} error estimates for~the~stress~field. The contributions \cite{bildhauer2009elastic,MR4432040,MR4593745} provide a numerical~study~of~\eqref{eq:var_model}. They operate under a simplified scenario, in which $f\in L^1(\Omega)$ and $\zeta\in L^\infty(\Omega)$ are constants and $\partial\Omega = \Gamma_D$ with $u_D = 0$. Under these assumptions,  problem \eqref{eq:var_model} simplifies to the classical~obstacle problem, 
where the obstacle function is given via the distance function to boundary $\partial\Omega$ (\textit{cf}.\ \cite[Thm.\ 1.2]{MR1989924}). They derive \textit{a priori} error estimates using classical techniques and under higher regularity assumptions. The contribution \cite{bildhauer2009elastic} equally derives an \textit{a posteriori} error estimator using a convex duality approach. However, it poses the primal problem on~the~Sobolev~space~$W_0^{1,2}(\Omega)$. As we notice in this paper, this could lead to an ill-posed dual problem as, then, the celebrated Fenchel duality theorem  (\emph{cf}.\ \cite[Rem.\ 4.2, (4.21), p.\ \hspace*{-0.1mm}61]{ET99})~is~not~applicable. The recent contribution~\cite{unknown} treats non-constant $f\in L^1(\Omega)$ and imposes the gradient  constraint via penalization. 
\textit{A priori} error estimates between the penalized solution and its approximation are derived. However, the complete \textit{a priori} error estimate is stated as an open problem in \cite[Rem.\ 6.1]{unknown}.

	\section{Preliminaries}\label{sec:preliminaries}\vspace{-1mm}
	
	\hspace{5mm}Throughout the entire  paper,  let ${\Omega\subseteq \mathbb{R}^d}$, ${d\in\mathbb{N}}$, be a bounded simplicial Lipschitz domain such that $\partial\Omega$ is divided into two disjoint
	(relatively) open sets:~a~\mbox{Dirichlet}~part~${\Gamma_D\subseteq \partial\Omega}$~with~${\vert \Gamma_D\vert>0}$\footnote{For a (Lebesgue) measurable set $M\subseteq \mathbb{R}^d$, $d\in \mathbb{N}$, we denote by $\vert M\vert $ its $d$-dimensional Lebesgue measure. For a $(d-1)$-dimensional submanifold $M\subseteq \mathbb{R}^d$, $d\in \mathbb{N}$, we denote by $\vert M\vert $ its $(d-1)$-dimensional~Hausdorff~measure.} and a Neumann part $\Gamma_N\subseteq \partial\Omega$ such that $\partial\Omega=\overline{\Gamma}_D\cup\overline{\Gamma}_N$.

    For a (Lebesgue) measurable set $\omega\subseteq \mathbb{R}^n$, $n\in \mathbb{N}$, and  (Lebesgue) measurable functions or vector fields $v,w\colon \omega\to \mathbb{R}^{\ell}$, $\ell\in\mathbb{N}$, we employ~the~inner~product $(v,w)_{\omega}\coloneqq \int_{\omega}{v\odot w\,\mathrm{d}x}$, 
	whenever the right-hand side is well-defined, where $\odot\colon \mathbb{R}^{\ell}\times \mathbb{R}^{\ell}\to \mathbb{R}$ either denotes scalar multiplication~or~the Euclidean inner product. The integral mean over a (Lebesgue)~\mbox{measurable}~set~${\omega\subseteq\mathbb{R}^n}$,~${n\in \mathbb{N}}$,~with $\vert \omega\vert>0$ of an integrable function or vector field $v\colon \omega\to \mathbb{R}^{\ell}$, $\ell\in\mathbb{N}$,~is~defined~by~${\langle v\rangle_\omega\coloneqq \smash{\frac{1}{\vert \omega\vert}\int_{\omega}{v\,\mathrm{d}x}}}$.\vspace{-1mm}
	
	\subsection{Classical function spaces}\vspace{-0.5mm}
	
	\hspace*{5mm}For $m\in \mathbb{N}$, $p\in [1,\infty]$, and an open set $\omega\subseteq \mathbb{R}^n$, $n\in \mathbb{N}$,~we~define 
	\begin{align*} 
		W^{m,p}(\omega)\coloneqq \big\{v\in L^p(\omega)\mid \mathrm{D}^{\boldsymbol{\alpha}} v\in L^p(\omega)\textup{ for all }\boldsymbol{\alpha}\in (\mathbb{N}_0)^n\text{ with }\vert \boldsymbol{\alpha}\vert \leq m\big\}\,,
	\end{align*}
	where $\mathrm{D}^{\boldsymbol{\alpha}}\coloneqq \frac{\partial^{\vert \boldsymbol{\alpha}\vert }}{\partial x_1^{\alpha_1}\cdot\ldots\cdot \partial x_n^{\alpha_n}}$ and $\vert \boldsymbol{\alpha}\vert \coloneqq\sum_{i=1}^n{\alpha_i}$ 
	for each multi-index ${\boldsymbol{\alpha}\coloneqq (\alpha_1,\ldots,\alpha_n)\in (\mathbb{N}_0)^n}$, equipped with the \emph{Sobolev  norm} $\|\cdot\|_{m,p,\omega}\coloneqq \|\cdot\|_{p,\omega}+\vert \cdot\vert_{m,p,\omega}$,~where~$\|\cdot\|_{p,\omega}\coloneqq (\int_\omega{\vert \cdot\vert^p\,\mathrm{d}x})^{\smash{\frac{1}{p}}}$~and
	\begin{align*} 
		\vert\cdot\vert _{m,p,\omega}\coloneqq\Bigg(\sum_{\boldsymbol{\alpha}\in (\mathbb{N}_0)^n\,:\,0<\vert\boldsymbol{\alpha}\vert\leq m }{\|\mathrm{D}^{\boldsymbol{\alpha}}(\cdot)\|_{p,\omega}^p}\Bigg)^{\smash{\frac{1}{p}}}\,.
	\end{align*}  
	Then, for $s \in (0,\infty)\setminus \mathbb{N}$, $p\in [1,\infty]$, and an open set $\omega\subseteq \mathbb{R}^n$, $n\in \mathbb{N}$, the  \emph{Sobolev--Slobodeckij~semi-norm}, for every $v \in  W^{m,p}(\omega)$,~is~defined~by\vspace{-0.5mm}
	\begin{align*}
		\vert v\vert_{s,p,\omega}\coloneqq \Bigg(\sum_{\vert \boldsymbol{\alpha}\vert= m}{\int_{\omega}{\int_{\omega}{\frac{\vert(\mathrm{D}^{\boldsymbol{\alpha}} v)(x)-(\mathrm{D}^{\boldsymbol{\alpha}} v)(y)\vert^p}{\vert x-y\vert^{p\theta+d}}\,\mathrm{d}x}\,\mathrm{d}y}}\Bigg)^{\smash{\frac{1}{p}}}\,,
	\end{align*}
	where $m\in \mathbb{N}_0$ and $\theta\in (0,1)$ are such that $s=m+\theta$.	
	Then, for $s \in (0,\infty)\setminus \mathbb{N}$ and an open set $\omega\subseteq \mathbb{R}^n$, $n\in \mathbb{N}$,  the \emph{Sobolev--Slobodeckij space} is defined by\enlargethispage{11.5mm}
	\begin{align*}
		\smash{W^{s,p}(\omega)\coloneqq \big\{v\in W^{m,p}(\omega)\mid \vert v\vert _{s,p,\omega}<\infty\big\}\,,}
	\end{align*} 
	where $m\hspace*{-0.1em}\in\hspace*{-0.1em} \mathbb{N}_0$ and $\theta\hspace*{-0.1em}\in\hspace*{-0.1em} (0,1)$ are such that $s\hspace*{-0.1em}=\hspace*{-0.1em}m+\theta$.
	
	\subsubsection{Integration-by-parts formula}\vspace{-0.5mm}
	
	\hspace*{5mm}For $m\in \mathbb{N}$, $p,p'\in [1,\infty]$, where $\frac{1}{p}+\frac{1}{p'}=1$, and an open set $\omega\subseteq \mathbb{R}^n$, $n\in \mathbb{N}$,~we~define\vspace{-1mm}
	\begin{align*} 
		\smash{W^{\smash{p'}}(\textup{div};\Omega)\coloneqq \big\{y\in (L^{\smash{p'}}(\Omega))^n\mid \textup{div}\,y\in L^{\smash{p'}}(\Omega)\big\}\,.}
	\end{align*}
	Next, denote by $\textup{tr}(\cdot)\colon W^{1,p}(\Omega)\to W^{\smash{1-\frac{1}{p}},p}(\partial\Omega)$ and ${\textup{tr}((\cdot)\cdot n)\colon W^{\smash{p'}}(\textup{div};\Omega)\to (W^{\smash{1-\frac{1}{p}},p}(\partial\Omega))^*}$ the trace operator  and the normal trace operator, respectively, where $n\colon \partial\Omega\to \mathbb{S}^{d-1}$ denotes the outward unit normal vector~field~to~$\partial\Omega$. 
	Then,  for every $v\in W^{1,p}(\Omega)$ and $y\in W^{\smash{p'}}(\textup{div};\Omega)$, there holds the integration-by-parts formula (\textit{cf}.\ \cite[Sec.\ 4.3, (4.12)]{EG21II}) 
	\begin{align}\label{eq:pi_cont}
		\smash{(\nabla v,y)_{\Omega}+(v,\textup{div}\,y )_{\Omega}=\langle \textup{tr}(y\cdot n),\textup{tr}(v)\rangle_{\partial\Omega}\,,}
	\end{align}
	where, for every $\widehat{y}\in \smash{(W^{\smash{1-\frac{1}{p}},p}(\partial\Omega))^*}$, $\widehat{v}\in \smash{W^{\smash{1-\frac{1}{p}},p}(\gamma)}$, and $\gamma\in \{\Gamma_D,\Gamma_N,\partial\Omega\}$, we  abbreviate
	\begin{align}\label{eq:abbreviation}
		\smash{\langle \widehat{y},\widehat{v}\rangle_{\gamma}\coloneqq \langle \widehat{y},\widehat{v}\rangle_{W^{1-\smash{\frac{1}{p}},p}(\gamma)}\,.}
	\end{align}
	Eventually, we employ the notation
	\begin{align*} 
		W^{1,p}_D(\Omega)&\coloneqq  \big\{v\in 	W^{1,p}(\Omega) \mid \textup{tr}(v)=0\textup{ a.e.\ on }\Gamma_D\big\}\,,\\
		\smash{W^{p'}_N(\textup{div};\Omega)}&\coloneqq  \big\{y\in W^{\smash{p'}}(\textup{div};\Omega)\mid \langle\textup{tr}(y\cdot n),\textup{tr}(v)\rangle_{\partial\Omega} =0\text{ for all  }v\in W^{1,p}_D(\Omega)\big\}\,.
	\end{align*}
	In what follows, we omit writing both $\textup{tr}(\cdot)$ and $\textup{tr}((\cdot)\cdot n)$ in this context.
	
	\subsubsection{Bounded and finitely additive vector measures} 
	
	\hspace{5mm}Denote by $\mathcal{M}(\mathrm{d}x;\Omega)$ the $\sigma$-algebra of Lebesgue measurable sets. Then,  the \textit{space of bounded and finitely additive vector measures on $\mathcal{M}(\mathrm{d}x;\Omega)$} is defined by (\textit{cf}.\ \cite[Def.\ 4.1]{Toland20})\enlargethispage{15mm}
	\begin{align*}
		(\mathrm{ba}(\Omega))^d\coloneqq \bigg\{\nu \colon \mathcal{M}(\mathrm{d}x;\Omega)\to \mathbb{R}^d\;\bigg|\;\begin{aligned}
			\nu(\emptyset)=0\,,\;\textcolor{black}{\vert\nu\vert(\Omega)}<\infty\,,\; \nu(A\cup B)=\nu(A)+\nu(B)\\ 
			\text{ for all }A,B\in \mathcal{M}(\mathrm{d}x;\Omega)\text{ with }A\cap B=\emptyset
		\end{aligned}\,\bigg\}\,.
	\end{align*}
	The \textit{total variation norm}, 
	for every $\nu \in (\mathrm{ba}(\Omega))^d$, is defined by
	\begin{align*}
		\vert \nu \vert (\Omega)\coloneqq \sup_{(A_i)_{i\in \mathbb{N}}\in \pi(\Omega)}{\sum_{i\in \mathbb{N}}{\vert \nu(A_i)\vert}}\,,
	\end{align*}
	where $\pi(\Omega)\subseteq 2^{\mathcal{M}(\mathrm{d}x;\Omega)}$ is the set of all  countable, disjoint partitions of $\Omega$ into Lebesgue~measurable sets.  By Lebesgue's measure decomposition theorem, for every $\nu \in  (\mathrm{ba}(\Omega))^d$, there exist unique $y\in (L^1(\Omega))^d$ and $\nu^s\in (\mathrm{ba}(\Omega))^d$ with $\nu^s\perp y\otimes \mathrm{d}x$\footnote{For every $y\in (L^1(\Omega))^d$, we employ the notation $\langle y\otimes \mathrm{d}x,\widehat{y}\rangle_{(L^\infty(\Omega))^d}\coloneqq \int_{\Omega}{y\cdot \widehat{y}\,\mathrm{d}x}$ for all $\widehat{y}\in (L^\infty(\Omega))^d$.}, such that (\textit{cf}.\ \cite[Thm.\ 2.28]{Toland20})
	\begin{align}\label{eq:lebesgue_decomposition}
		\begin{aligned}
			\nu &=y\otimes \mathrm{d}x+\nu^s\quad\text{ in } (\mathrm{ba}(\Omega))^d\,.
		\end{aligned}
	\end{align}
	The dual space of $(L^\infty(\Omega))^d$ (\textit{i.e.}, $((L^\infty(\Omega))^d)^*$) is isometrically isomorphic~to~$(\mathrm{ba}(\Omega))^d$;~in~the~sense that for every $y^*\in ((L^\infty(\Omega))^d)^*$, there exists a unique 
	$\nu\in (\mathrm{ba}(\Omega))^d$ such that (\textit{cf}.\ \cite[Thm.~3.1]{Toland20})
	\begin{align*}
		\langle y^*,y\rangle_{(L^\infty(\Omega))^d}&=\int_{\Omega}{y\,\mathrm{d}\nu}\quad\text{ for all }y\in (L^\infty(\Omega))^d\,,\\
		\|y^*\|_{((L^\infty(\Omega))^d)^*}&=\vert \nu\vert(\Omega)\,.
	\end{align*}

	\subsection{Triangulations and standard finite element spaces}
	
	\hspace{5mm}In what follows, we denote by $\{\mathcal{T}_h\}_{h>0}$ a family of uniformly shape regular triangulations~of $\Omega\subseteq \mathbb{R}^d$, $d\in\mathbb{N}$, (\textit{cf}.\  \cite{EG21I}), where
	$h>0$ refers to the \textit{averaged mesh-size}, \textit{i.e.}, ${h 
	\coloneqq (\vert \Omega\vert/\textup{card}(\mathcal{N}_h))^{\frac{1}{d}}}
	$, where $\mathcal{N}_h$  contains the vertices of $\mathcal{T}_h$. We define  the following sets of sides:
	\begin{align*}
		\mathcal{S}_h&\coloneqq \mathcal{S}_h^{i}\cup \mathcal{S}_h^{\partial\Omega}\,,\\
		\mathcal{S}_h^{i}&\coloneqq  \{T\cap T'\mid T,T'\in\mathcal{T}_h\,,\text{dim}_{\mathscr{H}}(T\cap T')=d-1\}\,,\\
		\mathcal{S}_h^{\partial\Omega}&\coloneqq\{T\cap \partial\Omega\mid T\in \mathcal{T}_h\,,\text{dim}_{\mathscr{H}}(T\cap \partial\Omega)=d-1\}\,,\\
		\mathcal{S}_h^\gamma&\coloneqq\{S\in \mathcal{S}_h^{\partial\Omega}\mid \textup{int}(S)\subseteq \gamma\}\text{ for } \gamma\in \{\Gamma_D,\Gamma_N\}\,,
	\end{align*}
	where the Hausdorff dimension is defined by $\text{dim}_{\mathscr{H}}(M)\hspace{-0.15em}\coloneqq\hspace{-0.15em}\inf\{d'\hspace{-0.15em}\geq\hspace{-0.15em} 0\mid \mathscr{H}^{d'}(M)\hspace{-0.15em}=\hspace{-0.15em}0\}$~for~all~${M\hspace{-0.15em}\subseteq \hspace{-0.15em} \mathbb{R}^d}$. 
	It is also assumed that the triangulations $\{\mathcal{T}_h\}_{h>0}$ and  boundary parts $\Gamma_D$~and~$\Gamma_N$~are~chosen such that  $\mathcal{S}_h^{\partial\Omega}\hspace*{-0.1em}=\hspace*{-0.1em}\mathcal{S}_h^{\Gamma_D}\dot{\cup} \mathcal{S}_h^{\Gamma_N}$, \textit{e.g.}, in the case $d\hspace*{-0.1em}=\hspace*{-0.1em}2$,  the boundary parts $\overline{\Gamma}_D$ and $\overline{\Gamma}_N$~touch~only~in~\mbox{vertices}.

	For $k\in \mathbb{N}\cup\{0\}$ and $T\in \mathcal{T}_h$, let $\mathbb{P}^k(T)$ denote the set of polynomials of maximal~degree~$k$~on~$T$. Then, for $k\in \mathbb{N}\cup\{0\}$, the set of  element-wise polynomial functions  is defined by
	\begin{align*}
		\smash{\mathcal{L}^k(\mathcal{T}_h)\coloneqq \big\{v_h\in L^\infty(\Omega)\mid v_h|_T\in\mathbb{P}^k(T)\text{ for all }T\in \mathcal{T}_h\big\}\,.}
	\end{align*} 
	For $\ell \in \mathbb{N}$, the (local) $L^2$-projection $\Pi_h\colon  (L^1(\Omega))^{\ell} \to (\mathcal{L}^0(\mathcal{T}_h))^{\ell} $ onto element-wise constant functions or vector fields, respectively, for every 
	$v\in  (L^1(\Omega))^{\ell} $ is defined by $\Pi_h v|_T\coloneqq\langle v\rangle_T$~for~all~${T\in \mathcal{T}_h}$.  
	
	For $m\in \mathbb{N}\cup\{0\}$ and $S\in \mathcal{S}_h$, let $\mathbb{P}^m(S)$ denote the set of polynomials of maximal degree~$m$~on~$S$. Then, for $m \in \mathbb{N}\cup\{0\}$, 
	the set of side-wise polynomial functions  is defined by
	\begin{align*}
		\smash{\mathcal{L}^m(\mathcal{S}_h)\coloneqq  \big\{v_h\in L^\infty(\cup\mathcal{S}_h)\mid v_h|_{\textcolor{black}{S}}\in\mathbb{P}^m(S)\text{ for all }S\in \mathcal{S}_h\big\}\,.}
	\end{align*} 
	For $\ell  \hspace*{-0.1em}\in  \hspace*{-0.1em}\mathbb{N}$, the (local) $L^2$-projection $\pi_h\colon  \hspace*{-0.1em}(L^1(\cup\mathcal{S}_h))^{\ell} \hspace*{-0.1em}\to \hspace*{-0.1em} (\mathcal{L}^0(\mathcal{S}_h))^{\ell}$ onto side-wise constant functions or vector  fields, respectively,  for every 
	$v\in  (L^1(\cup\mathcal{S}_h))^{\ell} $ is defined by ${\pi_h v|_S\coloneqq \langle v\rangle_S}$~for~all~${S\in \mathcal{S}_h}$. 
	
	For  every $v_h\in \mathcal{L}^k(\mathcal{T}_h)$ and $S\in\mathcal{S}_h$, the \textit{jump across} $S$ is defined by 
	\begin{align*}
		\jump{v_h}_S\coloneqq\begin{cases}
			v_h|_{T_+}-v_h|_{T_-}&\text{ if }S\in \mathcal{S}_h^{i}\,,\text{ where }T_+, T_-\in \mathcal{T}_h\text{ satisfy }\partial T_+\cap\partial T_-=S\,,\\
			v_h|_T&\text{ if }S\in\mathcal{S}_h^{\partial\Omega}\,,\text{ where }T\in \mathcal{T}_h\text{ satisfies }S\subseteq \partial T\,.
		\end{cases}
	\end{align*}
	
	For  every $y_h\in (\mathcal{L}^k(\mathcal{T}_h))^d$ and $S\in\mathcal{S}_h$, the \emph{normal jump across} $S$ is defined by 
	\begin{align*}
		\jump{y_h\cdot n}_S\coloneqq\begin{cases}
			y_h|_{T_+}\!\cdot n_{T_+}+y_h|_{T_-}\!\cdot n_{T_-}&\text{ if }S\in \mathcal{S}_h^{i}\,,\text{ where }T_+, T_-\in \mathcal{T}_h\text{ satisfy }\partial T_+\cap\partial T_-=S\,,\\
			y_h|_T\cdot n&\text{ if }S\in\mathcal{S}_h^{\partial\Omega}\,,\text{ where }T\in \mathcal{T}_h\text{ satisfies }S\subseteq \partial T\,,
		\end{cases}
	\end{align*}
	where, for every $T\in \mathcal{T}_h$, $\smash{n_T\colon\partial T\to \mathbb{S}^{d-1}}$ denotes the outward unit normal vector field to $ T$.\vspace{-0.5mm}
	
	\subsubsection{Crouzeix--Raviart element}\vspace{-0.5mm}
	
	\hspace{5mm}The \textit{Crouzeix--Raviart finite element space} (\textit{cf}.\ \cite{CR73}) is defined as 
	\begin{align*}
		\mathcal{S}^{1,cr}(\mathcal{T}_h)\coloneqq \big\{v_h\in \mathcal{L}^1(\mathcal{T}_h)\mid \pi_h\jump{v_h}=0\text{ a.e.\ on }\cup \mathcal{S}_h^{i}\big\}\,.
	\end{align*}
	The Crouzeix--Raviart finite element space with homogeneous Dirichlet boundary condition~on~$\Gamma_D$ is defined by\vspace{-1mm}\enlargethispage{5mm}
	\begin{align*}
		\smash{\mathcal{S}^{1,cr}_D(\mathcal{T}_h)}\coloneqq \big\{v_h\in\smash{\mathcal{S}^{1,cr}(\mathcal{T}_h)}\mid  \pi_h v_h=0\text{ a.e.\ on }\cup \mathcal{S}_h^{\Gamma_D}\big\}\,.
	\end{align*} 
	Denote by 	$\varphi_S \in \smash{\mathcal{S}^{1,cr}(\mathcal{T}_h)}$, $S \in \mathcal{S}_h$, satisfying 
	$\varphi_S(x_{S'}) = \delta_{S,S'}$~for~all~$S,S' \in \mathcal{S}_h$,~a~\mbox{basis}~of~$\smash{\mathcal{S}^{1,cr}(\mathcal{T}_h)}$.
	Then, the (Fortin) quasi-interpolation operator $\smash{\Pi_h^{cr}\colon W^{1,1}(\Omega)\to \smash{\mathcal{S}^{1,\textit{cr}}(\mathcal{T}_h)}}$  (\textit{cf}.\ \cite[Secs.\ 36.2.1, 36.2.2]{EG21II}), for every $v\in W^{1,1}(\Omega)$ defined by
	\begin{align}
		\Pi_h^{cr}v\coloneqq \sum_{S\in \mathcal{S}_h}{\langle v\rangle_S\,\varphi_S}\,,\label{CR-interpolant}
	\end{align}
	preserves averages of gradients and moments (on sides), \textit{i.e.}, for every $v\in W^{1,1}(\Omega)$, it holds that
	\begin{alignat}{2}
		\nabla_h\Pi_h^{cr}v&=\Pi_h\nabla v&& \quad\text{ a.e.\ in }\Omega\,,\label{eq:grad_preservation}\\
		\pi_h\Pi_h^{cr}v&=\pi_h  v&&\quad \text{ a.e.\ on }\cup\mathcal{S}_h\, ,\label{eq:trace_preservation}
	\end{alignat}
	where $\nabla_h\colon \mathcal{L}^1(\mathcal{T}_h)\to (\mathcal{L}^0(\mathcal{T}_h))^d$ is  defined by $(\nabla_hv_h)|_T\coloneqq \nabla(v_h|_T)$ for all $v_h\in \mathcal{L}^1(\mathcal{T}_h)$~and~${T\in \mathcal{T}_h}$.\vspace{-0.5mm}
	
	\subsubsection{Raviart--Thomas element}\vspace{-0.5mm}
	
	\hspace{5mm}The \textit{(lowest order) Raviart--Thomas finite element space} (\textit{cf}.\ \cite{RT77}) is defined as\enlargethispage{2mm} 
	\begin{align*}
		\mathcal{R}T^0(\mathcal{T}_h)\coloneqq \bigg\{y_h\in(\mathcal{L}^1(\mathcal{T}_h))^d\;\bigg|\; \begin{aligned}
			y_h|_T\cdot n_T=\textup{const}\text{ on }\partial T\text{ for all }T\in \mathcal{T}_h\,,\\[-1mm] 
			\jump{y_h\cdot n}_S=0\text{ on }S\text{ for all }S\in \mathcal{S}_h^{i}
		\end{aligned}\,\bigg\}\,.
	\end{align*} 
	The Raviart--Thomas finite element space with homogeneous normal boundary condition on $\Gamma_N$ is defined by\vspace{-1mm}
	\begin{align*}
		\mathcal{R}T^{0}_N(\mathcal{T}_h)\coloneqq \big\{y_h\in	\mathcal{R}T^0(\mathcal{T}_h)\mid y_h\cdot n=0\text{ a.e.\ on }\Gamma_N\big\}\,.
	\end{align*} 
	Denote by $\psi_S\in \mathcal{R}T^0(\mathcal{T}_h)$, $S\in \mathcal{S}_h$, satisfying   $\psi_S|_{S'}\cdot n_{S'}=\delta_{S,S'}$ on $S'$ for all $S'\in \mathcal{S}_h$,~a~basis~of $\mathcal{R}T^0(\mathcal{T}_h)$,  where~$n_S$ is the unit normal vector on $S$ pointing from $T_-$ to $T_+$~if~${T_+,T_-\in \mathcal{T}_h}$~with~$S=\partial T_+\cap \partial T_-$. Then,  
	the (Fortin) quasi-interpolation operator $\Pi_h^{rt}\colon V_{p,q}(\Omega)\coloneqq \{{y\in (L^p(\Omega))^d\mid  \textup{div}\,y\in}$ $ L^q(\Omega)\}\to \smash{\mathcal{R}T^{0}(\mathcal{T}_h)}$ (\textit{cf}.\ \cite[Sec.\ 16.1]{EG21I}), where $p>2$ and $q>\frac{2d}{d+2}$, for every $y\in V_{p,q}(\Omega)$~defined~by
	\begin{align}
		\Pi_h^{rt} y\coloneqq \sum_{S\in \mathcal{S}_h}{\langle y\cdot n_S\rangle_S\,\psi_S}\,,\label{RT-interpolant}
	\end{align}
	preserves averages of divergences and normal traces (on sides), \textit{i.e.}, for every $y\in V_{p,q}(\Omega)$, it holds that\vspace{-1mm}
	\begin{alignat}{2}
		\textup{div}\,\Pi_h^{rt}y&=\Pi_h\textup{div}\,y&&\quad \text{ a.e.\ in }\Omega\,,\label{eq:div_preservation}\\
		\Pi_h^{rt}y\cdot n&=\pi_hy\cdot n&&\quad \text{ a.e.\ on }\cup\mathcal{S}_h\,.\label{eq:normal_trace_preservation}
	\end{alignat}

	\subsubsection{Discrete integration-by-parts formula}
	
	\hspace{5mm}For every $v_h\hspace{-0.15em}\in \hspace{-0.15em}\mathcal{S}^{1,cr}(\mathcal{T}_h)$ and $y_h\hspace{-0.15em}\in\hspace{-0.15em} \mathcal{R}T^0(\mathcal{T}_h)$,~there~holds the \textit{discrete integration-by-parts~\mbox{formula}}
	\begin{align}
		(\nabla_hv_h,\Pi_h y_h)_{\Omega}+(\Pi_h v_h,\,\textup{div}\,y_h)_{\Omega}=(\pi_h v_h,y_h\cdot n)_{\partial\Omega}\,.\label{eq:pi}
	\end{align} 
	
    \section{Gradient constrained variational problem}\label{sec:continuous}

	\hspace{5mm}In this section, we discuss a gradient constrained variational problem.\medskip
	
	\hspace*{-2.5mm}$\bullet$ \textit{Primal problem.} Given $f\in L^1(\Omega)$, $g\in W^{-1,1}(\Gamma_N)$,  $\zeta\in L^\infty(\Omega)$ such that ${\textup{ess\,inf}_{x\in \Omega}{\zeta(x)}>0}$, and $u_D\in W^{1,\infty}(\Gamma_D)$ such that there exists a trace lift $\widehat{u}_D\in W^{1,\infty}(\Omega)$ satisfying ${\|\frac{\nabla\widehat{u}_D}{\zeta}\|_{\infty,\Omega}<1}$,
	the \textit{primal problem} is given via the minimization of $I\colon \hspace{-0.175em}W^{1,\infty}(\Omega)\hspace{-0.175em}\to\hspace{-0.175em} \mathbb{R}\cup\{+\infty\}$,~for~every~${v\hspace{-0.175em}\in \hspace{-0.175em}W^{1,\infty}(\Omega)}$ defined by 
	\begin{align} \label{eq:primal}
		\begin{aligned} 
			I(v)&\coloneqq  \tfrac{1}{2}\| \nabla v\|_{2,\Omega}^2+I_K(v)-(f,v)_{\Omega}-\langle g,v\rangle_{\Gamma_N}\\
			&\coloneqq \tfrac{1}{2}\| \nabla v\|_{2,\Omega}^2+\smash{I_{K_\zeta(0)}^{\Omega}}(\nabla v)-(f,v)_{\Omega}-\langle g,v\rangle_{\Gamma_N}+I_{\{u_D\}}^{\Gamma_D}(v)\,, 
		\end{aligned}
	\end{align} 
	where\enlargethispage{2.5mm}
	\begin{align*}
		K\coloneqq \big\{v\in W^{1,\infty}(\Omega)\mid \vert \nabla v\vert\leq \zeta \text{ a.e.\ in }\Omega\,,\; v=u_D\text{ a.e.\ on }\Gamma_D\big\}\,,
	\end{align*}
	$I_K\hspace*{-0.1em}\coloneqq \hspace*{-0.1em}I_{K_\zeta(0)}^{\Omega}\circ\nabla\hspace*{-0.1em}\colon\hspace*{-0.1em} W^{1,\infty}(\Omega)\hspace*{-0.175em}\to\hspace*{-0.175em} \mathbb{R}\cup\{+\infty\}$,  $\smash{I_{K_\zeta(0)}^{\Omega}}\hspace*{-0.1em}\colon\hspace*{-0.1em}(L^{\infty}(\Omega))^d\hspace*{-0.175em}\to\hspace*{-0.175em}\mathbb{R}\cup\{+\infty\}$,~for~every~$\widehat{y}\hspace*{-0.175em}\in\hspace*{-0.175em}(L^{\infty}(\Omega))^d$, is defined by 
	\begin{align*}
		\smash{I_{K_\zeta(0)}^{\Omega}}(\widehat{y})
		&\coloneqq 
		\begin{cases}
			0&\text{ if }\vert \widehat{y}\vert \leq \zeta\text{ a.e.\ in }\Omega\,,\\
			+\infty &\text{ else}\,,
		\end{cases}
	\end{align*}
	and $\smash{I_{\{u_D\}}^{\Gamma_D}}\colon W^{1,\infty}(\Gamma_D)\to \mathbb{R}\cup\{+\infty\}$, for every $\widehat{v}\in W^{1,\infty}(\Gamma_D)$, is defined by 
	\begin{align*}
		\smash{I_{\{u_D\}}^{\Gamma_D}}(\widehat{v})
		&\coloneqq 
		\begin{cases}
			0&\text{ if }\widehat{v}=u_D\text{ a.e.\ on }\Gamma_D\,,\\
			+\infty &\text{ else}\,.
		\end{cases}
	\end{align*}
	Since the functional \eqref{eq:primal} is proper, strictly convex, weakly coercive, and lower semi-continuous, the direct method in the calculus of variations yields the existence of a unique~minimizer~${u\in K}$, 
	called \textit{primal solution}. We reserve the notation $u\in K$ for the primal solution.\medskip\enlargethispage{7mm}

	\hspace*{-2.5mm}$\bullet$ \textit{Dual problem.} A \textit{(Fenchel) dual problem} (in the sense of \cite[Rem.\ 4.2, p.\ 60/61]{ET99}) to the minimization of \eqref{eq:primal} is given via~the~maximization of $D\colon (\textup{ba}(\Omega))^d\to \mathbb{R}\cup\{-\infty\}$, for every $\nu =y\otimes\mathrm{d}x+\nu^s\in (\textup{ba}(\Omega))^d$, where $y\in (L^1(\Omega))^d$ and $\nu^s\in (\textup{ba}(\Omega))^d$ with $\nu^s\perp y\otimes \mathrm{d}x$ (\textit{cf}.\ \eqref{eq:lebesgue_decomposition}), defined by\footnote{Here, 
 $\textcolor{black}{\zeta\nu^s}\in (\mathrm{ba}(\Omega))^d$ is defined by $(\textcolor{black}{\zeta\nu^s})(A)\coloneqq \int_A{\textcolor{black}{\zeta}\,\mathrm{d}\nu^s}$ for all $A\in \mathcal{M}(\mathrm{d}x;\Omega)$.}
	\begin{align} \label{eq:dual}
		\begin{aligned} 
			D(\nu)&\coloneqq -\int_{\Omega}{\phi^*(\cdot,y)\,\mathrm{d}x}-\vert\textcolor{black}{\zeta \nu^s}\vert(\Omega)-I_{K^*}(\nu)\\
			&\quad+\langle \nu,\nabla \widehat{u}_D\rangle_{\smash{(L^{\infty}(\Omega))^d}}-(f,\widehat{u}_D)_{\Omega}-\langle g,\widehat{u}_D\rangle_{\Gamma_N}\,, 
		\end{aligned}
	\end{align} 
	where 
	\begin{align*}
		K^*\coloneqq \big\{\nu \in (\mathrm{ba}(\Omega))^d\mid \langle \nu , \nabla v\rangle_{(L^\infty(\Omega))^d}=(f,v)_\Omega+\langle g,v\rangle_{\Gamma_N}\text{ for all }v\in W_D^{1,\infty}(\Omega)\big\}\,,
	\end{align*}
	$I_{K^*}\colon  (\textup{ba}(\Omega))^d\to \mathbb{R}\cup \{+\infty\}$, for every $\widehat{\nu} \in  (\textup{ba}(\Omega))^d$, is defined by
	\begin{align*}
		I_{K^*}(\widehat{\nu}) \coloneqq
		\begin{cases}
			0&\text{ if }\widehat{\nu}\in K^*\,,\\
			+\infty& \text{ else}\,,
		\end{cases}
	\end{align*}
	and
	$\phi^*\colon \Omega\times \mathbb{R}^d\to\mathbb{R}$, for a.e.\ $x\in \Omega$ and every $s\in \mathbb{R}^d$, defined by
	\begin{align*}
		\phi^*(x,s)=\begin{cases}
			\tfrac{1}{2}\vert  s\vert ^2&\text{ if }\vert s\vert \leq \zeta(x)\,,\\
			\zeta(x)\vert  s\vert -\tfrac{\zeta(x)^2}{2}&\text{ if }\vert s\vert >\zeta(x)\,,
		\end{cases} 
	\end{align*}
	denotes the Fenchel conjugate to $\phi\colon \Omega\times\mathbb{R}^d\to \mathbb{R} \cup\{+\infty\}$ (with respect to the second argument), for a.e.\ $x\in \Omega$ and every $t\in\mathbb{R}^d$,  defined by\vspace*{-0.5mm} 
	\begin{align*}
		\phi(x,t)
  \coloneqq \tfrac{1}{2}\vert t\vert^2+\begin{cases}
			0&\text{ if }\vert t\vert \leq \zeta(x)\,,\\
			+\infty&\text{ if }\vert t\vert >\zeta(x)\,.
		\end{cases}
	\end{align*} 
	
	Appealing to \cite[Thm.\ 2]{Rockafellar68}, for every $y\in W^1(\textup{div};\Omega)$, we have that
	\begin{align*}
		D(y\otimes \mathrm{d}x)= -\int_{\Omega}{\phi^*(\cdot,y)\,\mathrm{d}x}-\smash{I_{\{-f\}}^{\Omega}}(\textup{div}\,y)-\smash{I_{\{g\}}^{\Gamma_N}}(y\cdot n) +\langle y\cdot n,u_D\rangle_{\partial\Omega}-\langle g,u_D\rangle_{\Gamma_N}\,,
	\end{align*}
	where 
	$	\smash{I_{\{-f\}}^{\Omega}}\colon  L^1(\Omega)\to \mathbb{R}\cup \{+\infty\}$, for every $\widehat{v}\in  L^1(\Omega)$, is defined by
	\begin{align*}
		\smash{I_{\{-f\}}^{\Omega}}(\widehat{v}) \coloneqq
		\begin{cases}
			0&\text{ if }\widehat{v}=-f\text{ a.e.\ in }\Omega\,,\\
			+\infty& \text{ else}\,,
		\end{cases}
	\end{align*}
	and $	\smash{I_{\{g\}}^{\Gamma_N}}\colon  W^{-1,1}(\partial\Omega)\to \mathbb{R}\cup \{+\infty\}$, for every $\widehat{v}\in W^{-1,1}(\partial\Omega)$, is defined by
	\begin{align*}
		\smash{I_{\{g\}}^{\Gamma_N}}(\widehat{v})&\coloneqq \begin{cases}
			0&\text{ if }\langle \widehat{v},v\rangle_{\partial\Omega}=\langle g, v\rangle_{\Gamma_N}\text{ for all }v\in W^{1,\infty}_D(\Omega)\,,\\
			+\infty &\text{ else}\,.
		\end{cases}
	\end{align*}
	
	The identification of the (Fenchel) dual problem (in the sense of \cite[Rem.\ 4.2, p.\ 60/61]{ET99}) to the minimization of \eqref{eq:primal} with the maximization of  \eqref{eq:dual} can be found in the proof of the following result that also establishes the validity of a strong~duality~relation and convex~optimality~relations.\enlargethispage{2.5mm}
	
	\begin{theorem}[strong \hspace*{-0.1mm}duality \hspace*{-0.1mm}and \hspace*{-0.1mm}convex \hspace*{-0.1mm}optimality \hspace*{-0.1mm}relations]\label{thm:duality} \hspace*{-0.1mm}The \hspace*{-0.1mm}following \hspace*{-0.1mm}statements~\hspace*{-0.1mm}\mbox{apply}:
		\begin{itemize}[noitemsep,topsep=2pt,leftmargin=!,labelwidth=\widthof{(ii)}]
			\item[(i)]  A (Fenchel) dual problem to  the minimization of \eqref{eq:primal} is given via the maximization~of~\eqref{eq:dual}.  
			\item[(ii)]  There exists a maximizer $\mu =z\otimes\mathrm{d}x+\mu^s\in (\textup{ba}(\Omega))^d$, where $z\in (L^1(\Omega))^d$ and $\mu^s\in (\textup{ba}(\Omega))^d$ with $\mu^s\perp z\otimes \mathrm{d}x$ (\textit{cf}.\ \eqref{eq:lebesgue_decomposition}), of \eqref{eq:dual} satisfying the \emph{admissibility condition}\vspace*{-0.5mm}
			\begin{align}\label{eq:admissibility}
				\langle \mu, \nabla v\rangle_{(L^\infty(\Omega))^d}=(f,v)_\Omega+\langle g,v\rangle_{\Gamma_N}\quad\text{ for all }v\in W^{1,\infty}_D(\Omega)\,. 
			\end{align}
			In addition, there
			holds a \emph{strong duality relation}, \textit{i.e.}, we have that
			\begin{align}
				I(u) = D(\mu)\,.\label{eq:strong_duality}
			\end{align}
			\item[(iii)] There hold  \emph{convex optimality relations}, \textit{i.e.}, we have that\vspace*{-1mm}
			\begin{align}
				\nabla u&=D_t\phi^*(\cdot,z)=\left.\begin{cases}
					z&\text{ if }\vert z\vert\leq \zeta\,,\\
					\zeta \frac{z}{\vert z\vert}&\text{ if }\vert z\vert> \zeta
				\end{cases}\right\}\quad\text{ a.e.\ in }\Omega\,,\label{eq:optimality.1}\\
				\vert \textcolor{black}{\zeta \mu^s}\vert(\Omega)&=\langle \mu^s,\nabla u\rangle_{(L^\infty(\Omega))^d}\,.\label{eq:optimality.2}
			\end{align}
		\end{itemize}
	\end{theorem}
	
	\begin{remark}
		\begin{itemize}[noitemsep,topsep=2pt,leftmargin=!,labelwidth=\widthof{(ii)}]
			\item[(i)] By the standard equality condition in the Fenchel--Young inequality (\textit{cf}.\ \cite[Prop.\ 5.1, p.\ 21]{ET99}), the convex optimality relation \eqref{eq:optimality.1} is equivalent to\vspace*{-0.5mm}
			\begin{align}\label{eq:optimality_regular}
				z\cdot \nabla u=\phi^*(\cdot,z)+\phi(\cdot,\nabla u)\quad \text{ a.e.\ in }\Omega\,.
			\end{align}
			In addition, by \cite[Cor.\ 5.2(5.8), p.\ 22]{ET99},  the convex optimality relation \eqref{eq:optimality.1} is equivalent to 
			\begin{align}\label{eq:optimality_regular2}
				z\in \partial_t \phi(\cdot,\nabla u)\quad \text{ a.e.\ in }\Omega\,.
			\end{align} 
			\item[(ii)] 	If $\mu=z\otimes \mathrm{d}x\in (\textup{ba}(\Omega))^d$, \textit{i.e.}, $\mu^s=0$ in the Lebesgue~decomposition~\eqref{eq:lebesgue_decomposition}, then from the admissibility condition \eqref{eq:admissibility}, it follows that $z\in W^1(\textup{div};\Omega)$ with\vspace*{-1mm}
			\begin{alignat}{2}
				\textup{div}\,z&=-f&&\quad \text{ a.e.\ in }\Omega\,,\label{eq:admissibility_regular}\\
				\langle z\cdot n,v\rangle_{\partial\Omega}&=\langle g,v\rangle_{\Gamma_N}&&\quad \text{ for all }v\in W^{1,\infty}_D(\Omega)\,.\label{eq:admissibility_regular.2}
			\end{alignat}
		\end{itemize} 
	\end{remark}
	
	\begin{proof}[Proof (of Theorem \ref{thm:duality}).]
		\textit{ad (i).} 
        First, we introduce the functionals $G\colon (L^\infty(\Omega))^d\to \mathbb{R}\cup\{+\infty\}$ and $F\colon W^{1,\infty}(\Omega)\to \mathbb{R}\cup\{+\infty\}$, for every $y\in  (L^\infty(\Omega))^d$ and $v\in W^{1,\infty}(\Omega)$, respectively,~defined~by\enlargethispage{5mm}
		\begin{align*}
			G(y)&\coloneqq \int_{\Omega}{\phi(\cdot,\textcolor{black}{y})\,\mathrm{d}x}\\&=\tfrac{1}{2}\|y\|_{2,\Omega}^2+I_{K_\zeta(0)}^{\Omega}(y)\,,\\
			F(v)&\coloneqq -(f,v)_{\Omega}-\langle g,v\rangle_{\Gamma_N}+I_{\{u_D\}}^{\Gamma_D}(v)\,.
		\end{align*}
		Then, for every $v\in W^{1,\infty}(\Omega)$, we have that 
		\begin{align*}
			I(v)= G(\nabla v)+F(v)\,.
		\end{align*}\newpage
		\hspace*{-5mm}Thus, \hspace{-0.15mm}in \hspace{-0.15mm}accordance \hspace{-0.15mm}with \hspace{-0.15mm}\cite[Rem.\ \hspace{-0.15mm}4.2, \hspace{-0.15mm}p.\ \hspace{-0.15mm}60/61]{ET99}, \hspace{-0.15mm}the \hspace{-0.15mm}(Fenchel) \hspace{-0.1mm}dual \hspace{-0.1mm}problem \hspace{-0.15mm}to \hspace{-0.15mm}the~\hspace{-0.15mm}\mbox{minimization}~\hspace{-0.15mm}of \eqref{eq:primal} \hspace{-0.1mm}is \hspace{-0.1mm}given \hspace{-0.1mm}via \hspace{-0.1mm}the \hspace{-0.1mm}maximization \hspace{-0.1mm}of \hspace{-0.1mm}$D\colon\hspace*{-0.1em} (\mathrm{ba}(\Omega))^d\hspace*{-0.1em}\to\hspace*{-0.1em} \mathbb{R}\hspace*{-0.1em}\cup\hspace*{-0.1em}\{-\infty\}$, 
		\hspace{-0.1mm}for \hspace{-0.1mm}every \hspace{-0.1mm}$\nu\hspace*{-0.1em}\in\hspace*{-0.1em} (\mathrm{ba}(\Omega))^d$~\hspace{-0.1mm}defined~\hspace{-0.1mm}by 
		\begin{align}\label{thm:duality.1}
			D(\nu)\coloneqq -G^*(\nu)- F^*(-\nabla^*\nu)\,,
		\end{align}
		where $\nabla^*\colon  (\mathrm{ba}(\Omega))^d\to (W^{1,\infty}(\Omega))^*$ denotes the adjoint operator to  $\nabla \colon W^{1,\infty}(\Omega)\to (L^\infty(\Omega))^d$.\enlargethispage{5mm}
		
		First, resorting to \cite[Thm.\ 1]{Rockafellar71}, for every $\nu =y\otimes \mathrm{d}x+\nu^s\in (\mathrm{ba}(\Omega))^d$, where $y\in (L^1(\Omega))^d$ and $\nu^s\in (\mathrm{ba}(\Omega))^d$ with $\nu^s\perp y\otimes\mathrm{d}x$ (\textit{cf}.\ \eqref{eq:lebesgue_decomposition}), we find that
		\begin{align}\label{thm:duality.2}
			G^*(\nu)= \int_{\Omega}{\phi^*(\cdot,y)\,\mathrm{d}x}+(I_C)^*(\nu^s)\,,
		\end{align}
		where\vspace{-0.5mm}
		\begin{align*}
			C\coloneqq \{y\in (L^{\infty}(\Omega))^d\mid G(y)<+\infty\}=\{y\in (L^{\infty}(\Omega))^d\mid \vert y\vert\leq \zeta\text{ a.e.\ in }\Omega\}\,,
		\end{align*}
		$I_C\colon (L^{\infty}(\Omega))^d\to \mathbb{R}\cup\{+\infty\}$, for every $\hat{y}\in (L^{\infty}(\Omega))^d$, is defined by\vspace{-0.5mm}
		\begin{align*}
			I_C(\hat{y})\coloneqq \begin{cases}
				0&\text{ if }\hat{y}\in C\,,\\
				+\infty&\text{ else}\,,
			\end{cases}
		\end{align*}
		and, thus,\vspace{-0.5mm}
		\begin{align}\label{thm:duality.3}
			\begin{aligned} 
				(I_C)^*(\nu^s)=\sup_{y\in C}{\langle \nu^s,y\rangle_{(L^{\infty}(\Omega))^d}}=\vert\textcolor{black}{\zeta \nu^s}\vert(\Omega)\,.
			\end{aligned}
		\end{align}
		
		Second, for every $\nu\in (\mathrm{ba}(\Omega))^d$,  we have that 
		\begin{align}\label{thm:duality.4}
			\begin{aligned}
				F^*(-\nabla^*\nu)&
				=\sup_{v\in W^{1,\infty}(\Omega)}{\big\{-\langle\nabla^* \nu,v\rangle_{\smash{W^{1,\infty}(\Omega)}}+(f,v)_{\Omega}+\langle g,v\rangle_{\Gamma_N}-I_{\{u_D\}}^{\Gamma_D}(v)\big\}}
				\\&=\sup_{v\in W^{1,\infty}_D(\Omega)}{\big\{-\langle \nu,\nabla v\rangle_{\smash{(L^{\infty}(\Omega))^d}}+(f,v)_{\Omega}+\langle g,v\rangle_{\Gamma_N}\big\}}
				\\&\qquad-\langle \nu,\nabla \widehat{u}_D\rangle_{\smash{(L^{\infty}(\Omega))^d}}+(f,\widehat{u}_D)_{\Omega}+\langle g,\widehat{u}_D\rangle_{\Gamma_N}
				\\&
				=I_{K^*}(\nu)-\langle \nu,\nabla \widehat{u}_D\rangle_{\smash{(L^{\infty}(\Omega))^d}}+(f,\widehat{u}_D)_{\Omega}+\langle g,\widehat{u}_D\rangle_{\Gamma_N}\,.
			\end{aligned}
		\end{align}
		In summary,  using \eqref{thm:duality.2}, \eqref{thm:duality.3}, and \eqref{thm:duality.4} in \eqref{thm:duality.1}, for every $\nu\in (\mathrm{ba}(\Omega))^d$, we arrive at the claimed representation \eqref{eq:dual}. 
		
		\textit{ad (ii).} Since both $G\colon (L^\infty(\Omega))^d\to \mathbb{R}\cup\{+\infty\}$ and $F\colon W^{1,\infty}(\Omega)\to \mathbb{R}\cup\{+\infty\}$~are~proper,~convex, and lower semi-continuous and  since 
		$G\colon  (L^\infty(\Omega))^d\to \mathbb{R}$  is continuous at
		$\nabla\widehat{u}_D\in \textup{dom}(G\circ \nabla)$ (recall that  $\|\frac{\nabla\widehat{u}_D}{\zeta}\|_{\infty,\Omega}<1$) with $ \widehat{u}_D\in \textup{dom}(F)$, \textit{i.e.},  we have that\vspace{-0.5mm}
		\begin{align*}
			G(y)\to G(\nabla\widehat{u}_D)\quad(y\to \nabla\widehat{u}_D\quad\text{ in }(L^\infty(\Omega))^d)\,,
		\end{align*}
		by \hspace*{-0.1mm}the \hspace*{-0.1mm}celebrated \hspace*{-0.1mm}Fenchel \hspace*{-0.1mm}duality \hspace*{-0.1mm}theorem \hspace*{-0.1mm}(\emph{cf}.\ \hspace*{-0.1mm}\cite[\hspace*{-0.1mm}Rem.\ \hspace*{-0.1mm}4.2, \hspace*{-0.1mm}(4.21), \hspace*{-0.1mm}p.\ \hspace*{-0.1mm}61]{ET99}) and Lebesgue~decompo-sition theorem (\textit{cf}.\ \eqref{eq:lebesgue_decomposition}),  there exists a maximizer $\mu=z\otimes\mathrm{d}x+\mu^s\in (\mathrm{ba}(\Omega))^d$, where $z\in (L^1(\Omega))^d$ and $\mu^s\hspace*{-0.1em}\in\hspace*{-0.1em}  (\mathrm{ba}(\Omega))^d$ with $\mu^s\hspace*{-0.1em}\perp \hspace*{-0.1em} z\otimes\mathrm{d}x$, of  \eqref{thm:duality.1} and~a~strong~duality~\mbox{relation}~\mbox{applies},~\textit{i.e.},~it~holds~that
        \begin{align*}
            I(u)=D(\mu)\,.
        \end{align*}
		
		\textit{ad (iii).} \hspace*{-0.1mm}By \hspace*{-0.1mm}the \hspace*{-0.1mm}standard \hspace*{-0.1mm}(Fenchel) \hspace*{-0.1mm}convex \hspace*{-0.1mm}duality \hspace*{-0.1mm}theory \hspace*{-0.1mm}(\emph{cf}.\ \hspace*{-0.1mm}\cite[Rem.\ \hspace*{-0.1mm}4.2,~\hspace*{-0.1mm}(4.24),~\hspace*{-0.1mm}(4.25),\hspace*{-0.1mm}~p.~\hspace*{-0.1mm}61]{ET99}), there hold the convex optimality relations\vspace{-0.5mm}
		\begin{align}
			-\nabla^*\mu &\in \partial F(u)\,,\label{thm:duality.5}\\
			\mu&\in \partial G(\nabla u)\,.\label{thm:duality.6}
		\end{align} 
		While the inclusion \eqref{thm:duality.5} is equivalent to the admissibility condition \eqref{eq:admissibility}, the inclusion~\textcolor{black}{\eqref{thm:duality.6}}, by \cite[Cor.\ 1B]{Rockafellar71},  is equivalent to\vspace{-0.5mm}
		\begin{align}
			z&\in \partial_t \phi(\cdot,\nabla u)\quad \text{ a.e.\ in }\Omega\,,\label{thm:duality.7}\\
			\vert\textcolor{black}{\zeta \mu^s}\vert (\Omega)&=\langle   \mu^s,\nabla u\rangle_{(L^\infty(\Omega))^d}\,.\label{thm:duality.8}
		\end{align}
		Eventually, by \cite[Cor.\ 5.2(5.8), p.\ 22]{ET99}, \eqref{thm:duality.7} is equivalent to \eqref{eq:optimality.1}.
	\end{proof}\newpage

	\section{\emph{A posteriori} error analysis}\label{sec:aposteriori} 
	\qquad In this section, resorting to convex duality arguments from \cite{BarGudKal24}, we derive an \emph{a posteriori} error identity for conforming approximations of  the primal problem \eqref{eq:primal} and the~dual~problem~\eqref{eq:dual} at~the~same~time. 
	To this end, we introduce the 
	\emph{primal-dual gap estimator} ${\eta^2_{\textup{gap}}\colon K\times K^*\to \mathbb{R}}$, for every $v\in K$ and $\nu\in K^*$ defined by 
	\begin{align}\label{eq:primal-dual.1}
		\begin{aligned}
			\eta^2(v,\nu)&\coloneqq I(v)-D(\nu)\,. 
		\end{aligned}
	\end{align}
	
	The primal-dual gap estimator \eqref{eq:primal-dual.1} can be decomposed into two contributions that precisely measure the violation of the convex optimality relations \eqref{eq:optimality.1},\eqref{eq:optimality.2}.
	
	\begin{lemma}[decomposition of the primal-dual gap estimator]\label{lem:primal_dual_gap_estimator}
		For every $v\hspace*{-0.1em}\in\hspace*{-0.1em} K$ and $\nu\hspace*{-0.1em}=\hspace*{-0.1em}{y\otimes\mathrm{d}x}$ $+\nu^s\in K^*$, where $y\in (L^1(\Omega))^d$ and $\nu^s\in (\mathrm{ba}(\Omega))^d$ with $\nu^s\perp y\otimes\mathrm{d}x$ (\textit{cf}.\ \eqref{eq:lebesgue_decomposition}), we have that
		\begin{align*} 
			\eta_{\textup{gap}}^2(v,\nu)&=	\eta_{\textup{gap},A}^2(v,y)+\eta_{\textup{gap},B}^2(v,\nu^s)\,,\\
			\eta_{\textup{gap},A}^2(v,y)&\coloneqq	\int_{\Omega}{\{\phi(\cdot,\nabla v)-\nabla v\cdot y+\phi^*(\cdot,y)\}\,\mathrm{d}x}\,,\\
			\eta_{\textup{gap},B}^2(v,\nu^s)&\coloneqq \vert\textcolor{black}{\zeta \nu^s}\vert(\Omega)-\langle \nu^s,\nabla v\rangle_{(L^\infty(\Omega))^d}\,.
		\end{align*}
	\end{lemma}
	
	\begin{remark}[interpretation of the components of the primal-dual gap estimator]\hphantom{     }
		\begin{itemize}[noitemsep,topsep=2pt,leftmargin=!,labelwidth=\widthof{(ii)}]
			\item[(i)] The estimator $\eta_{\textup{gap},A}^2$ measures the violation of the convex optimality relation \eqref{eq:optimality.1};
			\item[(ii)] The estimator $\eta_{\textup{gap},B}^2$ measures the violation of the convex optimality relation \eqref{eq:optimality.2}.
		\end{itemize} 
	\end{remark}
	
	\begin{proof}[Proof (of Lemma \ref{lem:primal_dual_gap_estimator}).]\let\qed\relax
		Using the admissibility condition \eqref{eq:admissibility}, 
		for every $v\in K$ and $\nu={y\otimes\mathrm{d}x}+\nu^s\in K^*$, where $y\in (L^1(\Omega))^d$ and $\nu^s\in (\mathrm{ba}(\Omega))^d$ with $\nu^s\perp y\otimes\mathrm{d}x$ (\textit{cf}.\ \eqref{eq:lebesgue_decomposition}),  we find that
		\begin{align*}
			I(v)-D(\nu)&= \int_{\Omega}{\phi(\cdot,\nabla v)\,\mathrm{d}x}-(f,v)_{\Omega}-\langle g,v\rangle_{\Gamma_N}\\&\quad+\int_{\Omega}{\phi^*(\cdot,y)\,\mathrm{d}x}+\vert\textcolor{black}{\zeta \nu^s}\vert(\Omega)\\&\quad-\langle \nu,\nabla \widehat{u}_D\rangle_{\smash{(L^{\infty}(\Omega))^d}}+(f,\widehat{u}_D)_{\Omega}+\langle g,\widehat{u}_D\rangle_{\Gamma_N}
			\\&= \int_{\Omega}{\phi(\cdot,\nabla v)\,\mathrm{d}x}-(f,v-\widehat{u}_D)_{\Omega}-\langle g,v-\widehat{u}_D\rangle_{\Gamma_N}\\&\quad+\int_{\Omega}{\phi^*(\cdot,y)\,\mathrm{d}x}+\vert\textcolor{black}{\zeta \nu^s}\vert(\Omega)-\langle \nu,\nabla \widehat{u}_D\rangle_{\smash{(L^{\infty}(\Omega))^d}} 
			\\&= \int_{\Omega}{\phi(\cdot,\nabla v)\,\mathrm{d}x}-\langle \nu , \nabla(v-\widehat{u}_D)\rangle_{(L^\infty(\Omega))^d} \\&\quad+\int_{\Omega}{\phi^*(\cdot,y)\,\mathrm{d}x}+\vert\textcolor{black}{\zeta \nu^s}\vert(\Omega)-\langle \nu,\nabla \widehat{u}_D\rangle_{\smash{(L^{\infty}(\Omega))^d}} 
			\\&= \int_{\Omega}{\{\phi(\cdot,\nabla v)-\nabla v\cdot y+\phi^*(\cdot,y)\}\,\mathrm{d}x}+\vert\textcolor{black}{\zeta \nu^s}\vert(\Omega)-\langle \nu^s,\nabla v\rangle_{(L^\infty(\Omega))^d}\,.\tag*{$\qedsymbol$}
		\end{align*}
	\end{proof}
	In the next step, we  derive meaningful representations of the \emph{optimal strong convexity measures} for the primal energy functional~\eqref{eq:primal} at the primal solution $u\in K$, \textit{i.e.}, for 
	$\rho_I^2\colon K\to [0,+\infty)$, for every $v\in K$ defined by
	\begin{align}\label{def:optimal_primal_error}
		\rho_I^2(v)\coloneqq I(v)-I(u)\,,
	\end{align}
	and \hspace{-0.1mm}for \hspace{-0.1mm}the \hspace{-0.1mm}negative \hspace{-0.1mm}of \hspace{-0.1mm}the \hspace{-0.1mm}dual \hspace{-0.1mm}energy \hspace{-0.1mm}functional \eqref{eq:dual}, \hspace{-0.1mm}\textit{i.e.}, \hspace{-0.1mm}for \hspace{-0.1mm}$\rho_{-D}^2\colon\hspace*{-0.175em} K^*\hspace*{-0.175em}\to\hspace*{-0.175em} [0,+\infty)$,~\hspace{-0.1mm}for~\hspace{-0.1mm}\mbox{every}~\hspace{-0.1mm}${\nu\hspace*{-0.175em}\in\hspace*{-0.175em} K^*}$ defined by
	\begin{align}\label{def:optimal_dual_error}
		\rho_{-D}^2(\nu)\coloneqq- D(\nu)+D(\mu)\,,
	\end{align}
	which will serve as \emph{`natural'} error quantities in the primal-dual gap identity (\textit{cf}.\ Theorem \ref{thm:prager_synge_identity}). 
	
	\begin{lemma}[representations of the optimal strong convexity measures]\label{lem:strong_convexity_measures}
		The following statements apply:
		\begin{itemize}[noitemsep,topsep=2pt,leftmargin=!,labelwidth=\widthof{(ii)}]
			\item[(i)] For every $v\in K$, we have that
			\begin{align*}
				\rho_I^2(v)=\tfrac{1}{2}\|\nabla v-\nabla u\|_{2,\Omega}^2+\big(\tfrac{\vert z\vert}{\zeta}-1,\zeta^2-\nabla u\cdot\nabla v \big)_{\smash{\{\vert \nabla u\vert=\zeta\}}}+\textcolor{black}{\vert\textcolor{black}{\zeta \mu^s}\vert(\Omega)-\langle \mu^s,\nabla v\rangle_{(L^\infty(\Omega))^d}}\,.
			\end{align*}
			\item[(ii)] For every $\nu=y\otimes \mathrm{d}x+\nu^s\in K^*$, where $y\in (L^1(\Omega))^d$ and $\nu^s\in (\mathrm{ba}(\Omega))^d$ with $\nu^s\perp y\otimes \mathrm{d}x$ (\textit{cf}.\ \eqref{eq:lebesgue_decomposition}), we have that
			\begin{align*}
				\rho_{-D}^2(\nu)=(D_t^2\phi^*\{y,z\}(y-z),y-z)_{\Omega}+\vert\textcolor{black}{\zeta \nu^s}\vert(\Omega)-\langle \nu^s,\nabla u\rangle_{(L^\infty(\Omega))^d}\,,
			\end{align*}
			where $D_t^2\phi^*\{\cdot,\cdot\}(\cdot) : \mathbb{R}^d\times \mathbb{R}^d\times \Omega\to \mathbb{R}^{d\times d}$, for a.e.\ $x\in \Omega$ and every $t,s\in \mathbb{R}^d$, is defined by
			\begin{align*}
				D_t^2\phi^*\{t,s\}(x)\coloneqq \int_{0}^{1}{\textcolor{black}{(1-\lambda)}D_t^2\phi^*(x,\lambda t+(1-\lambda)s)\,\mathrm{d}\lambda}\,.
			\end{align*}
		\end{itemize}
	\end{lemma}
	
	\begin{remark}
		\begin{itemize}[noitemsep,topsep=2pt,leftmargin=!,labelwidth=\widthof{(ii)}]
			\item[(i)] Since $\{\vert z\vert\ge \zeta\}=\{\vert \nabla u\vert=\zeta\}$ (\textit{cf}.\ \eqref{eq:optimality.1}) and $\nabla u\cdot\nabla v\leq \zeta^2$ a.e.\ in~$\{\vert \nabla u\vert=\zeta\}$ for all $v\in K$, for every $v\in K$, we have that
			\begin{align*}
				\big(\tfrac{\vert z\vert}{\zeta}-1\big)(\zeta^2-\nabla u\cdot\nabla v)\ge 0\quad\text{ a.e.\ in }\{\vert \nabla u\vert=\zeta\}\,.
			\end{align*}
			
			\item[(ii)] Since for a.e.\ $x\in \Omega$ and every $t\in \mathbb{R}^d$,  we have that
			\begin{align}\label{eq:D2_phi_prime}
				D_t^2\phi^*(x,t) =\begin{cases}
					\mathrm{I}_{d\times d}&\text{ if }\vert t\vert \leq \zeta(x)\,,\\
					\frac{\zeta(x)}{\vert t\vert }(\mathrm{I}_{d\times d}-\frac{t\otimes t}{\vert t\vert^2})&\text{ if }\vert t\vert>\zeta(x)\,,
				\end{cases}
			\end{align} 
			for a.e.\ $x\in \Omega$ and every $t,s\in \mathbb{R}^d$, by the Cauchy--Schwarz inequality, we observe that
			\begin{align*}
				D_t^2\phi^*(x,t):s\otimes s&=\begin{cases}
					\vert s\vert^2&\text{ if }\vert t\vert \leq \zeta(x)\,,\\
					\frac{\zeta(x)}{\vert t\vert }(\vert s\vert^2-\frac{(t\cdot s)^2}{\vert t\vert^2})&\text{ if }\vert t\vert>\zeta(x)\,,
				\end{cases}\\&\ge \begin{cases}
					\vert s\vert^2&\text{ if }\vert t\vert \leq \zeta(x)\,,\\
					0&\text{ if }\vert t\vert>\zeta(x)\,.
				\end{cases}
			\end{align*}
		\end{itemize}
	\end{remark}
	
	\begin{proof}[Proof \hspace*{-0.15mm}(of \hspace*{-0.15mm}Lemma \hspace*{-0.1mm}\ref{lem:strong_convexity_measures})]
		
		\emph{ad (i).} Using the admissibility condition \eqref{eq:admissibility} and  the convex optimality relation \eqref{eq:optimality.2},   for every $v\in K$, we find that
		\begin{align}\label{lem:strong_convexity_measures.0.1}
            \begin{aligned} 
			I(v)-I(u)&=\tfrac{1}{2}\|\nabla v\|_{2,\Omega}^2-\tfrac{1}{2}\|\nabla u\|_{2,\Omega}^2-(f,v-u)_{\Omega}-\langle g,v-u\rangle_{\Gamma_N}
            \\&=\tfrac{1}{2}\|\nabla v\|_{2,\Omega}^2-\tfrac{1}{2}\|\nabla u\|_{2,\Omega}^2-\langle \mu,\nabla v-\nabla u\rangle_{(L^\infty(\Omega))^d}
			\\&=\tfrac{1}{2}\|\nabla v\|_{2,\Omega}^2-\tfrac{1}{2}\|\nabla u\|_{2,\Omega}^2-(z,\nabla v-\nabla u)_{\Omega}-\langle \mu^s,\nabla v-\nabla u\rangle_{(L^\infty(\Omega))^d}
			\\&=\tfrac{1}{2}\|\nabla v\|_{2,\Omega}^2-\tfrac{1}{2}\|\nabla u\|_{2,\Omega}^2-(z,\nabla v-\nabla u)_{\Omega}+\vert\textcolor{black}{\zeta \mu^s}\vert(\Omega)-\langle \mu^s,\nabla v\rangle_{(L^\infty(\Omega))^d}\,.
            \end{aligned}
		\end{align}
		Next, due to the convex optimality relation \eqref{eq:optimality.1}, we have that
		\begin{align}\label{lem:strong_convexity_measures.0.2}
			\begin{aligned}
				z&=\nabla u&&\quad \text{ a.e.\ in }\{\vert \nabla u\vert<\zeta\}\,,\\
				z&=\tfrac{\vert z\vert}{\zeta}\,\nabla u&&\quad \text{ a.e.\ in }\{\vert \nabla u\vert=\zeta\}\,.
			\end{aligned}
		\end{align} 
		Therefore, using \eqref{lem:strong_convexity_measures.0.2} and the binomial formula in \eqref{lem:strong_convexity_measures.0.1}, we find that\enlargethispage{5mm}
		\begin{align*}
			I(v)-I(u)&=\tfrac{1}{2}\|\nabla v\|_{2,\Omega}^2-\tfrac{1}{2}\|\nabla u\|_{2,\Omega}^2-(\nabla u,\nabla v-\nabla u)_{\Omega}\\&\quad+
			\big(\big(\tfrac{\vert z\vert}{\zeta}-1\big)\nabla u,\nabla u-\nabla v\big)_{\smash{\{\vert \nabla u\vert=\zeta\}}}\\&\quad+\vert\textcolor{black}{\zeta \mu^s}\vert(\Omega)-\langle \mu^s,\nabla v\rangle_{(L^\infty(\Omega))^d}
			\\&=\tfrac{1}{2}\|\nabla v-\nabla u\|_{2,\Omega}^2
			+\big(\tfrac{\vert z\vert}{\zeta}-1,\zeta^2-\nabla u\cdot\nabla v\big)_{\smash{\{\vert \nabla u\vert=\zeta\}}}\\&\quad+\vert\textcolor{black}{\zeta \mu^s}\vert(\Omega)-\langle \mu^s,\nabla v\rangle_{(L^\infty(\Omega))^d}\,,
		\end{align*}
    which is the claimed representation of the optimal strong convexity measure $\rho_I^2\colon K\to [0,+\infty)$.
		
		\emph{ad (ii).} Using the admissibility condition \eqref{eq:admissibility} and Taylor expansion, for every $\nu=y\otimes \mathrm{d}x+\nu^s\in K^*$, where $y\in (L^1(\Omega))^d$ and $\nu^s\in (\mathrm{ba}(\Omega))^d$ with $\nu^s\perp y\otimes \mathrm{d}x$ (\textit{cf}.\ \eqref{eq:lebesgue_decomposition}), we find that\vspace{-0.5mm} 
		\begin{align}\label{lem:strong_convexity_measures.1}
			\begin{aligned} 
			-D(\nu)+D(\mu)&=\int_{\Omega}{\phi^*(\cdot,y)\,\mathrm{d}x}  -\int_{\Omega}{\phi^*(\cdot,z)\,\mathrm{d}x}\\&\quad+\vert\textcolor{black}{\zeta \nu^s}\vert(\Omega)-\vert\textcolor{black}{\zeta \mu^s}\vert(\Omega)-\langle\nu- \mu,\nabla \widehat{u}_D\rangle_{\smash{(L^{\infty}(\Omega))^d}}
			\\&=(D_t\phi^*(\cdot,z),y-z)_{\Omega}+(D_t^2\phi^*\{y,z\}(y-z),y-z)_{\Omega}\\&\quad
			+\vert\textcolor{black}{\zeta \nu^s}\vert(\Omega)-\vert\textcolor{black}{\zeta \mu^s}\vert(\Omega)-\langle\nu- \mu,\nabla \widehat{u}_D\rangle_{\smash{(L^{\infty}(\Omega))^d}}\,.
			\end{aligned}
		\end{align}
		Next, using the convex optimality relations \eqref{eq:optimality.1},\eqref{eq:optimality.2},  for every $\nu=y\otimes \mathrm{d}x+\nu^s\in K^*$, where $y\in (L^1(\Omega))^d$ and $\nu^s\in (\mathrm{ba}(\Omega))^d$ with $\nu^s\perp y\otimes \mathrm{d}x$ (\textit{cf}.\ \eqref{eq:lebesgue_decomposition}), we observe that
			\begin{align}\label{lem:strong_convexity_measures.2}
					\begin{aligned} 
		(D_t\phi^*(\cdot,z),y-z)_{\Omega}-\langle\nu- \mu,\nabla \widehat{u}_D\rangle_{\smash{(L^{\infty}(\Omega))^d}}
		&=	(\nabla u,y-z)_{\Omega}-(y-z,\nabla \widehat{u}_D)_{\Omega}\\&\quad-\langle\nu^s- \mu^s,\nabla \widehat{u}_D\rangle_{\smash{(L^{\infty}(\Omega))^d}}
		\\&=(y-z,\nabla u-\nabla \widehat{u}_D)_{\Omega}\\&\quad+\langle\nu^s- \mu^s,\nabla u-\nabla \widehat{u}_D\rangle_{\smash{(L^{\infty}(\Omega))^d}}
		\\&\quad -\langle\nu^s- \mu^s,\nabla u\rangle_{\smash{(L^{\infty}(\Omega))^d}}
		\\&=\langle\nu- \mu,\nabla u-\nabla \widehat{u}_D\rangle_{\smash{(L^{\infty}(\Omega))^d}}
		\\&\quad -\langle\nu^s,\nabla u\rangle_{\smash{(L^{\infty}(\Omega))^d}}+\vert \textcolor{black}{\zeta \mu^s}\vert(\Omega)
		\\&=-\langle\nu^s,\nabla u\rangle_{\smash{(L^{\infty}(\Omega))^d}}+\vert \textcolor{black}{\zeta \mu^s}\vert(\Omega)\,,
			\end{aligned}
		\end{align}
		where we used in the last step that, due to the admissibility condition \eqref{eq:admissibility}, for every $\widehat{v}\in W^{1,\infty}_D(\Omega)$, we have that\vspace{-0.5mm}
		\begin{align*}
			\langle\nu,\nabla \widehat{v}\rangle_{\smash{(L^{\infty}(\Omega))^d}}=\langle\mu,\nabla \widehat{v}\rangle_{\smash{(L^{\infty}(\Omega))^d}}\,.
		\end{align*}
		Eventually, using \eqref{lem:strong_convexity_measures.2} in \eqref{lem:strong_convexity_measures.1}, we conclude that the claimed representation of the optimal strong convexity measure $\rho_{-D}^2\colon K^*\to [0,+\infty)$ applies.
	\end{proof}
	
	Towards this end, we have everything at our disposal to establish an \emph{a posteriori} error identity that identifies the \emph{primal-dual total error} $\rho_{\textup{tot}}^2\colon K\times K^*\to [0,+\infty)$, for every $v\in K$ and $\nu\in K^*$ defined by\vspace{-0.5mm}
	\begin{align}\label{def:primal_dual_total_error}
		\rho_{\textup{tot}}^2(v,\nu)\coloneqq \rho_I^2(v)+\rho_{-D}^2(\nu)\,,
	\end{align}
	with the primal-dual gap estimator $\eta_{\textup{gap}}^2\colon K\times K^*\to [0,+\infty)$ (\textit{cf}.\ \eqref{eq:primal-dual.1}). This leads to an \textit{a posteriori} error identity, called \textit{primal-dual gap identity}, which can  be interpreted as a generalization of the celebrated Prager--Synge identity (\textit{cf}.\ \cite{PraSyn47}).\enlargethispage{10mm}

	\begin{theorem}[primal-dual gap identity]\label{thm:prager_synge_identity}
		For every $v\in K$ and $\nu\in K^*$, we have that
		\begin{align*}
			\rho_{\textup{tot}}^2(v,\nu)
			=\eta_{\textup{gap}}^2(v,\nu)\,.
		\end{align*}
	\end{theorem}
	
	\begin{proof}\let\qed\relax
		Combining the definitions \eqref{eq:primal-dual.1}, \eqref{def:optimal_primal_error}, \eqref{def:optimal_dual_error}, \eqref{def:primal_dual_total_error}, and the strong duality relation \eqref{eq:strong_duality}, 
		for every $v\in K$ and $\nu\in K^*$, we find that
		\begin{align*} \rho_{\textup{tot}}^2(v,\nu)&=\rho_I^2(v)+\rho_{-D}^2(\nu)
			\\&=I(v)-I(u)+D(\mu )-D(\nu)
			\\&=I(v)-D(\nu)
			\\&=\eta_{\textup{gap}}^2(v,\nu)\,.\tag*{$\qedsymbol$}
		\end{align*}
	\end{proof} 

    The primal-dual gap identity (\textit{cf}.\ Theorem \ref{thm:prager_synge_identity}) requires approximating the primal and the dual problem at the same time in a conforming way --a potentially computationally expensive task. On the basis of orthogonality relations between the Crouzeix--Raviart  and the Raviart--Thomas element, in the subsequent section, 
	we transfer all convex duality relations from Section \ref{sec:continuous} to a discrete level to arrive at a discrete reconstruction formula that allows us to approximate the primal and the dual problem at the same time using only the Crouzeix--Raviart element.
	
	\newpage
	 
    \section{Discrete gradient constrained variational problem}\label{sec:discrete}
	
	\hspace*{5mm}In this section, we discuss a discrete gradient constrained variational problem.\medskip\enlargethispage{7.5mm}
	
	\hspace*{-2.5mm}$\bullet$ \textit{Discrete primal problem.} Let   $f_h\hspace*{-0.1em}\in\hspace*{-0.1em} \mathcal{L}^0(\mathcal{T}_h)$, $g_h\hspace*{-0.1em}\in\hspace*{-0.1em} \mathcal{L}^0(\mathcal{S}_h^{\Gamma_N})$, ${\zeta_h\hspace*{-0.1em}\in \hspace*{-0.1em}\mathcal{L}^0(\mathcal{T}_h)}$ with ${\textup{ess\,inf}_{x\in \Omega}{\zeta_h(x)}\hspace*{-0.1em}>\hspace*{-0.1em}0}$, and $u_D^h\in \mathcal{L}^0(\mathcal{S}_h^{\Gamma_D})$ such that there exists a trace lift $\widehat{u}_D\in \mathcal{S}^{1,cr}(\mathcal{T}_h)$ satisfying $\|\frac{\nabla\widehat{u}_D^h}{\zeta_h}\|_{\infty,\Omega}<1$ be approximations of $f\in L^1(\Omega)$, $g\in W^{-1,1}(\Gamma_N)$, $\zeta\in L^\infty(\Omega)$, and  $u_D\in W^{1,\infty}(\Gamma_D)$,~\mbox{respectively}. Then, the \emph{discrete primal problem} is given via the minimization of $I_h^{cr}\colon \mathcal{S}^{1,cr}(\mathcal{T}_h)\to \mathbb{R}\cup\{+\infty\}$, for every $v_h\in \mathcal{S}^{1,cr}(\mathcal{T}_h)$ defined by
	\begin{align}
		\begin{aligned}
			I_h^{cr}(v_h)&\coloneqq \tfrac{1}{2}\|\nabla_h v_h\|_{2,\Omega}^2+I_{K_h^{cr}}(v_h)  -(f_h,\Pi_hv_h)_{\Omega}-(g_h,\pi_hv_h)_{\Gamma_N}\\
			&\coloneqq\tfrac{1}{2}\|\nabla_h v_h\|_{2,\Omega}^2+I_{K_{\zeta_h}(0)}^\Omega(\nabla_h v_h)  -(f_h,\Pi_hv_h)_{\Omega}-(g_h,\pi_hv_h)_{\Gamma_N}+I_{\{u_D^h\}}^{\Gamma_D}(v_h)\,,
		\end{aligned}\label{eq:discrete_primal}
	\end{align}
	where
	\begin{align*}
		K_h^{cr}\coloneqq \big\{v_h\in \mathcal{S}^{1,cr}(\mathcal{T}_h)\mid \vert \nabla_h v_h\vert\leq \zeta_h \text{ a.e.\ in }\Omega\,,\; \pi_hv_h=u_D^h\text{ a.e.\ on }\Gamma_D\big\}\,,
	\end{align*}
	$I_{K_h^{cr}}\hspace*{-0.15em}\coloneqq \hspace*{-0.15em} I_{K_{\zeta_h}(0)}^{\Omega}\circ\nabla_h\colon \hspace*{-0.15em}\mathcal{S}^{1,cr}(\mathcal{T}_h)\hspace*{-0.15em}\to \hspace*{-0.15em}\mathbb{R}\cup\{+\infty\}$,  $\smash{I_{K_{\zeta_h}(0)}^{\Omega}}\hspace*{-0.1em}\colon\hspace*{-0.1em}(\mathcal{L}^0(\mathcal{T}_h))^d\hspace*{-0.175em}\to\hspace*{-0.175em}\mathbb{R}\cup\{+\infty\}$,~for~every~${\widehat{y}_h\hspace*{-0.175em}\in\hspace*{-0.175em}(\mathcal{L}^0(\mathcal{T}_h))^d}$, is defined by 
	\begin{align*}
		\smash{I_{K_{\zeta_h}(0)}^{\Omega}}(\widehat{y}_h)
		&\coloneqq 
		\begin{cases}
			0&\text{ if }\vert \widehat{y}_h\vert \leq \zeta_h\text{ a.e.\ in }\Omega\,,\\
			+\infty &\text{ else}\,,
		\end{cases}
	\end{align*}
	and $\smash{I_{\{u_D^h\}}^{\Gamma_D}}\colon \mathcal{L}^0(\mathcal{S}_h^{\Gamma_D})\to \mathbb{R}\cup\{+\infty\}$, for every $\widehat{v}_h\in \mathcal{L}^0(\mathcal{S}_h^{\Gamma_D})$, is defined by 
	\begin{align*}
		\smash{I_{\{u_D^h\}}^{\Gamma_D}}(\widehat{v}_h)
		&\coloneqq 
		\begin{cases}
			0&\text{ if }\widehat{v}_h=u_D^h\text{ a.e.\ on }\Gamma_D\,,\\
			+\infty &\text{ else}\,.
		\end{cases}
	\end{align*}
	Since the functional \eqref{eq:discrete_primal} is proper, strictly convex, weakly coercive, and lower semi-continuous, the direct method in the calculus of variations yields the existence~of~a~unique~minimizer~$u_h^{cr}\in K_h^{cr}$, called  \textit{discrete primal solution}. We reserve the notation $u_h^{cr}\in K_h^{cr}$ for the discrete~primal solution.\medskip
	
	\hspace*{-2.5mm}$\bullet$ \textit{Discrete dual problem.} A \textit{(Fenchel) dual problem} (in the sense of \cite[Rem.\ 4.2, p.\ 60/61]{ET99}) to the minimization of \eqref{eq:discrete_primal} is given via the maximization of $D_h^{rt}\colon \mathcal{R}T^0(\mathcal{T}_h)\to \mathbb{R}\cup\{-\infty\}$, for every $y_h\in \mathcal{R}T^0(\mathcal{T}_h)$~defined~by 
	\begin{align}
		\begin{aligned}
			D_h^{rt}(y_h)&\coloneqq -\int_{\Omega}{\phi^*_h(\cdot,\Pi_hy_h)\,\mathrm{d}x}+(y_h\cdot n,u_D^h)_{\Gamma_D}-I_{\smash{K_h^{rt,*}}}(y_h)\\
			&\coloneqq -\int_{\Omega}{\phi^*_h(\cdot,\Pi_hy_h)\,\mathrm{d}x}+(y_h\cdot n,u_D^h)_{\Gamma_D}-\smash{I_{\{-f_h\}}^{\Omega}}(\textup{div}\,y_h)-\smash{I_{\{g_h\}}^{\Gamma_N}}(y_h\cdot n)\,,
		\end{aligned}\label{eq:discrete_dual}
	\end{align}
	where
	\begin{align*}
		K_h^{rt,*}\coloneqq \big\{y_h\in \mathcal{R}T^0(\mathcal{T}_h)\mid \textup{div}\,y_h=-f_h\text{ a.e.\ in }\Omega\,,\;y_h\cdot n=g_h\text{ a.e.\ on }\Gamma_N\big\}\,,
	\end{align*}
	$I_{\smash{K_h^{rt,*}}}\coloneqq \smash{I_{\{-f_h\}}^{\Omega}}\circ\textup{div}+\smash{I_{\{g_h\}}^{\Gamma_N}}((\cdot)\cdot n)\colon \mathcal{R}T^0(\mathcal{T}_h)\to \mathbb{R}\cup\{+\infty\}$, $\smash{I_{\{g_h\}}^{\Gamma_N}}\colon \mathcal{L}^0(\mathcal{S}_h^{\Gamma_N})\to \mathbb{R}\cup\{+\infty\}$, for every $\widehat{v}_h\in \mathcal{L}^0(\mathcal{S}_h^{\Gamma_N})$, is defined by
	\begin{align*}
		\smash{I_{\{g_h\}}^{\Gamma_N}}(\widehat{v}_h)\coloneqq\begin{cases}
			0&\text{ if }\widehat{v}_h=g_h\text{ a.e.\ on }\Gamma_N\,,\\
			+\infty&\text{ else}\,,
		\end{cases}
	\end{align*}
	and $\phi^*_h\colon \Omega\times \mathbb{R}^d\to \mathbb{R}$, for a.e.\ $x\in \Omega$ and every $s\in \mathbb{R}^d$, defined by
	\begin{align*}
		\phi^*_h(x,s)=\begin{cases}
			\tfrac{1}{2}\vert  s\vert ^2&\text{ if }\vert s\vert \leq \zeta_h(x)\,,\\
			\zeta_h(x)\vert  s\vert -\tfrac{\zeta_h(x)^2}{2}&\text{ if }\vert s\vert >\zeta_h(x)\,,
		\end{cases} 
	\end{align*}
	denotes the Fenchel conjugate to $\phi_h\colon \Omega\times\mathbb{R}^d\to \mathbb{R} \cup\{+\infty\}$ (with respect to the second argument), for a.e.\ $x\in \Omega$ and every $t\in\mathbb{R}^d$,  defined by 
	\begin{align*}
		\phi_h(x,t)
  \coloneqq \tfrac{1}{2}\vert t\vert^2+\begin{cases}
			0&\text{ if }\vert t\vert\leq \zeta_h(x)\,,\\
			+\infty&\text{ if }\vert t\vert >\zeta_h(x)\,.
		\end{cases}
	\end{align*} 
	
	\begin{remark}
        Note that the discrete dual problem is not only defined on
        finitely additive vector measures but on a subset of finitely additive vector measures that are absolutely continuous with respect to the $d$-dimensional Lebesgue measure $\mathrm{d}x$. Hence, if $f=f_h\in \mathcal{L}^0(\mathcal{T}_h)$, $g=g_h\in \mathcal{L}^0(\mathcal{S}_h^{\Gamma_N})$, $u_D=u_D^h\in \mathcal{L}^0(\mathcal{S}_h^{\Gamma_D})$, and $\zeta=\zeta_h\in \mathcal{L}^0(\mathcal{T}_h)$,
        vector fields  that are admissible in the discrete dual problem, \textit{i.e.}, $y_h\in K_h^{rt,*}$, are also admissible for the continuous~dual~problem,~\textit{i.e.},~$y_h\in K^{*}$. 
	\end{remark}
	
	The identification of the (Fenchel) dual problem (in the sense of \cite[Rem.\ 4.2, p.\ 60/61]{ET99}) to the minimization of \eqref{eq:discrete_primal} with the maximization of  \eqref{eq:discrete_dual} can be found in the proof of the following result that also establishes the validity of a strong~duality~relation and convex~optimality~relations.
	
	\begin{theorem}[strong duality and convex optimality relations]\label{thm:discrete_duality} The following statements apply:
		\begin{itemize}[noitemsep,topsep=2pt,leftmargin=!,labelwidth=\widthof{(ii)}]
			\item[(i)]  The (Fenchel) dual problem to the minimization of \eqref{eq:discrete_primal} is given via the maximization~of~\eqref{eq:discrete_dual}.  
			\item[(ii)]  There \hspace{-0.1mm}exists \hspace{-0.1mm}a \hspace{-0.1mm}maximizer \hspace{-0.1mm}$z_h^{rt}\hspace{-0.15em} \in\hspace{-0.15em} \mathcal{R}T^0(\mathcal{T}_h)$ \hspace{-0.1mm}of \hspace{-0.1mm}\eqref{eq:discrete_dual} \hspace{-0.1mm}satisfying \hspace{-0.1mm}the \hspace{-0.1mm}\emph{discrete~\hspace{-0.1mm}admissibility~\hspace{-0.1mm}conditions}
			\begin{alignat}{2}
				\textup{div}\,z_h^{rt}&=-f_h&&\quad\text{ a.e.\ in }\Omega\label{eq:discrete_admissibility.1}\,,\\
				z_h^{rt}\cdot n &= g_h&&\quad \text{ a.e.\ on }\Gamma_N\label{eq:discrete_admissibility.2}\,. 
			\end{alignat}
			In addition, there
			holds a \emph{discrete strong duality relation}, \textit{i.e.}, we have that
			\begin{align}
				I_h^{cr}(u_h^{cr}) = D_h^{rt}(z_h^{rt})\,.\label{eq:discrete_strong_duality}
			\end{align}
			\item[(iii)]  If \hspace{-0.1mm}a \hspace{-0.1mm}discrete \hspace{-0.1mm}strong \hspace{-0.1mm}duality \hspace{-0.1mm}relation \hspace{-0.1mm}\textcolor{black}{holds}, \hspace{-0.1mm}then \hspace{-0.1mm}there \hspace{-0.1mm}hold \hspace{-0.1mm}the \hspace{-0.1mm}\emph{discrete \hspace{-0.1mm}convex \hspace{-0.1mm}optimality~\hspace{-0.1mm}\mbox{relation}}\enlargethispage{5mm}
			\begin{align}\label{eq:discrete_optimality}
				\nabla_h u_h^{cr}=D_t\phi^*_h(\cdot,\Pi_h z_h^{rt})=\left.\begin{cases}
					\Pi_h z_h^{rt}&\text{ if }\vert \Pi_h z_h^{rt}\vert\leq \zeta_h\,,\\
					\zeta_h\frac{\Pi_h z_h^{rt}}{\vert \Pi_h z_h^{rt}\vert }&\text{ if }\vert \Pi_h z_h^{rt}\vert> \zeta_h
				\end{cases}\right\}\quad \text{ a.e.\ in }\Omega\,.
			\end{align}  
			
		\end{itemize}
	\end{theorem}

	\begin{remark}[equivalent discrete convex optimality relations]
		By the standard equality condition in the Fenchel--Young inequality (\textit{cf}.\ \cite[Prop.\ 5.1, p.\ 21]{ET99}), the discrete convex optimality relation \eqref{eq:discrete_optimality} is equivalent to 
		\begin{align}
			\Pi_h z_h^{rt}\cdot \nabla_h u_h^{cr}=\phi^*_h(\cdot,	\Pi_h z_h^{rt})+\phi_h(\cdot,\nabla_h u_h^{cr})\quad \text{ a.e.\ in }\Omega \,.\label{eq:discrete_optimality_reg}
		\end{align}
		In addition, according to \cite[Cor.\ 5.2(5.8), p.\ 22]{ET99},  the discrete convex optimality relation \eqref{eq:discrete_optimality} is equivalent to 
		\begin{align}\label{eq:discrete_optimality_regular2}
			\Pi_hz_h^{rt}\in \partial_t \phi_h(\cdot,\nabla_h u_h^{cr})\quad \text{ a.e.\ in }\Omega\,.
		\end{align} 
	\end{remark}
	\begin{proof}[Proof (of Theorem \ref{thm:discrete_duality}).]
		\textit{ad (i).} First, we introduce the functionals $G_h\colon  (\mathcal{L}^0(\mathcal{T}_h))^d\to \mathbb{R}\cup\{+\infty\}$ and $F_h\colon \mathcal{S}^{1,cr}(\mathcal{T}_h)\to \mathbb{R}\cup\{+\infty\}$, for every $\overline{y}_h\in  (\mathcal{L}^0(\mathcal{T}_h))^d$ and $v_h\in\mathcal{S}^{1,cr}(\mathcal{T}_h)$, respectively, defined by
		\begin{align*}
			G_h(\overline{y}_h)&\coloneqq \int_{\Omega}{\phi_h(\cdot,\overline{y}_h)\,\mathrm{d}x}\\&= \tfrac{1}{2}\|\overline{y}_h\|_{2,\Omega}^2+I_{K_{\zeta_h}(0)}^\Omega(\overline{y}_h) \,,\\
			F_h(v_h)&\coloneqq -(f_h,\Pi_h v_h)_{\Omega}-(g_h,\pi_hv_h)_{\Gamma_N}+I_{\{u_D^h\}}^{\Gamma_D}(v_h)\,.
		\end{align*}
		Then, for every $v_h\in \mathcal{S}^{1,cr}(\mathcal{T}_h)$, we have that
		\begin{align*}
			I_h^{cr}(v_h)= G_h(\nabla_h v_h)+F_h(v_h)\,.
		\end{align*}
		Thus, \hspace*{-0.175mm}in \hspace*{-0.175mm}accordance \hspace*{-0.175mm}with \hspace*{-0.175mm}\cite[\hspace*{-0.2mm}Rem.\ \hspace*{-0.2mm}4.2, \hspace*{-0.2mm}p.\ \hspace*{-0.2mm}60]{ET99}, \hspace*{-0.175mm}the \hspace*{-0.175mm}(Fenchel) \hspace*{-0.175mm}dual \hspace*{-0.175mm}problem \hspace*{-0.175mm}to \hspace*{-0.175mm}the \hspace*{-0.175mm}\mbox{minimization}~\hspace*{-0.175mm}of~\hspace*{-0.175mm}\eqref{eq:discrete_primal} is given via the maximization of $D_h^0\colon (\mathcal{L}	^0(\mathcal{T}_h))^d\to \mathbb{R}\cup\{-\infty\}$, 
		for every $\overline{y}_h\in (\mathcal{L}^0(\mathcal{T}_h))^d$~defined~by 
		\begin{align}\label{helper_functional}
			D_h^0(\overline{y}_h)\coloneqq -G^*_h(\overline{y}_h)- F^*_h(-\nabla^*_h\overline{y}_h)\,,
		\end{align}
		where $\nabla^*_h\colon  (\mathcal{L}^0(\mathcal{T}_h))^d\to (\mathcal{S}^{1,cr}(\mathcal{T}_h))^*$ denotes the adjoint operator to $\nabla_h \colon \mathcal{S}^{1,cr}(\mathcal{T}_h)\to (\mathcal{L}^0(\mathcal{T}_h))^d$.
		
		First, using element-wise the definition of the Fenchel conjugate of $\phi_h\colon \Omega\times \mathbb{R}^d\to \mathbb{R}\cup\{+\infty\}$ (with respect to the second argument), for every $\overline{y}_h\in (\mathcal{L}^0(\mathcal{T}_h))^d$, we have that
		\begin{align}\label{prop:discrete_duality.1}
			G^*_h(\overline{y}_h)=\int_{\Omega}{\phi^*_h(\cdot,\overline{y}_h)\,\mathrm{d}x}\,.
		\end{align}
		
		Second, for every $\overline{y}_h\in (\mathcal{L}^0(\mathcal{T}_h))^d$, using the lifting lemma \cite[Lemma A.1 in the case $\Gamma_C=\emptyset$]{BarGudKal24} and the discrete integration-by-parts formula \eqref{eq:pi}, we have that
		\begin{align}\label{prop:discrete_duality.2} 
			&F^*_h(-\nabla^*_h\overline{y}_h)
			\\&\quad=\sup_{v_h\in  \mathcal{S}^{1,cr}(\mathcal{T}_h)}{ \big\{ - (\overline{y}_h,\nabla_h v_h )_{\Omega} + (f_h,\Pi_h v_h)_{\Omega}+(g_h,\pi_hv_h)_{\Gamma_N}-I_{\{u_D^h\}}^{\Gamma_D}(v_h)\big\}} \notag
			\\&\quad=\sup_{v_h\in  \mathcal{S}^{1,cr}_D(\mathcal{T}_h)}{ \big\{ - (\overline{y}_h,\nabla_h v_h )_{\Omega} + (f_h,\Pi_h v_h )_{\Omega}+(g_h,\pi_hv_h )_{\Gamma_N}\big\}} \notag
			\\&\quad\quad- (\overline{y}_h,\nabla_h \widehat{u}_D^h )_{\Omega} + (f_h,\Pi_h \widehat{u}_D^h )_{\Omega}+(g_h,\pi_h\widehat{u}_D^h )_{\Gamma_N} \notag
			\\&\quad= 
			\begin{cases}
				\left.\begin{aligned} 
					I_{\{-f_h\}}^{\Omega}(\textup{div}\,y_h)+I_{\{g_h\}}^{\Gamma_N}(y_h\cdot n)\\
					-(y_h,\nabla_h \widehat{u}_D^h )_{\Omega} + (f_h,\Pi_h \widehat{u}_D^h )_{\Omega}+(g_h,\pi_h\widehat{u}_D^h )_{\Gamma_N}
				\end{aligned}\right\}&\left\{\begin{aligned}
					&	\text{ if there exists }y_h\in \mathcal{R}T^0(\mathcal{T}_h)\\&\text{ such that }\overline{y}_h=\Pi_hy_h\text{ a.e.\ in }\Omega\,,
				\end{aligned}\right.\\
				+\infty&\text{ else}\,,
			\end{cases}  \notag
			\\&\quad= 
			\begin{cases}
				\left.\begin{aligned} 
					I_{\{-f_h\}}^{\Omega}(\textup{div}\,y_h)+I_{\{g_h\}}^{\Gamma_N}(y_h\cdot n)\\
					-(y_h\cdot n, u_D^h )_{\Gamma_D} 
				\end{aligned}\right\}&\left\{\begin{aligned}
					&	\text{ if there exists }y_h\in \mathcal{R}T^0(\mathcal{T}_h)\\&\text{ such that }\overline{y}_h=\Pi_hy_h\text{ a.e.\ in }\Omega\,,
				\end{aligned}\right.\\
				+\infty&\text{ else}\,.
			\end{cases}  \notag
		\end{align} 
		Hence, using \eqref{prop:discrete_duality.1} and \eqref{prop:discrete_duality.2} in \eqref{helper_functional}, for every $\overline{y}_h=\Pi_h y_h\in  \Pi_h(\mathcal{R}T^0(\mathcal{T}_h))$,~where~${y_h\in \mathcal{R}T^0(\mathcal{T}_h)}$, 
		we arrive at the~claimed~\mbox{representation}~\eqref{eq:discrete_dual}. Since $D_h^0= -\infty$ in $(\mathcal{L}^0(\mathcal{T}_h))^d\setminus \Pi_h(\mathcal{R}T^0(\mathcal{T}_h))$, we can  restrict to $ \Pi_h(\mathcal{R}T^0(\mathcal{T}_h))$. Finally, we simply define the functional $D_h^{rt}\colon \mathcal{R}T^0(\mathcal{T}_h)\to \mathbb{R}\cup\{-\infty\}$, for every $y_h\in \mathcal{R}T^0(\mathcal{T}_h)$, by 
		\begin{align*}
			D_h^{rt}(y_h)\coloneqq D_h^0(\Pi_hy_h)\,.
		\end{align*} 
		
		\textit{ad (ii).} Since $G_h\colon (\mathcal{L}^0(\mathcal{T}_h))^d\to \mathbb{R}\cup\{+\infty\}$ and $F_h\colon  \mathcal{S}^{1,cr}(\mathcal{T}_h)\to \mathbb{R}\cup\textcolor{black}{\{+\infty\}}$ are proper, convex, and lower semi-continuous and  $G_h\colon (\mathcal{L}^0(\mathcal{T}_h))^d\to \mathbb{R}\cup\{+\infty\}$ is continuous at
		$ \nabla_h\widehat{u}_D^h\in \textup{dom}(G_h\circ \nabla_h)$ (recall that  $\|\frac{\nabla_h\widehat{u}_D^h}{\zeta_h}\|_{\infty,\Omega}<1$) with $\widehat{u}_D^h\in \textup{dom}(F_h)$, \textit{i.e.}, we have that
		\begin{align*}
			G_h(\overline{y}_h)\to G_h(\nabla_h\widehat{u}_D^h)\quad\big(\overline{y}_h\to \nabla_h\widehat{u}_D^h\quad\text{ in }(\mathcal{L}^0(\mathcal{T}_h))^d\big)\,,
		\end{align*}
		by \hspace*{-0.1mm}the \hspace*{-0.1mm}celebrated \hspace*{-0.1mm}Fenchel \hspace*{-0.1mm}duality \hspace*{-0.1mm}theorem \hspace*{-0.1mm}(\emph{cf}.\ \hspace*{-0.1mm}\cite[\hspace*{-0.1mm}Rem.\ \hspace*{-0.1mm}4.2, \hspace*{-0.1mm}(4.21), \hspace*{-0.1mm}p.\ \hspace*{-0.1mm}61]{ET99}),  \hspace*{-0.1mm}there~\hspace*{-0.1mm}exists~\hspace*{-0.1mm}a~\hspace*{-0.1mm}\mbox{maximizer}  $\overline{z}_h^0\in  (\mathcal{L}^0(\mathcal{T}_h))^d$ of  \eqref{helper_functional} and a strong duality relation applies, \textit{i.e.}, it holds that
		\begin{align*}
			I_h^{cr}(u_h^{cr})=D_h^0(\overline{z}_h^0)\,.
		\end{align*}
		Since $I_h^{cr}(u_h^{cr})<+\infty$ and  $\textup{dom}(-D_h^0)\subseteq \Pi_h(\mathcal{R}T^0(\mathcal{T}_h))$, there exists vector field $z_h^{rt}\in \mathcal{R}T^0(\mathcal{T}_h)$~such that $\overline{z}_h^0=\Pi_h z_h^{rt}$ a.e.\ in $\Omega$.
		In particular, we have that $D_h^0(\overline{z}_h^0)=D_h^{rt}(z_h^{rt})$, so~that~$z_h^{rt}\in \mathcal{R}T^0(\mathcal{T}_h)$ is a maximizer of \eqref{eq:discrete_dual} and a discrete strong duality relation applies, \textit{i.e.},  it holds that
        \begin{align*}
            I_h^{cr}(u_h^{cr})=D_h^{rt}(z_h^{rt})\,.
        \end{align*}
		
		\textit{ad (iii).} \hspace*{-0.15mm}By \hspace*{-0.15mm}the \hspace*{-0.15mm}standard \hspace*{-0.15mm}(Fenchel) \hspace*{-0.15mm}convex \hspace*{-0.15mm}duality \hspace*{-0.15mm}theory \hspace*{-0.15mm}(\emph{cf}.\ \hspace*{-0.15mm}\cite[Rem.\ \hspace*{-0.15mm}4.2, \hspace*{-0.15mm}(4.24),~\hspace*{-0.15mm}(4.25),~\hspace*{-0.15mm}p.~\hspace*{-0.15mm}61]{ET99}), there hold the convex optimality relations
		\begin{align}
			-\nabla_h^*\Pi_h z_h^{rt}&\in \partial F_h(u_h^{cr})\,,\label{prop:discrete_duality.3}\\
			\Pi_h z_h^{rt}&\in \partial G_h(\nabla_h u_h^{cr})\,.\label{prop:discrete_duality.4}
		\end{align}
		The inclusion \eqref{prop:discrete_duality.4} is equivalent to the discrete convex optimality relation \eqref{eq:discrete_optimality}. The inclusion \eqref{prop:discrete_duality.3} is equivalent to the discrete admissibility conditions \eqref{eq:discrete_admissibility.1}--\eqref{eq:discrete_admissibility.2}.
	\end{proof}
	
	\newpage Given  a discrete dual solution  (\textit{i.e.}, a maximizer of \eqref{eq:discrete_dual}) and a Lagrange multiplier associated with the discrete admissibility condition \eqref{eq:discrete_admissibility.1}, which we obtain as a by-product in the computation of the discrete dual solution and which is actually 
	a placeholder for the (local) $L^2$-projection of the discrete primal solution, the discrete primal solution is immediately available via an \emph{inverse generalized Marini formula} (\textit{cf}.\ \cite{Mar85,ArnoldBrezzi85,ArbogastChen95,BartelsKaltenbachOverview24}).

	\begin{theorem}[inverse \hspace*{-0.1mm}generalized \hspace*{-0.1mm}Marini \hspace*{-0.1mm}formula]\label{thm:discrete_reconstruction} 
		\hspace*{-0.1mm}Given \hspace*{-0.1mm}a \hspace*{-0.1mm}discrete \hspace*{-0.1mm}dual \hspace*{-0.1mm}solution \hspace*{-0.1mm}$z_h^{rt}\hspace*{-0.175em}\in \hspace*{-0.175em} \mathcal{R}T^0(\mathcal{T}_h)$ and a Lagrange multiplier $\smash{\overline{\lambda}}_h\in \mathcal{L}^0(\mathcal{T}_h)$ such that for every $(y_h,\overline{\eta}_h)^\top\in \mathcal{R}T^0_N(\mathcal{T}_h)\times \mathcal{L}^0(\mathcal{T}_h)$, it holds that
		\begin{align} 
				-(D_t\phi^*_h(\cdot,\Pi_hz_h^{rt}),\Pi_hy_h)_{\Omega}-(\smash{\overline{\lambda}}_h,\textup{div}\,y_h)_{\Omega}&=-(y_h\cdot n,u_D^h )_{\Gamma_D}\,,\label{eq:saddle_point_formulation.1}\\
				(\textup{div}\,z_h^{rt},\overline{\eta}_h)_{\Omega}&=-(f_h,\overline{\eta}_h)_{\Omega}\,,\label{eq:saddle_point_formulation.2} 
		\end{align}
		the discrete primal solution $u_h^{cr}\in \smash{\mathcal{S}^{1,cr}(\mathcal{T}_h)}$ is given via the \emph{inverse generalized Marini formula}
		\begin{align}\label{eq:gen_marini}
u_h^{cr}=\smash{\overline{\lambda}}_h+D_t\phi^*_h(\cdot,\Pi_hz_h^{rt})\cdot(\textup{id}_{\mathbb{R}^d}-\Pi_h\textup{id}_{\mathbb{R}^d})\quad \text{ a.e.\ in }\Omega\,.
		\end{align} 
        In particular, from \eqref{eq:gen_marini}, it follows that
        \begin{align}\label{def:lambda}
            \Pi_hu_h^{cr}=\smash{\overline{\lambda}}_h\quad \text{ a.e.\ in }\Omega\,.
        \end{align}
	\end{theorem}  
	
	\begin{remark}[well-posedness of the non-linear saddle point formulation \eqref{eq:saddle_point_formulation.1},\eqref{eq:saddle_point_formulation.2}]\hphantom{   }
    \begin{itemize}[noitemsep,topsep=2pt,leftmargin=!,labelwidth=\widthof{(ii)}]
        \item[(i)] While equation \eqref{eq:saddle_point_formulation.2} is equivalent to the discrete admissibility condition \eqref{eq:discrete_admissibility.1}, setting $\smash{\overline{\lambda}}_h\coloneqq \Pi_hu_h^{cr}\in \mathcal{L}^0(\mathcal{T}_h)$ and using the discrete convex optimality relation \eqref{eq:discrete_optimality},
    the equation \eqref{eq:saddle_point_formulation.1} is an immediate consequence of the discrete integration-by-parts formula \eqref{eq:pi} together with $\pi_h u_h^{cr}=u_D^h$ a.e.\ on $\Gamma_D$.\enlargethispage{10mm}

    \item[(ii)] The inverse generalized Marini formula (\textit{cf}.\ \eqref{eq:gen_marini}) together with the regularity of the discrete dual energy functional \eqref{eq:discrete_dual} motivates to, first, compute the discrete dual solution (\textit{e.g.}, using a semi-implicit discretized $L^2$-gradient flow) and, then, to compute the discrete primal solution for free via the inverse generalized Marini formula (\textit{cf}.\ \eqref{eq:gen_marini}).
    \end{itemize}
	\end{remark}
	
	\begin{proof}[Proof (of Theorem \ref{thm:discrete_reconstruction}).]
		To begin with, we introduce the abbreviation
		\begin{align}\label{thm:discrete_reconstruction.1}
			\widehat{u}_h^{cr}\coloneqq\smash{\overline{\lambda}}_h+D_t\phi^*_h(\cdot,\Pi_hz_h^{rt})\cdot(\textup{id}_{\mathbb{R}^d}-\Pi_h\textup{id}_{\mathbb{R}^d})\in \mathcal{L}^1(\mathcal{T}_h)\,.
		\end{align}
		Then, by definition \eqref{thm:discrete_reconstruction.1}, we have that
		\begin{alignat}{2}
    \label{thm:discrete_reconstruction.2.1} 
				\nabla_h \widehat{u}_h^{cr}&=D_t\phi^*_h(\cdot,\Pi_hz_h^{rt})&&\quad\text{ a.e.\ in }\Omega\,,\\
				\Pi_h\widehat{u}_h^{cr}&=\smash{\overline{\lambda}}_h&&\quad\text{ a.e.\ in }\Omega\,.\label{thm:discrete_reconstruction.2.2}  
		\end{alignat}
		From the equations \eqref{eq:saddle_point_formulation.1},\eqref{thm:discrete_reconstruction.2.2}  and $\nabla_h u_h^{cr}=D_t\phi^*_h(\cdot,\Pi_hz_h^{rt})$ (\textit{cf}.\ \eqref{eq:discrete_optimality}), for every ${y_h\in \mathcal{R}T^0_N(\mathcal{T}_h)}$, it follows that
		\begin{align}\label{thm:discrete_reconstruction.3}
			(\nabla_h u_h^{cr},\Pi_hy_h)_{\Omega}+(\Pi_h\widehat{u}_h^{cr},\textup{div}\,y_h)_{\Omega}-(u_D^h ,y_h\cdot n)_{\Gamma_D}=0\,.
		\end{align}
		Resorting to the discrete integration-by-parts~formula~\eqref{eq:pi} in \eqref{thm:discrete_reconstruction.3}, for every $y_h\in \mathcal{R}T^0_N(\mathcal{T}_h)$, we find that
		\begin{align}\label{thm:discrete_reconstruction.4}
			\begin{aligned}
				( \widehat{u}_h^{cr}-u_h^{cr},\textup{div}\,y_h)_{\Omega} &=(  \Pi_h\widehat{u}_h^{cr}- \Pi_hu_h^{cr},\textup{div}\,y_h)_{\Omega} \\
				&= ( \Pi_h\widehat{u}_h^{cr},\textup{div}\,y_h)_{\Omega} +(\nabla_h u_h^{cr},\Pi_hy_h)_{\Omega} -(y_h\cdot n,\pi_h u_h^{cr} )_{\Gamma_D}
				\\&=0\,.
			\end{aligned}
		\end{align} 
		Since, on the other hand, we have that $\nabla_h( \widehat{u}_h^{cr}-u_h^{cr})=0$ a.e.\ in $\Omega$, \textit{i.e.}, $ \widehat{u}_h^{cr}-u_h^{cr}\in\mathcal{L}^0(\mathcal{T}_h)$, \eqref{thm:discrete_reconstruction.4}  
		implies that
        \begin{align}\label{thm:discrete_reconstruction.5}
            \widehat{u}_h^{cr}-u_h^{cr}\in (\textup{div}\,(\mathcal{R}T^0_N(\mathcal{T}_h)))^{\perp}\,,
        \end{align}
        where the orthogonality needs to be understood in $\mathcal{L}^0(\mathcal{T}_h)$ equipped with $(\cdot,\cdot)_{\Omega}$.
        Due to  $\vert \Gamma_D\vert>0$, and, thus, ${\Gamma_N\neq \partial\Omega}$, we have that 
        \begin{align*}
            \textup{div}\,(\mathcal{R}T^0_N(\mathcal{T}_h))=\mathcal{L}^0(\mathcal{T}_h)\,,
        \end{align*}
        so that from \eqref{thm:discrete_reconstruction.5},  we conclude that $\widehat{u}_h^{cr}=u_h^{cr}$ in $\mathcal{S}^{1,cr}(\mathcal{T}_h)$.
	\end{proof}
	\newpage
	
	\section{\textit{A priori} error analysis}\label{sec:apriori}
	
	\qquad In this section, resorting to the discrete convex duality relations established in Section~\ref{sec:discrete}, following \cite{BarGudKal24}, 
	we derive an \textit{a priori} error identity for  the discrete primal problem and the discrete dual problem at the same time. From this \textit{a priori} error identity, we derive an \textit{a priori} error estimate with an explicit error decay rate that is quasi-optimal. 
	To this end, we proceed analogously to the continuous setting and start with examining the \textit{discrete primal-dual gap estimator} $\eta_{\textup{gap},h}^2\colon K_h^{cr}\times K_h^{rt,*}\to [0,+\infty)$, for every $v_h\in K_h^{cr}$ and $y_h\in K_h^{rt,*}$ defined by
	\begin{align}\label{eq:discrete_primal_dual_gap_estimator}
		\eta_{\textup{gap},h}^2(v_h,y_h)\coloneqq I_h^{cr}(v_h)-D_h^{rt}(y_h)\,.
	\end{align}
	
	The discrete primal-dual gap estimator has the following integral representation.
	
	\begin{lemma}\label{lem:discrete_primal_dual_gap_estimator}
		For every $v_h\in K_h^{cr}$ and $y_h\in K_h^{rt,*}$,   we have that
		\begin{align*}
			\eta_{\textup{gap},h}^2(v_h,y_h)\coloneqq  \int_{\Omega}{\big\{\phi^*_h(\cdot,\Pi_h y_h)-\Pi_h y_h\cdot\nabla_hv_h+ \phi_h(\cdot,\nabla_hv_h)\big\}\,\mathrm{d}x} \,.
		\end{align*}
	\end{lemma}
	
	\begin{remark}[Interpretation of the discrete primal-dual gap estimator] 
		The estimator $\eta_{\textup{gap},h}^2$ measures the violation of the discrete convex optimality relations \eqref{eq:discrete_optimality}.
	\end{remark}
	
	\begin{proof}[Proof (of Lemma \ref{lem:discrete_primal_dual_gap_estimator})]\let\qed\relax
		Using the discrete admissibility condition \eqref{eq:discrete_admissibility.1} and the discrete integration-by-parts formula \eqref{eq:pi}, 
		for every $v_h\in K_h^{cr}$ and $y_h\in K_h^{rt,*}$,~we~find~that
		\begin{align*}
			I_h^{cr}(v_h)-D_h^{rt}(y_h)&=
			\tfrac{1}{2}\| \nabla_h v_h\|_{2,\Omega}^2-(f_h,\Pi_h v_h)_{\Omega}-(g_h,\pi_hv_h)_{\Gamma_N}\\&\quad+\int_{\Omega}{\phi^*_h(\cdot,\Pi_h y_h)\,\mathrm{d}x}-(y_h\cdot n,u_D^h)_{\Gamma_D}
			\\&= \int_{\Omega}{\phi_h(\cdot,\nabla_h v_h)\,\mathrm{d}x}+(\textup{div}\,y_h,\Pi_hv_h)_{\Omega}-(y_h\cdot n,\pi_h v_h)_{\Gamma_N}\\&\quad+\int_{\Omega}{\phi^*_h(\cdot,\Pi_h y_h)\,\mathrm{d}x}-(y_h\cdot n,\pi_h v_h)_{\Gamma_D}
			\\&=  \int_{\Omega}{\big\{\phi^*_h(\cdot,\Pi_h y_h)-\Pi_h y_h\cdot\nabla_hv_h+ \phi_h(\cdot,\nabla_hv_h)\big\}\,\mathrm{d}x}\,,
		\end{align*}
    which is the claimed integral representation of $\eta_{\textup{gap},h}^2\colon K_h^{cr}\times K_h^{rt,*}\to \mathbb{R}$.
	\end{proof}

    In the next step, we  derive meaningful representations of the \emph{discrete optimal strong convexity measures} for the discrete primal energy functional~\eqref{eq:discrete_primal}  at  the 
    discrete primal solution $u_h^{cr}\in K_h^{cr}$, \textit{i.e.}, for  $\rho_{I_h^{cr}}^2\colon\hspace*{-0.2em} K_h^{cr}\to [0,+\infty)$, for every $v_h\in K_h^{cr}$ defined by
	\begin{align}\label{def:discrete_optimal_primal_error}
		\smash{\rho_{I_h^{cr}}^2(v_h)\coloneqq I_h^{cr}(v_h)-I_h^{cr}(u_h^{cr})\,,}
	\end{align}
	and for the negative of the discrete dual energy functional \eqref{eq:discrete_dual}, \textit{i.e.}, for $\rho_{\smash{-D_h^{rt}}}^2\colon\hspace*{-0.1em} K_h^{rt,*}\hspace*{-0.1em}\to\hspace*{-0.1em} [0,+\infty)$,  
	for every~$\smash{y_h\hspace*{-0.1em}\in\hspace*{-0.1em} K_h^{rt,*}}$ defined by
	\begin{align}\label{def:discrete_optimal_dual_error}
		\smash{\rho_{-D_h^{rt}}^2(y_h)\coloneqq- D_h^{rt}(y_h)+D_h^{rt}(z_h^{rt})\,,}
	\end{align}
	which \hspace*{-0.1em}will \hspace*{-0.1em}serve \hspace*{-0.1em}as \hspace*{-0.1em}\emph{`natural'} \hspace*{-0.1em}error \hspace*{-0.1em}quantities \hspace*{-0.1em}in \hspace*{-0.1em}the \hspace*{-0.1em}discrete \hspace*{-0.1em}primal-dual \hspace*{-0.1em}gap \hspace*{-0.1em}identity \hspace*{-0.1em}(\textit{cf}.\ \hspace*{-0.1em}Theorem~\hspace*{-0.1em}\ref{thm:discrete_prager_synge_identity}).\enlargethispage{5mm}
	
	\begin{lemma}[representations of the discrete optimal strong convexity measures]\label{lem:discrete_strong_convexity_measures}
		The following statements apply:
		\begin{itemize}[noitemsep,topsep=2pt,leftmargin=!,labelwidth=\widthof{(ii)}]
			\item[(i)] For every $v_h\in K_h^{cr}$, we have that
			\begin{align*}
				\rho_{I_h^{cr}}^2(v_h,u_h^{cr})=\tfrac{1}{2}\|\nabla_h v_h-\nabla_h u_h^{cr}\|_{2,\Omega}^2+\big(\tfrac{\vert \Pi_h z_h^{rt}\vert}{\zeta_h}-1,\zeta_h^2-\nabla_h u_h^{cr}\cdot\nabla_h v_h\big)_{\smash{\{\vert \nabla_h u_h^{cr}\vert=\zeta_h\}}}\,.
			\end{align*}
			\item[(ii)] For every $y_h\in K_h^{rt,*}$, we have that
			\begin{align*}
				\rho_{-D_h^{rt}}^2(y_h,z_h^{rt})=(D_t^2\phi^*_h\{\Pi_h y_h,\Pi_h z_h^{rt}\}(\Pi_h y_h-\Pi_hz_h^{rt}), \Pi_h y_h-\Pi_hz_h^{rt})_{\Omega} \,.
			\end{align*}
		\end{itemize}
	\end{lemma}
	
	\begin{remark}
		\begin{itemize}[noitemsep,topsep=2pt,leftmargin=!,labelwidth=\widthof{(ii)}]
			\item[(i)] Since $\{\vert \Pi_hz_h^{rt}\vert\ge \zeta_h\}=\{\vert \nabla_h u_h^{cr}\vert=\zeta_h\}$ (\textit{cf}.\ \eqref{eq:discrete_optimality}) and $\nabla_h u_h^{cr}\cdot\nabla_h v_h\leq \zeta_h^2$ a.e.\ in $\{\vert \nabla_h u_h^{cr}\vert=\zeta_h\}$ for all $v_h\in K_h^{cr}$, for every $v_h\in K_h^{cr}$, we have that
			\begin{align*}
				\smash{\big(\tfrac{\vert \Pi_hz_h^{rt}\vert}{\zeta_h}-1\big)(\zeta_h^2-\nabla_h u_h^{cr}\cdot\nabla_h v_h)\ge 0\quad\text{ a.e.\ in }\{\vert \nabla_h u_h^{cr}\vert=\zeta_h\}}\,.
			\end{align*}
			
			\item[(ii)] Analogously to \eqref{eq:D2_phi_prime}, since for a.e.\ $x\in \Omega$ and every $t\in \mathbb{R}^d$, we have that
			\begin{align}\label{eq:D2_phih_prime}
				D_t^2\phi^*_h(x,t) =\begin{cases}
					\mathrm{I}_{d\times d}&\text{ if }\vert t\vert \leq \zeta_h(x)\,,\\
					\frac{\zeta_h(x)}{\vert t\vert }(\mathrm{I}_{d\times d}-\frac{t\otimes t}{\vert t\vert^2})&\text{ if }\vert t\vert>\zeta_h(x)\,,
				\end{cases} 
			\end{align} 
			for a.e.\ $x\in \Omega$ and every $t,s\in \mathbb{R}^d$, by the Cauchy--Schwarz inequality, we observe that
			\begin{align*}
				D_t^2\phi^*_h(x,t):s\otimes s &= \begin{cases}
					\vert s\vert^2&\text{ if }\vert t\vert \leq \zeta_h(x)\,,\\
						\frac{\zeta_h(x)}{\vert t\vert }(\vert s\vert^2-\frac{(t\cdot s)^2}{\vert t\vert^2})&\text{ if }\vert t\vert>\zeta_h(x)
				\end{cases}\\
					&\ge \begin{cases}
					\vert s\vert^2&\text{ if }\vert t\vert \leq \zeta_h(x)\,,\\
					0&\text{ if }\vert t\vert>\zeta_h(x)\,.
				\end{cases}
			\end{align*}
		\end{itemize}
	\end{remark}

	\begin{proof}[Proof (of Lemma \ref{lem:discrete_strong_convexity_measures})]
		\textit{ad (i).} Using the discrete integration-by-parts formula \eqref{eq:pi},  for every $v_h\in K_h^{cr}$, due to $\pi_h v_h=\pi_h u_h^{cr}$ a.e.\ on $\Gamma_D$, we find that 
		\begin{align*}
			I_h^{cr}(v_h)-I_h^{cr}(u_h^{cr})&=\tfrac{1}{2}\|\nabla_h v_h\|_{2,\Omega}^2-\tfrac{1}{2}\|\nabla_h u_h^{cr}\|_{2,\Omega}^2\\&\quad-(f_h,\Pi_h v_h-\Pi_h u_h^{cr})_{\Omega}-(g_h,\pi_h v_h-\pi_h u_h^{cr})_{\Gamma_N}
            \\&=\tfrac{1}{2}\|\nabla_h v_h\|_{2,\Omega}^2-\tfrac{1}{2}\|\nabla_h u_h^{cr}\|_{2,\Omega}^2
            \\&\quad+(\textup{div}\,z_h^{rt},\Pi_h v_h-\Pi_h u_h^{cr})_{\Omega}-(z_h^{rt}\cdot n,\pi_h v_h-\pi_h u_h^{cr})_{\partial\Omega}\\&
			=\tfrac{1}{2}\|\nabla_h v_h\|_{2,\Omega}^2-\tfrac{1}{2}\|\nabla_h u_h^{cr}\|_{2,\Omega}^2-(\Pi_h z_h^{rt},\nabla_h v_h-\nabla_h u_h^{cr})_{\Omega} \,.
		\end{align*}
		Next, using that, due to the discrete convex optimality relation \eqref{eq:discrete_optimality}, we have that
		\begin{align*}
			\begin{aligned}
				\Pi_h z_h^{rt}&=\nabla_h u_h^{cr}&&\quad \text{ a.e.\ in }\{\vert\nabla_h u_h^{cr}\vert<\zeta_h\}\,,\\
				\Pi_h z_h^{rt}&=\tfrac{\vert \Pi_h z_h^{rt}\vert}{\zeta_h}\,\nabla_h u_h^{cr}&&\quad \text{ a.e.\ in }\{\vert \nabla_h u_h^{cr}\vert=\zeta_h\}\,.
			\end{aligned}
		\end{align*} 
		Therefore, using the binomial formula, we find that 
		\begin{align*}
			I_h^{cr}(v_h)-I_h^{cr}(u_h^{cr})&=\tfrac{1}{2}\|\nabla_h v_h\|_{2,\Omega}^2-\tfrac{1}{2}\|\nabla_h u_h^{cr}\|_{2,\Omega}^2-(\nabla_h u_h^{cr},\nabla_h v_h-\nabla_h u_h^{cr})_{\Omega}\\&\quad+
			\big(\big(\tfrac{\vert \Pi_h z_h^{rt}\vert}{\zeta_h}-1\big)\nabla_h u_h^{cr},\nabla_h u_h^{cr}-\nabla_h v_h\big)_{\smash{\{\vert\nabla_h u_h^{cr} \vert=\zeta_h\}}}
			\\&=\tfrac{1}{2}\|\nabla_h v_h-\nabla_h u_h^{cr}\|_{2,\Omega}^2
			+\big(\big(\tfrac{\vert \Pi_h z_h^{rt}\vert}{\zeta_h}-1\big),\zeta_h^2-\nabla_h u_h^{cr}\cdot\nabla_h v_h\big)_{\smash{\{\vert\nabla_h u_h^{cr} \vert=\zeta_h\}}}\,,
		\end{align*}
        which is the claimed representation of  $\rho_{I_h^{cr}}^2\colon \smash{K_h^{cr}}\to [0,+\infty)$.

		\textit{ad (ii).} Using the discrete admissibility condition \eqref{eq:discrete_admissibility.1} and 
		 Taylor expansion,~for~every~$y_h\in K_h^{rt,*}$, due to $y_h\cdot n =z_h^{rt}\cdot n$ a.e.\ on $\Gamma_N$, we find that\vspace{-0.5mm}\enlargethispage{10mm}
		\begin{align*}
			-D_h^{rt}(y_h)+D_h^{rt}(z_h^{rt})&=\int_{\Omega}{\phi^*_h(\cdot,\Pi_h y_h)}-\int_{\Omega}{\phi^*_h(\cdot,\Pi_h z_h^{rt})}-(y_h\cdot n-z_h^{rt}\cdot n, u_D^h)_{\Gamma_D}\\&
			=(D_t\phi^*_h(\cdot,\Pi_h z_h^{rt}),\Pi_h y_h-\Pi_hz_h^{rt})_{\Omega}-(y_h\cdot n-z_h^{rt}\cdot n, u_D^h)_{\partial\Omega}\\&\quad+(D_t^2\phi^*_h\{\Pi_h y_h,\Pi_h z_h^{rt}\}(\Pi_h y_h-\Pi_hz_h^{rt}), \Pi_h y_h-\Pi_hz_h^{rt})_{\Omega}
			\\&=(D_t^2\phi^*_h\{\Pi_h y_h,\Pi_h z_h^{rt}\}(\Pi_h y_h-\Pi_hz_h^{rt}), \Pi_h y_h-\Pi_hz_h^{rt})_{\Omega}
			\,,
		\end{align*}
		where \hspace{-0.1mm}we \hspace{-0.1mm}used \hspace{-0.1mm}that, \hspace{-0.1mm}by \hspace{-0.1mm}the \hspace{-0.1mm}discrete \hspace{-0.1mm}convex \hspace{-0.1mm}optimality \hspace{-0.1mm}condition \hspace{-0.1mm}\eqref{eq:discrete_optimality}, \hspace{-0.1mm}the \hspace{-0.1mm}discrete \hspace{-0.1mm}\mbox{integration-by-} parts \hspace{-0.1mm}formula \hspace{-0.1mm}\eqref{eq:pi}, \hspace{-0.1mm}and \hspace{-0.1mm}the \hspace{-0.1mm}discrete \hspace{-0.1mm}admissibility \hspace{-0.1mm}condition \hspace{-0.1mm}\eqref{eq:discrete_admissibility.1},  \hspace{-0.1mm}for \hspace{-0.1mm}every \hspace{-0.1mm}$y_h\hspace{-0.15em}\in \hspace{-0.15em}K_h^{rt,*}$,~\hspace{-0.1mm}it~\hspace{-0.1mm}holds~\hspace{-0.1mm}that
		\begin{align*}
			(D_t\phi^*_h(\cdot,\Pi_h z_h^{rt}),\Pi_h y_h-\Pi_hz_h^{rt})_{\Omega}&=(\nabla_h u_h^{cr},\Pi_h y_h-\Pi_hz_h^{rt})_{\Omega}\\&=(\textup{div}\,z_h^{rt}-\textup{div}\,y_h,\Pi_hu_h^{cr} )_{\Omega}+(y_h\cdot n-z_h^{rt}\cdot n, \pi_hu_h^{cr})_{\partial\Omega}
			\\&=(y_h\cdot n-z_h^{rt}\cdot n, u_D^h)_{\partial\Omega}\,,
		\end{align*} 
            which is the claimed representation of  $\smash{\rho_{-D_h^{rt}}^2\colon K_h^{rt,*}\to [0,+\infty)}$.
	\end{proof}
	
	Eventually, we have everything at our disposal to establish a discrete \textit{a posteriori} error identity that identifies the \textit{discrete primal-dual total error} $\rho_{\textup{tot},h}^2\colon K_h^{cr}\times K_h^{rt,*}\to [0,+\infty)$, for every $v_h\in K_h^{cr}$ and $y_h\in K_h^{rt,*}$ defined by 
	\begin{align}\label{eq:discrete_primal_dual_error}
		\rho_{\textup{tot},h}^2(v_h,y_h)\coloneqq \rho_{I_h^{cr}}^2(v_h)+\rho_{-D_h^{rt}}^2(y_h)\,,
	\end{align}
	with the discrete primal-dual gap estimator \eqref{eq:discrete_primal_dual_gap_estimator}.
	
	\begin{theorem}[discrete primal-dual gap identity]\label{thm:discrete_prager_synge_identity}
		For every $v_h\hspace{-0.15em}\in\hspace{-0.15em} K_h^{cr}$ and $y_h\hspace{-0.15em}\in \hspace{-0.15em}K_h^{rt,*}$,~we~have~that
		\begin{align*}
			\rho_{\textup{tot},h}^2(v_h,y_h)=\eta_{\textup{gap},h}^2(v_h,y_h)\,.
		\end{align*}
	\end{theorem}
	
	\begin{proof}\let\qed\relax
		Using the definitions \eqref{eq:discrete_primal_dual_gap_estimator}, \eqref{def:discrete_optimal_primal_error},  \eqref{def:discrete_optimal_dual_error}, \textcolor{black}{\eqref{eq:discrete_primal_dual_error}}, and the discrete strong duality relation~\eqref{eq:discrete_strong_duality}, for every $v_h\in K_h^{cr}$ and $y_h\in K_h^{rt,*}$, we find that
		\begin{align*}
			\rho_{\textup{tot},h}^2(v_h,y_h)&=\rho_{I_h^{cr}}^2(v_h,u_h^{cr})+\rho_{-D_h^{rt}}^2(y_h,z_h^{rt})
			\\&=I_h^{cr}(v_h)-I_h^{cr}(u_h^{cr})+D_h^{rt}(z_h^{rt})-D_h^{rt}(y_h)
			\\&=I_h^{cr}(v_h)-D_h^{rt}(y_h)
			\\&=\eta_{\textup{gap},h}^2(v_h,y_h)\,.\tag*{$\qedsymbol$}
		\end{align*}
	\end{proof}
	
	If we insert the (Fortin) quasi-interpolations \eqref{CR-interpolant} and \eqref{RT-interpolant} of the primal and the~dual~solution (assuming that the singular part vanishes, \textit{i.e.}, $\mu^s=0$), respectively,  in the discrete primal-dual identity (\textit{cf}.\  Theorem \ref{thm:discrete_prager_synge_identity}), we arrive at an \textit{a priori} error identity, which allows us to derive error decay rates depending on the regularity of the dual solution.\enlargethispage{5mm}
 
	\begin{theorem}[\textit{a priori} error identity and error decay rates]\label{thm:apriori_identity}
		If $f_h\hspace*{-0.1em}\coloneqq \hspace*{-0.1em}\Pi_h f\hspace*{-0.1em}\in \hspace*{-0.1em}\mathcal{L}^0(\mathcal{T}_h)$,~${g_h\hspace*{-0.1em}\coloneqq\hspace*{-0.1em}\pi_h g}$ $\in \mathcal{L}^0(\mathcal{S}_h^{\Gamma_N})$, $u_D^h\coloneqq \pi_h u_D\in \mathcal{L}^0(\mathcal{S}_h^{\Gamma_D})$, and $\zeta_h\coloneqq\Pi_h \zeta\in \mathcal{L}^0(\mathcal{T}_h)$, the following statements~apply:
		\begin{itemize}[noitemsep,topsep=2pt,leftmargin=!,labelwidth=\widthof{(ii)}]
			\item[(i)] If $\mu^s=0$ and $z\in V_{p,q}(\Omega)$, where $p>2$ and $q>\frac{2d}{d+2}$, then
			$\Pi_h^{cr}u\in K_h^{cr}$, $\Pi_h^{rt}z\in K_h^{rt,*}$, and 
			\begin{align*}
				\rho_{\textup{tot},h}^2(\Pi_h^{cr}u,\Pi_h^{rt}z)&= \int_{\Omega}{\big\{\phi^*_h(\cdot,\Pi_h\Pi_h^{rt}z)-\Pi_h \Pi_h^{rt}z\cdot\Pi_h \nabla u+ \phi_h(\cdot,\Pi_h \nabla u)\big\}\,\mathrm{d}x} 
				\\&\leq 
                (\sqrt{d}+1)\,\|z-\Pi_h\Pi_h^{rt}z\|_{2,\Omega}^2+\tfrac{3}{2}
				\|\zeta_h-	\zeta\|_{2,\Omega}^2\,.
			\end{align*} 
			\item[(ii)] If, in addition, $z\in (W^{\nu,2}(\Omega))^d$ and $\zeta\in W^{\nu,2}(\Omega)$, where $\nu \in (0,1]$, and $\{\mathcal{T}_h\}_{h>0}$ is quasi-uniform, 
			then 
			\begin{align*}
				\rho_{\textup{tot},h}^2(\Pi_h^{cr}u,\Pi_h^{rt}z)\leq c\, h^{2\nu}\,\big(\vert z\vert_{\nu,2,\Omega}^2+\vert\zeta\vert_{\nu,2,\Omega}^2\big)\,.
			\end{align*}
		\end{itemize} 
	\end{theorem}
	
	\begin{proof}
		\emph{ad (i)}. First, using \eqref{eq:grad_preservation} and \eqref{eq:trace_preservation}, we observe that
		\begin{align*}
			\begin{aligned}
				\vert \nabla_h \Pi_h^{cr}u\vert&=	\vert \Pi_h \nabla u\vert\leq \Pi_h\vert \nabla u\vert\leq \zeta_h&&\quad\text{ a.e.\ in }\Omega\,,\\
				\pi_h \Pi_h^{cr}u&=\pi_h u=\pi_h u_D=u_D^h&&\quad\text{ a.e.\ on }\Gamma_D\,,
			\end{aligned}
		\end{align*}
		\textit{i.e.}, it holds that $\Pi_h^{cr}u\in K_h^{cr}$.
		Second, 
		using \eqref{eq:div_preservation} and \eqref{eq:normal_trace_preservation}, 
		we observe that
		\begin{align*}
			\begin{aligned}  
				\textup{div}\,\Pi_h^{rt}z&=\Pi_h \textup{div}\,z=\Pi_h(-f)=-f_h&&\quad\text{ a.e.\ in }\Omega\,, \\
				\Pi_h^{rt}z\cdot n&=\pi_h(z\cdot n)=\pi_h g=g_h&&\quad\text{ a.e.\ on }\Gamma_N\,,
			\end{aligned}
		\end{align*}
		\textit{i.e.}, it holds that $\Pi_h^{rt}z\in K_h^{rt,*}$. Then, 
		using Theorem \ref{thm:discrete_prager_synge_identity} together with  Lemma~\ref{lem:discrete_primal_dual_gap_estimator}~as~well~as~\eqref{eq:grad_preservation}, abbreviating $\widetilde{z}_h\coloneqq \Pi_h\Pi_h^{rt}z\in (\mathcal{L}^0(\mathcal{T}_h))^d$, we find that
		\begin{align}\label{thm:apriori_identity.1}
			\rho_{\textup{tot},h}^2(\Pi_h^{cr}u,\Pi_h^{rt}z)&= \int_{\Omega}{\big\{\phi^*_h(\cdot,\widetilde{z}_h)-\widetilde{z}_h\cdot\Pi_h \nabla u+ \phi_h(\cdot,\Pi_h \nabla u)\big\}\,\mathrm{d}x}\,.
		\end{align}
        Due to $\phi_h(\cdot,\Pi_h \nabla u)=\frac{1}{2}\vert \Pi_h \nabla u\vert^2$ a.e.\ in $\Omega$ and $\phi(\cdot, \nabla u)=\frac{1}{2}\vert \nabla u\vert^2$ a.e.\ in $\Omega$, 
		by Jensen's~inequality, it holds that $\int_{\Omega}{\phi_h(\cdot,\Pi_h \nabla u)\,\mathrm{d}x}\leq\int_{\Omega}{\phi(\cdot, \nabla u)\,\mathrm{d}x}$,
        so that from \eqref{thm:apriori_identity.1}, we infer that
        \begin{align}\label{thm:apriori_identity.2}
            \rho_{\textup{tot},h}^2(\Pi_h^{cr}u,\Pi_h^{rt}z)&\leq \int_{\Omega}{\big\{\phi^*_h(\cdot,\widetilde{z}_h)-\widetilde{z}_h\cdot \nabla u+ \phi(\cdot, \nabla u)\big\}\,\mathrm{d}x}\,.
        \end{align}
		Then, using in \eqref{thm:apriori_identity.2} the convex optimality relations \eqref{eq:optimality.1},\eqref{eq:optimality_regular}, that by the convexity~of~${\phi^*\!\in\! C^1(\mathbb{R}^d)}$, it holds that
		\begin{align*}
			- \phi^*(\cdot,z)\leq - \phi^*(\cdot,\widetilde{z}_h)-D_t\phi^*(\cdot,\widetilde{z}_h)\cdot(z-\widetilde{z}_h)\quad\text{ a.e.\ in }\Omega\,,
		\end{align*} 
        we find that
		\begin{align}\label{thm:apriori_identity.3}
            \begin{aligned} 
			\rho_{\textup{tot},h}^2(\Pi_h^{cr}u,\Pi_h^{rt}z)&\leq \int_{\Omega}{\big\{\phi^*_h(\cdot,\widetilde{z}_h)+(z-\widetilde{z}_h)\cdot D_t\phi^*(\cdot,z)- \phi^*(\cdot,z)\big\}\,\mathrm{d}x}
			\\&\leq (D_t\phi^*(\cdot,z)-D_t\phi^*(\cdot,\widetilde{z}_h),z-\widetilde{z}_h)_{\Omega}
			+\int_{\Omega}{\big\{\phi^*_h(\cdot,\widetilde{z}_h)-\phi^*(\cdot,\widetilde{z}_h)\big\}\,\mathrm{d}x}
            \\&\eqqcolon I_h^1+I_h^2\,.
            \end{aligned}
		\end{align}
    So, it remains to estimate $I_h^1$ and $I_h^2$:

    \emph{ad $I_h^1$.} Using the Taylor expansion
        \begin{align*}
        D_t\phi^*(\cdot,z)=D_t\phi^*(\cdot,\widetilde{z}_h)+\textcolor{black}{\int_0^1D_t^2\phi^*(\lambda z+(1-\lambda)\widetilde{z}_h)(z-\widetilde{z}_h)\,\mathrm{d}\lambda}
        \quad \text{ a.e.\ in }\Omega\,,
        \end{align*}
        and that, by   \eqref{eq:D2_phi_prime}, it holds that $\vert D_t^2\phi^*(x,t) \vert\leq \sqrt{d}+1$ for a.e.\ $x\in \Omega$ and all $t\in \mathbb{R}^d$, 
		we obtain
		\begin{align}\label{thm:apriori_identity.4}
            \begin{aligned} 
			     I_h^1&=\textcolor{black}{\int_0^1(D_t^2\phi^*(\lambda z+(1-\lambda)\widetilde{z}_h)(z-\widetilde{z}_h),z-\widetilde{z}_h)_{\Omega}\,\mathrm{d}\lambda}
                \\&\leq    (\sqrt{d}+1)\,\|z-\widetilde{z}_h\|_{2,\Omega}^2\,.
            \end{aligned}
		\end{align}

    \emph{ad $I_h^2$.} 
		Due to $\phi^*_h(\cdot,\widetilde{z}_h)\hspace{-0.1em}=\hspace{-0.1em}\zeta_h\vert \widetilde{z}_h\vert-\tfrac{1}{2}\zeta_h^2$ a.e.\ in $\{\vert \widetilde{z}_h\vert\hspace{-0.1em}>\hspace{-0.1em} \zeta_h\}$, $\phi^*(\cdot,\widetilde{z}_h)\hspace{-0.1em}=\hspace{-0.1em}\zeta\vert \widetilde{z}_h\vert-\tfrac{1}{2}\zeta^2$ a.e.~in~$\{\vert \widetilde{z}_h\vert\hspace{-0.1em}>\hspace{-0.1em}\zeta\}$, $\phi^*_h(\cdot,\widetilde{z}_h)\hspace{-0.1em}=\hspace{-0.1em}\frac{1}{2}\vert \widetilde{z}_h\vert^2$ a.e.\ in $\{\vert \widetilde{z}_h\vert\hspace{-0.1em}\leq\hspace{-0.1em} \zeta_h\}$, and $\phi^*(\cdot,\widetilde{z}_h)\hspace{-0.1em}=\hspace{-0.1em}\frac{1}{2}\vert \widetilde{z}_h\vert^2$ a.e.\ in $\{\vert\widetilde{z}_h\vert\hspace{-0.1em}\leq\hspace{-0.1em} \zeta\}$, it holds that
		\begin{align}\label{thm:apriori_identity.5}
            \begin{aligned}
			I_h^2&= 
            \tfrac{1}{2}(2\vert \widetilde{z}_h\vert-	\zeta-\zeta_h,\zeta_h-\zeta)_{\{\vert\widetilde{z}_h\vert>\max\{\zeta_h,\zeta\}\}}   
			\\&\quad-\tfrac{1}{2}\|\vert\widetilde{z}_h\vert-\zeta_h\|_{2,\{\zeta_h<\vert\widetilde{z}_h\vert\leq \zeta\}}^2 
			+\tfrac{1}{2}\|\vert\widetilde{z}_h\vert-\zeta\|_{2,\{\zeta<\vert\widetilde{z}_h\vert\leq \zeta_h\}}^2
   \,.
            \end{aligned}
		\end{align}
        Using that $\vert \widetilde{z}_h\vert, \zeta_h\perp \zeta_h-\zeta$ in 
        $L^2(T)$ for all $T\in \mathcal{T}_h$ and that $\{\vert \widetilde{z}_h\vert>\zeta_h\}$ can be decomposed into elements from $\mathcal{T}_h$ as $\vert \widetilde{z}_h\vert,\zeta_h\in \mathcal{L}^0(\mathcal{T}_h)$, we find that
        \begin{align}\label{thm:apriori_identity.6}
            \begin{aligned}
            \tfrac{1}{2}(2\vert \widetilde{z}_h\vert-	\zeta-\zeta_h,\zeta_h-\zeta)_{\{\vert \widetilde{z}_h\vert>\max\{\zeta_h,\zeta\}\}}   &= \tfrac{1}{2}(2\vert \widetilde{z}_h\vert-	\zeta-\zeta_h,\zeta_h-\zeta)_{\{\vert \widetilde{z}_h\vert>\zeta_h\}}
            \\&\quad -\tfrac{1}{2}(2\vert \widetilde{z}_h\vert-	\zeta-\zeta_h,\zeta_h-\zeta)_{\{\zeta_h<\vert  \widetilde{z}_h\vert\leq \zeta\}}   
            \\&\leq\tfrac{1}{2}\|\zeta_h-	\zeta\|_{\{\vert  \widetilde{z}_h\vert>\zeta_h\}}^2
            \\&\quad+\tfrac{1}{2}\|\vert\widetilde{z}_h\vert-	\zeta\|_{2,\{\zeta_h<\vert\widetilde{z}_h\vert\leq \zeta\}}\|\zeta_h-\zeta\|_{2,\{\zeta_h<\vert\widetilde{z}_h\vert\leq \zeta\}}
            \\&\quad+\tfrac{1}{2}\|\vert\widetilde{z}_h\vert-	\zeta_h\|_{2,\{\zeta_h<\vert\widetilde{z}_h\vert\leq \zeta\}}\|\zeta_h-\zeta\|_{2,\{\zeta_h<\vert\widetilde{z}_h\vert\leq \zeta\}}
            \\&\leq \tfrac{3}{2}\|\zeta_h-\zeta\|_{2,\{\vert  \widetilde{z}_h\vert>\zeta_h\}}^2\,,
            \end{aligned}\hspace{-5mm}
        \end{align}
        Then, inserting \eqref{thm:apriori_identity.6} into  \eqref{thm:apriori_identity.5}, we deduce that
        \begin{align}\label{thm:apriori_identity.7}
            I_h^2\leq \tfrac{3}{2}\|\zeta_h-\zeta\|_{2,\Omega}^2\,.
        \end{align}
        Eventually, combining \eqref{thm:apriori_identity.7} and \eqref{thm:apriori_identity.4} in \eqref{thm:apriori_identity.3}, we conclude that the claimed \textit{a priori} error estimate applies.
		
		\emph{ad (ii).} From (i), using Jensen's inequality and the fractional approximation properties of $\Pi_h$ (\textit{cf}.\ \cite[Thm.\ 18.16]{EG21I}) and $\Pi_h^{rt}$ (\emph{cf}.\ \cite[Thms.\ 16.4, 16.6]{EG21I}), we conclude that
		\begin{align*}
			\rho_{\textup{tot},h}^2(\Pi_h^{cr}u,\Pi_h^{rt}z)&\leq  c\,\|z-\Pi_hz\|_{2,\Omega}^2+c\,\|\Pi_h(z-\Pi_h^{rt}z)\|_{2,\Omega}^2+c\,\|\zeta-\Pi_h\zeta\|_{2,\Omega}^2
			\\&\leq  c\,\|z-\Pi_hz\|_{2,\Omega}^2+c\,\|z-\Pi_h^{rt}z\|_{2,\Omega}^2+c\,\|\zeta-\Pi_h\zeta\|_{2,\Omega}^2
			\\&\leq  c\, h^{2\nu}\,\big(\vert z\vert_{\nu,2,\Omega}^2+\vert\zeta\vert_{\nu,2,\Omega}^2\big)\,, 
		\end{align*}
        which is the claimed \textit{a priori} error estimate.
	\end{proof}

		\section{Numerical experiments}\label{sec:experiments}
	
	\hspace{5mm}In this section, we review the theoretical findings of Section \ref{sec:aposteriori} and Section \ref{sec:apriori} via numerical experiments. All experiments were conducted using the finite element software package \textsf{FEniCS}  (version 2019.1.0, \textit{cf}.\  \cite{LW10}). All graphics were generated using the \textsf{Matplotlib}~library~(version~3.5.1,~\textit{cf}.~\cite{Hun07}).\enlargethispage{2.5mm}
 
	\subsection{Implementation details regarding the optimization procedure}
 
	\hspace{5mm}The iterative maximization of the discrete dual energy functional \eqref{eq:discrete_dual} is realized using a semi-implicit discretized $L^2$-gradient flow modified with a residual stopping criterion guaranteeing the necessary accuracy~in~the~optimization~procedure. For the semi-implicit treatment,
    we introduce the auxiliary function $\varphi_h^*\colon \Omega\times \mathbb{R}_{\ge 0}\to \mathbb{R}$, for a.e. $x\in \Omega$ and every $r\ge 0$ defined by\vspace{-0.5mm}
    \begin{align*}
        \varphi_h^*(x,r)\coloneqq \begin{cases}
            \frac{1}{2}r^2&\text{ if } r\leq \zeta_h(x)\,,\\
            \zeta_h(x)r-\frac{\zeta_h(x)^2}{2}&\text{ if } r> \zeta_h(x)\,.
        \end{cases}
    \end{align*}

	\begin{algorithm}[Semi-implicit discretized $L^2$-gradient flow]\label{algorithm}
		Let $z^0_h\in K_h^{rt,*}$ 
        and let $\tau,\varepsilon_{stop}^h> 0$. 
  Then, for every $k\in \mathbb{N}$:
		\begin{description}[noitemsep,topsep=2pt,labelwidth=\widthof{\textit{(ii)}},leftmargin=!,font=\normalfont\itshape]
			\item[(i)] Compute the iterates $(z_h^k,\smash{\overline{\lambda}}_h^k)^\top\in \mathcal{R}T^0(\mathcal{T}_h)\times\mathcal{L}^0(\mathcal{T}_h)$ with $z_h^k\cdot n =g_h$ a.e.\ on $\Gamma_N$ such that for every $(y_h,\overline{\eta}_h)^\top\in \mathcal{R}T^0_N(\mathcal{T}_h)\times \mathcal{L}^0(\mathcal{T}_h)$, it holds that\vspace{-0.5mm}
			\begin{align}
   \begin{aligned}
       \Big(\Pi_h\mathrm{d}_{\tau}  z_h^k+\tfrac{D_t\varphi_h^*(\cdot,\vert \Pi_h z_h^{k-1}\vert )}{\vert \Pi_h z_h^{k-1}\vert}\Pi_h z_h^k ,\Pi_hy_h \Big)_{\smash{\Omega}}+(\smash{\overline{\lambda}}_h^k,\textup{div}\,y_h)_{\Omega}&=(y_h\cdot n,u_D^h)_{\Gamma_D}\,,\\
       (\textup{div}\,z_h^k,\overline{\eta}_h)_{\Omega}&=-(f_h,\overline{\eta}_h)_{\Omega}\,.
  \end{aligned}\label{alg:step_1}
			\end{align}
			where $\smash{\mathrm{d}_{\tau} z_h^k\coloneqq \frac{1}{\tau}(z_h^k-z_h^{k-1})\in \mathcal{R}T^0_N(\mathcal{T}_h)}$ denotes the backward difference quotient.
		\item[(ii)] Compute the residual $\smash{r_h^k\in \mathcal{R}T^0_N(\mathcal{T}_h)}$ such that for every $\smash{y_h\in\mathcal{R}T^0_N(\mathcal{T}_h)}$, it holds that\vspace{-0.5mm}
		\begin{align}
			(r_h^k,y_h)_{\Omega}=\Big(\tfrac{D_t\varphi_h^*(\cdot,\vert \Pi_h z_h^k\vert )}{\vert \Pi_h z_h^k\vert} \Pi_h z_h^k,\Pi_h y_h \Big)_{\smash{\Omega}}+(\smash{\overline{\lambda}}_h^k,\textup{div}\,y_h)_{\Omega}- (y_h\cdot n,u_D^h)_{\Gamma_D}\,.\label{alg:step_2}
		\end{align}
		Stop if $\smash{\|r_h^k\|_{2,\Omega}\leq \varepsilon_{stop}^h}$; otherwise, increase $k\!\to\! k+1$~and~continue~with~step~\textit{(i)}.
		\end{description}
	\end{algorithm}

    The following proposition establishes the  well-posedness (\textit{i.e.}, existence of iterates), stability (\textit{i.e.}, \textit{a priori} bounds), and convergence (\textit{i.e.}, the convergence of the iterates to a discrete dual solution) of  Algorithm \ref{algorithm}.

 \begin{proposition}[well-posedness, stability, and convergence of Algorithm~\ref{algorithm}]\label{prop:stability}
 Let the assumptions of Algorithm~\ref{algorithm} be satisfied. Then, the following statements apply:
 \begin{itemize}[noitemsep,topsep=2pt,leftmargin=!,labelwidth=\widthof{(iii)}]
    \item[(i)] Algorithm \ref{algorithm} is well-posed, \textit{i.e.}, for every $k\hspace{-0.15em}\in\hspace{-0.15em} \mathbb{N}$, given the most-recent iterate ${z_h^{k-1}\hspace{-0.15em}\in\hspace{-0.15em} \smash{\mathcal{R}T^0(\mathcal{T}_h)}}$, there exists a unique iterate $z_h^k\in \smash{\mathcal{R}T^0(\mathcal{T}_h)}$ solving \eqref{alg:step_1}.
  \item[(ii)] Algorithm \ref{algorithm} is unconditionally strongly stable, \textit{i.e.}, for every  $L\in \mathbb{N}$,~it~holds~that\vspace{-0.5mm}
  \begin{align*}
  -D_h^{rt}(z_h^L)+\tau \sum_{k=1}^L{\|\Pi_h \mathrm{d}_{\tau}z_h^k\|_{2,\Omega}^2}\leq -D_h^{rt}(z_h^0)\,.
  \end{align*}

  \item[(iii)] Algorithm \ref{algorithm} terminates after a finite number of steps, \textit{i.e.}, there exists $k^*\in \mathbb{N}$ such~that $\|r_h^{k^*}\|_{2,\Omega}\leq \varepsilon_{stop}^h$.
 \end{itemize}
 \end{proposition}

 The proof of Proposition \ref{prop:stability}(ii) is essentially based on the following inequality.

 \begin{lemma}\label{lem:stability}
  For a.e.\ $x\in \Omega$ and every $a,b\in \mathbb{R}^d$, it holds that
  \begin{align*}
  \smash{\tfrac{D_t\varphi_h^*(x,\vert a\vert)}{\vert a\vert } b\cdot(b-a)\ge \varphi_h^*(x,\vert b\vert)-\varphi_h^*(x,\vert a\vert)+\tfrac{1}{2}\tfrac{D_t\varphi_h^*(\cdot,\vert a\vert)}{\vert a\vert}\vert b-a\vert^2\,.}
  \end{align*}
 \end{lemma}

 \begin{proof}
    Follows from \cite[Appx.\ A.2]{Bar21}, since $\varphi_h^*(x,\cdot)\in  C^1(\mathbb{R}_{\ge 0})$ and $(t\mapsto D_t\varphi_h^*(x,t)/t)\in  C^0(\mathbb{R}_{\ge 0})$ is positive and \textcolor{black}{non-increasing} for a.e.\ $x\in \Omega$.
 \end{proof}

 \begin{proof}[Proof (of Proposition \ref{prop:stability}).]
 \textit{ad (i).} Since  $D_t\varphi_h^*(x,t)/t\ge  0$ for a.e.\ $x\in \Omega$ and all  $t\ge 0$,~the~well-posedness of Algorithm \ref{algorithm} is a direct consequence of the Brezzi splitting~theorem~(\textit{cf}.~\mbox{\cite[Thm.~6.4]{Bartels16}}).\newpage
 
 \textit{ad (ii).}
 Let $L\hspace{-0.1em}\in\hspace{-0.1em} \mathbb{N}$ be arbitrary. Then,
  for every $k\hspace{-0.1em}\in\hspace{-0.1em} \{1,\dots,L\}$, choosing ${y_h\hspace{-0.1em}=\hspace{-0.1em}\mathrm{d}_{\tau} z_h^k\hspace{-0.1em}\in\hspace{-0.1em} \mathcal{R}T^0_N(\mathcal{T}_h)}$ in \eqref{alg:step_1}, due to
  $\textup{div}\,(\mathrm{d}_{\tau} z_h^k)=0$ a.e.\ in $\Omega$, we find that\vspace{-0.25mm}
  \begin{align}\label{prop:stability.1}
  \|\Pi_h \mathrm{d}_{\tau} z_h^k\|_{2,\Omega}^2+\Big(\tfrac{D_t\varphi_h^*(\cdot,\vert \Pi_h z_h^{k-1}\vert )}{\vert \Pi_h z_h^{k-1}\vert}\Pi_h z_h^k,\Pi_h \mathrm{d}_{\tau} z_h^k\Big)_{\smash{\Omega}}-(\mathrm{d}_{\tau} z_h^k\cdot n,u_D^h)_{\Gamma_D}=0\,.
  \end{align}
  According to Lemma \ref{lem:stability} with $a=\Pi_h z_h^{k-1}|_T\in \mathbb{R}^d$ and $b=\Pi_h z_h^k|_T\in \mathbb{R}^d$ applied for all $T\in \mathcal{T}_h$, for every $k\in \{1,\dots,L\}$, we have that\vspace{-0.25mm}
  \begin{align}\label{prop:stability.2}
    \tfrac{D_t\varphi_h^*(\cdot,\vert \Pi_h z_h^{k-1}\vert )}{\vert \Pi_h z_h^{k-1}\vert}\Pi_h z_h^k\cdot\Pi_h \mathrm{d}_{\tau} z_h^k\ge \mathrm{d}_{\tau}[\phi_h^*(\cdot,\Pi_h z_h^k )]\quad\text{ a.e.\ in }\Omega\,.
  \end{align}
  Using \eqref{prop:stability.2} in \eqref{prop:stability.1}, for every $k\in \{1,\dots,L\}$, 
  we arrive at\vspace{-0.25mm}
  \begin{align}\label{prop:stability.4}
  \smash{\|\Pi_h \mathrm{d}_{\tau}z_h^k\|_{2,\Omega}^2-\mathrm{d}_{\tau}[D_h^{rt}(z_h^k)]\leq 0\,.}
  \end{align}
  Summation of \eqref{prop:stability.4} with respect to $k\hspace{-0.1em}\in\hspace{-0.1em}\{1,\dots,L\}$, using $\tau\smash{\sum_{k=1}^L{\hspace{-0.1em}\mathrm{d}_{\tau} [D_h^{rt}(z_h^k)]}\hspace{-0.1em}=\hspace{-0.1em}D_h^{rt}(z_h^L)\hspace{-0.1em}-\hspace{-0.1em}D_h^{rt}(z_h^0)}$, yields the claimed stability estimate.

 \textit{ad (iii).} Due to (ii), we have that $\smash{\|\Pi_h \mathrm{d}_{\tau} z_h^k\|_{2,\Omega}^2\to 0}$ $(k\to \infty)$. Thus, by $\smash{\textup{div}\,(\mathrm{d}_{\tau}z_h^k)=0}$~a.e.~in~$\Omega$, the finite-dimensionality of $\smash{\mathcal{R}T^0_N(\mathcal{T}_h)}$, and the equivalence of norms, it holds that\vspace{-0.25mm}
 \begin{align}\label{prop:stability.5}
  \smash{z_h^k-z_h^{k-1}}\to 0\quad \text{ in }\mathcal{R}T^0_N(\mathcal{T}_h)\quad (k\to \infty)\,.
 \end{align}
 In addition, due to (ii), we have that $D_h^{rt}(z_h^k)\geq D_h^{rt}(z_h^0)$ and, thus,
 $(z_h^k)_{k\in \mathbb{N}}\subseteq \mathcal{R}T^0(\mathcal{T}_h)$~is~bounded. In addition, due to \cite[Prop.\ 7.4]{Bartels16} together with $\vert \Gamma_D\vert>0$, for every $k\in \mathbb{N}$, we have that\vspace{-0.25mm}
 \begin{align*}
     \|\smash{\overline{\lambda}}_h^k\|_{2,\Omega}&\leq \sup_{y_h\in \mathcal{R}T^0_N(\mathcal{T}_h)\,:\,\|y_h\|_{H(\textup{div};\Omega)\leq 1}}{\smash{\big\{(\smash{\overline{\lambda}}_h^k,\textup{div}\,y_h)_{\Omega}\big\}}}
     \\[-1mm]&=\sup_{y_h\in \mathcal{R}T^0_N(\mathcal{T}_h)\,:\,\|y_h\|_{H(\textup{div};\Omega)\leq 1}}{\Big\{\textcolor{black}{(y_h\cdot n,u_D^h)_{\Gamma_D}-
      \Big(\Pi_h \mathrm{d}_{\tau} z_h^k-\tfrac{D_t\varphi_h^*(\cdot,\vert \Pi_h z_h^{k-1}\vert )}{\vert \Pi_h z_h^{k-1}\vert}\Pi_h z_h^k ,\Pi_hy_h \Big)_{\smash{\Omega}}}\Big\}}
      \\[-0.5mm]&\leq \|\Pi_h \mathrm{d}_{\tau} z_h^k\|_{2,\Omega}+\|\Pi_h z_h^k\|_{2,\Omega}+c_h\,\|u_D^h\|_{2,\Gamma_D}\,,
 \end{align*}
 implying \hspace{-0.15mm}that \hspace{-0.15mm}$(\smash{\overline{\lambda}}_h^k)_{k\in \mathbb{N}}\!\subseteq \!\mathcal{L}^0(\mathcal{T}_h)$ \hspace{-0.15mm}is \hspace{-0.15mm}bounded.
 \hspace{-0.5mm}Due \hspace{-0.15mm}to \hspace{-0.15mm}the \hspace{-0.15mm}finite-dimensionality \hspace{-0.15mm}of \hspace{-0.15mm}$\mathcal{R}T^0(\mathcal{T}_h)$~\hspace{-0.15mm}and~\hspace{-0.15mm}$\mathcal{L}^0(\mathcal{T}_h)$, the Bolzano--Weierstrass theorem yields subsequences $(z_h^{k_\ell})_{\ell\in \mathbb{N}}\subseteq \mathcal{R}T^0(\mathcal{T}_h)$, $(\smash{\overline{\lambda}}_h^{k_\ell})_{\ell\in \mathbb{N}}\subseteq \mathcal{L}^0(\mathcal{T}_h)$ as well as limits $\widetilde{z}_h\in \mathcal{R}T^0(\mathcal{T}_h)$, $\widetilde{\lambda}_h\in \mathcal{L}^0(\mathcal{T}_h)$ such that\vspace{-0.25mm}
 \begin{alignat}{3}
       \smash{z_h^{k_\ell}}&\to \widetilde{z}_h&&\quad\text{ in }\mathcal{R}T^0(\mathcal{T}_h)&&\quad (\ell\to \infty)\,,\label{prop:stability.6} \\[-0.5mm]
       \smash{\smash{\overline{\lambda}}_h^{k_\ell}}&\to \widetilde{\lambda}_h&&\quad\text{ in }\mathcal{L}^0(\mathcal{T}_h)&&\quad (\ell\to \infty)\,. \label{prop:stability.6.2} 
 \end{alignat}
 Due to $\smash{\textup{div}\,z_h^k=-f_h}$ a.e.\ in $\Omega$ and $\smash{z_h^k\cdot n=g_h}$ a.e.\ on $\Gamma_N$ for all $k\in \mathbb{N}$, \eqref{prop:stability.6} implies that\vspace{-0.25mm}
 \begin{alignat}{2}
     \textup{div}\, \widetilde{z}_h&=-f_h&&\quad \text{ a.e.\ in }\Omega\,,\label{prop:stability.6.3} \\
     \widetilde{z}_h\cdot n&=g_h&&\quad\textup{ a.e.\ on }\Gamma_N\,.\label{prop:stability.6.4} 
 \end{alignat} 
 In addition, due to \eqref{prop:stability.5}, \eqref{prop:stability.6} implies that\vspace{-0.25mm}\enlargethispage{11.5mm}
 \begin{align}\label{prop:stability.7}
  \smash{z_h^{k_\ell-1}}\to \widetilde{z}_h\quad\text{ in }\mathcal{R}T^0_N(\mathcal{T}_h)\quad (\ell\to \infty)\,.
 \end{align}
 Thus, using \eqref{prop:stability.6}--\eqref{prop:stability.7}, by passing for $\ell\to \infty$ in \eqref{alg:step_1}, for every $y_h\in \mathcal{R}T^0_N(\mathcal{T}_h)$, we obtain\vspace{-0.25mm}
 \begin{align}\label{prop:stability.8}
  \Big(\tfrac{D_t\varphi_h^*(\cdot,\vert \Pi_h \widetilde{z}_h\vert )}{\vert \Pi_h \widetilde{z}_h\vert}\Pi_h \widetilde{z}_h ,\Pi_h y_h \Big)_{\smash{\Omega}}+(\widetilde{\lambda}_h,\textup{div}\,y_h)_{\Omega}-(y_h\cdot n,u_D^h)_{\Gamma_D}=0\,,
 \end{align}
 which together with \eqref{prop:stability.6.3} and \eqref{prop:stability.6.4} proves that $\smash{\widetilde{z}_h\in \mathcal{R}T^0(\mathcal{T}_h)}$ is a discrete dual solution and $\smash{\widetilde{\lambda}_h=\Pi_h u_h^{cr}}$ a.e.\ in $\Omega$.
 Hence, using \eqref{prop:stability.5} and \eqref{prop:stability.8},  for every $y_h\in \mathcal{R}T^0_N(\mathcal{T}_h)$,~we~obtain\vspace{-0.5mm}
 \begin{align*}
(r_h^{k_\ell},y_h)_{\smash{\Omega}}&=\Big(\tfrac{D_t\textcolor{black}{\varphi_h^*}(\cdot,\vert \Pi_h z_h^{k_\ell}\vert )}{\vert \Pi_h z_h^{k_\ell}\vert}\Pi_h z_h^{k_\ell},\Pi_h y_h \Big)_{\smash{\Omega}}+(\smash{\overline{\lambda}}_h^{k_\ell},\textup{div}\,y_h)_{\Omega}- (y_h\cdot n,u_D^h)_{\Gamma_D}
  \\[-0.5mm]&\to \Big(\tfrac{D_t\textcolor{black}{\varphi_h^*}(\cdot,\vert \Pi_h \widetilde{z}_h\vert )}{\vert \Pi_h \widetilde{z}_h\vert}\Pi_h \widetilde{z}_h ,\Pi_h y_h \Big)_{\smash{\Omega}}+(\widetilde{\lambda}_h,\textup{div}\,y_h)_{\Omega}- (y_h\cdot n,u_D^h)_{\Gamma_D}=0\quad(\ell\to \infty)\,,
 \end{align*}
 \textit{i.e.}, $\smash{r_h^{k_\ell}}\rightharpoonup 0$ in $\mathcal{R}T^0_N(\mathcal{T}_h)$ $(\ell\to \infty)$, and, thus, by the finite-dimensionality of $\mathcal{R}T^0_N(\mathcal{T}_h)$,  $\smash{r_h^{k_\ell}}\to 0$ in $\mathcal{R}T^0_N(\mathcal{T}_h)$ $(\ell\to \infty)$, implying that $\smash{r_h^{k_\ell}}\to 0$ in $(L^2(\Omega))^d$ $(\ell\to \infty)$.~As~this~\mbox{argument}~\mbox{remains} valid for each subsequence of $(r_h^k)_{k\in \mathbb{N}}\subseteq \mathcal{R}T^0_N(\mathcal{T}_h)$, the standard convergence principle yields that $r_h^k\to 0$ in $(L^2(\Omega))^d$ $(k\to \infty)$. In particular, there exists $k^*\in \mathbb{N}$  such that $\|r_h^{k^*}\|_{2,\Omega}\leq \varepsilon^h_{\textit{stop}}$.
 \end{proof} 

    \subsection{Experimental setup}\label{subsec:exp_setup}\enlargethispage{5mm}

    \hspace{5mm}For \hspace{-0.1mm}our \hspace{-0.1mm}numerical \hspace{-0.1mm}experiments, \hspace{-0.1mm}we \hspace{-0.1mm}consider \hspace{-0.1mm}the \hspace{-0.1mm}manufactured \hspace{-0.1mm}solution \hspace{-0.1mm}in  \hspace{-0.1mm}\cite[Expl.~\hspace{-0.1mm}2.3.2,~\hspace{-0.1mm}p.~\hspace{-0.1mm}122]{GLT1981}.
	More precisely, let $\Omega\hspace{-0.1em}\coloneqq \hspace{-0.1em}B_r^2(0)\hspace{-0.1em}\coloneqq \hspace{-0.1em}\{ x\hspace{-0.1em}\in\hspace{-0.1em} \mathbb{R}^2\mid \vert x\vert \hspace{-0.1em}<\hspace{-0.1em}r\}$, $r\hspace{-0.1em}>\hspace{-0.1em}0$,  $	\Gamma_D\hspace{-0.1em}\coloneqq\hspace{-0.1em}\partial\Omega$
	(\textit{i.e.}, $\Gamma_N\hspace{-0.1em}\coloneqq\hspace{-0.1em}\emptyset$),~$ {f\hspace{-0.1em}\equiv\hspace{-0.1em} C\hspace{-0.1em}\in\hspace{-0.1em} L^1(\Omega)}$, $C>0$, $u_D\equiv 0\in W^{1,\infty}(\Gamma_D)$, and $\zeta\equiv 1\in L^{\infty}(\Omega)$. Then, the primal solution $u\in K $ and the dual solution $z\in K^*$, for every $x\in \Omega$, are defined by
	\begin{align*}
		u(x)&\coloneqq \left.\begin{cases}
            \frac{C}{4}(r^2-\vert x\vert^2)&\text{ if }C\leq \frac{2}{r}\,,\\
        \left.\begin{cases}
		    r-\vert x\vert&\text{ if } \frac{2}{C}\leq \vert x\vert \leq r\,,\\
            -\frac{C}{4}\vert x\vert^2+r-\frac{1}{C}&\text{ if } 0\leq \vert x\vert \leq \frac{2}{C}\,,
		\end{cases}\right\}&\text{ if } C\ge \frac{2}{r}\,,
		\end{cases}\right\}\,,\\
         z(x)&\coloneqq -\tfrac{C}{2}x\,.
	\end{align*}

    In what follows, we set $r\hspace{-0.1em}=\hspace{-0.1em}1$ (\textit{i.e.}, $\Omega\hspace{-0.1em}=\hspace{-0.1em}B_1^2(0)$). Then,
	using the~\texttt{MATLAB}~(version~R2024a,~\textit{cf}.~\cite{matlab}) library~\texttt{DistMesh} (version 1.1, \textit{cf}.\  \cite{distmesh}), we generate approximate triangulations
    $\mathcal
	T_{h_i}$, $i=0,\ldots,6$, where $h_i\hspace{-0.1em}\approx\hspace{-0.1em} 0.36\hspace{-0.1em}\times \hspace{-0.1em}\frac{1}{2^{i}}$ for all $i\hspace{-0.1em}=\hspace{-0.1em}0,\ldots,6$, such that $\Omega_{h_i}\hspace{-0.1em}\subseteq\hspace{-0.1em} \Omega$, where $\Omega_{h_i}\hspace{-0.1em}\coloneqq \hspace{-0.1em}\textup{int}(\cup\mathcal{T}_{h_i})$~for~all~$i\hspace{-0.1em}=\hspace{-0.1em}0,\ldots,6$. As approximations of the Dirichlet boundary data serve $u_D^{h_i}\coloneqq \pi_{h_i} u\in \mathcal{L}^0(\mathcal{S}_{h_i}^{\partial\Omega_{h_i}})$,  $i=0,\ldots,6$.
	
	For this series of triangulations $\mathcal{T}_{h_i}$, $i=0,\ldots,6$, we apply 
    the semi-implicit discretized~\mbox{$L^2$-grad-} ient flow (\textit{cf}.\ Algorithm \ref{algorithm}) with step-size $\tau=1.0$, stopping parameter $\smash{\varepsilon^h_{stop}}=1.0\times 10^{-4}$, and initial iterate $(z_{h_i}^0,\smash{\overline{\lambda}}^0_{h_i})^\top\hspace{-0.1em}\in\hspace{-0.1em} \mathcal{R}T^0_N(\mathcal{T}_{h_i})\times \mathcal{L}^0(\mathcal{T}_{h_i})$ such that for every ${(y_{h_i},\overline{\eta}_{h_i})^\top\hspace{-0.1em}\in \hspace{-0.1em}\mathcal{R}T^0_N(\mathcal{T}_{h_i})\times \mathcal{L}^0(\mathcal{T}_{h_i})}$, it holds that
    \begin{align*}
        (\Pi_{h_i}z_{h_i}^0,\Pi_{h_i} y_{h_i})_{\Omega_{h_i}}-(\smash{\overline{\lambda}}^0_{h_i},\textup{div}\,y_{h_i})_{\Omega_{h_i}}&=(y_{h_i}\cdot n,u_D^{h_i})_{\partial\Omega_{h_i}}\,,\\
        (\textup{div}\,z_{h_i}^0,\overline{\eta}_{h_i})_{\Omega_{h_i}}&=-(f_{h_i},\overline{\eta}_{h_i})_{\Omega_{h_i}}\,,
    \end{align*}
	to approximate the discrete dual solution $z_{h_i}^{rt}\in  K_{h_i}^{rt,*}$, $i=0,\ldots,6$, the discrete Lagrange~\mbox{multiplier} $\smash{\overline{\lambda}}_{h_i}^0\in \mathcal{L}^0(\mathcal{T}_{h_i})$, $i=0,\ldots,6$, and, subsequently, using inverse generalized Marini~formula~(\textit{cf}.~\eqref{eq:gen_marini}), the discrete primal solution $ u_{h_i}^{cr}\in K_{h_i}^{cr}$, $i=0,\ldots,6$.

    For determining the convergence rates,  
	the experimental order of convergence~(EOC),~\textit{i.e.},
	\begin{align*}
		\texttt{EOC}_i(e_i)\coloneqq \frac{\log(e_i)-\log(e_{i-1})}{\log(h_i)-\log(h_{i-1})}\,,\quad i=1,\ldots,6\,,
	\end{align*}
	where, for every $i=0,\ldots, 6$, we denote by $e_i$ a generic error quantity.
 
 	\subsection{Numerical experiments concerning the a priori error analysis}\label{subsec:num_a_priori}
	
	\hspace{5.5mm}In this subsection, we review the theoretical findings of Section \ref{sec:apriori}.\enlargethispage{5mm}

    More precisely, given the experimental setup of Subsection \ref{subsec:exp_setup}, we compute the error quantities
	\begin{align}\label{eq:error_quantities.0}
		\left.\begin{aligned}
			e_i^{\textup{tot},h_i}&\coloneqq 	\smash{\rho_{\textup{tot},h_i}^2(\Pi_{h_i}^{cr}u,\Pi_{h_i}^{rt} z)}\,,\\
			e_i^{\textup{gap},h_i}&\coloneqq 	\smash{\eta_{\textup{gap},h_i}^2(\Pi_{h_i}^{cr}u,\Pi_{h_i}^{rt} z)}\,,
		\end{aligned}\quad\right\}\quad i=0,\ldots,6\,.
	\end{align}
	Inasmuch as $z\in (W^{1,2}(\Omega))^2$ and $\zeta\in W^{1,2}(\Omega)$, Theorem \ref{thm:apriori_identity}(ii) predicts the error decay rate $\mathcal{O}(h_i^2)= \mathcal{O}(N_i^{\textcolor{black}{-1}})$, where $N_i\coloneqq\textup{dim}(\mathcal{R}T^0(\mathcal{T}_{h_i}))+\textup{dim}(\mathcal{L}^0(\mathcal{T}_{h_i}))$, ${i\in \mathbb{N}}$, for the discrete primal-dual total errors (\textit{cf}.\ \eqref{eq:discrete_primal_dual_error}), which are equal to the discrete primal-dual gap estimators~(\textit{cf}.~\eqref{eq:discrete_primal_dual_gap_estimator}), \textit{i.e.}, we expect~(\textit{cf}.~\mbox{Theorem}~\ref{thm:apriori_identity}(i))\enlargethispage{1mm}
	\begin{align*}
 \texttt{EOC}_i(e_i^{\textup{gap}})=\texttt{EOC}_i(e_i^{\textup{tot}})=2\,.
	\end{align*}

	 In Figure \ref{fig:apriori_exp1}, 
  for $C\in \{2.5,5.0,7.5,10.0\}$, we report the expected optimal convergence rate of about $\texttt{EOC}_i(e_i^{\textup{tot}})\approx\texttt{EOC}_i(e_i^{\textup{gap}})\approx  2$, $i=1,\ldots, 6$,
	\textit{i.e.}, an error decay of order 
	$\mathcal{O}(h_i^2)=  \mathcal{O}(N_i^{\textcolor{black}{-1}})$, $i=1,\ldots, 6$, 
 which confirms the optimality of the error decay rate derived in Theorem \ref{thm:apriori_identity}(ii).
 Moreover,~we~observe~that the \textit{a priori} error identity in Theorem~\ref{thm:apriori_identity}(i)~is~approximately~satisfied. 
 Note that, in the case $C\leq 2$, the dual solution is an affine polynomial, so that the latter (up to machine precision) coincides with the discrete dual solution.
 In addition, in the case $C\leq 2$, the primal solution is a quadratic polynomial, so that, by \eqref{eq:grad_preservation}, (up to machine precision) we have that $\nabla_h\Pi_h^{cr}u=\nabla_h u_h^{cr}$ a.e.\ in $\Omega$. For this reason, we did not add error plots for the case $C\leq 2$.

	 \begin{figure}[H]
	 	\hspace*{-2mm}
	 	\centering
	 	\includegraphics[width=15cm]{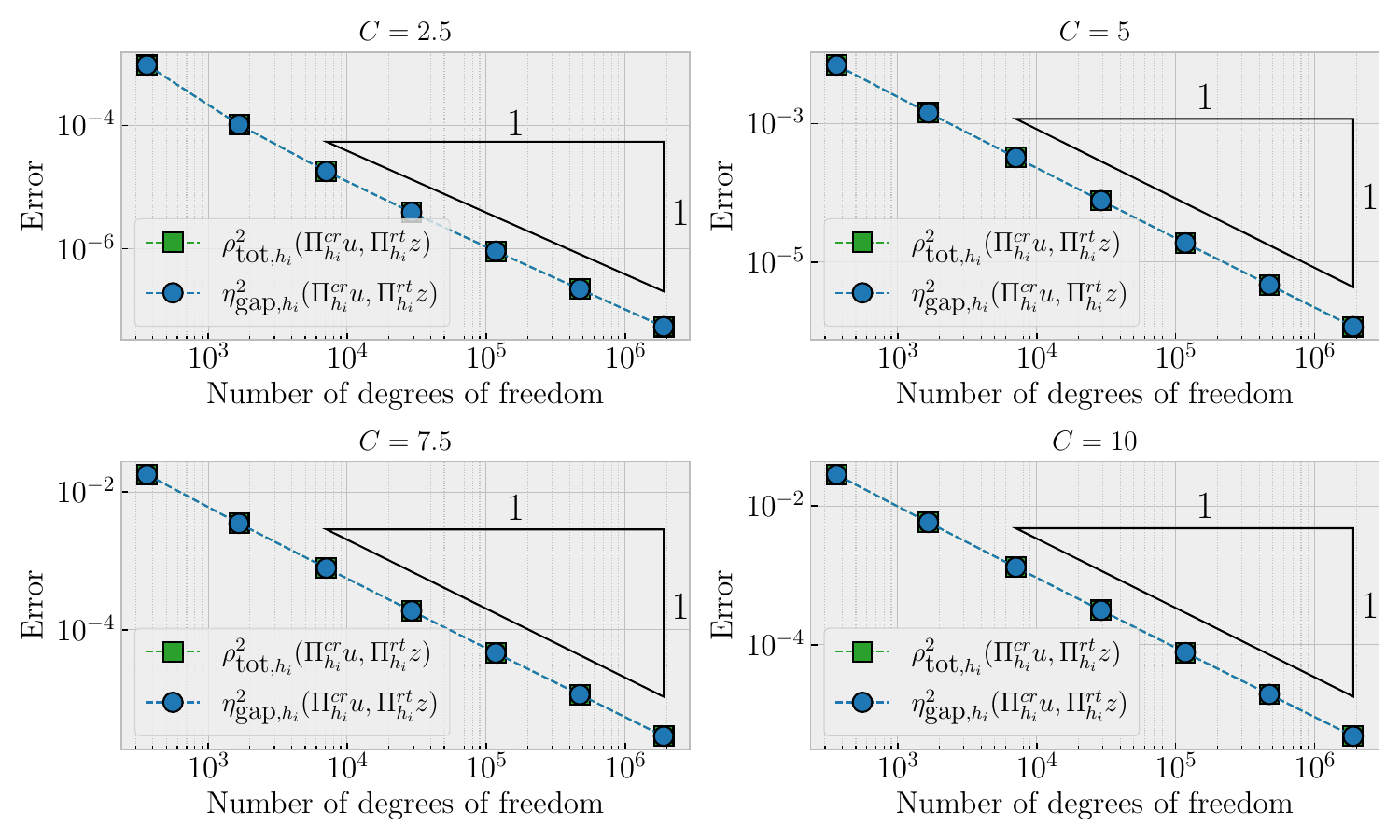}
	 	\caption{Logarithmic plots of the experimental convergence rates of the error~\mbox{quantities}~\eqref{eq:error_quantities.0}. For forcing terms  $C\in \{2.5,5.0,7.5,10.0\}$, we report the expected quadratic error decay rate, \textit{i.e.}, $\texttt{EOC}_i(e_i^{\textup{tot}})\approx\texttt{EOC}_i(e_i^{\textup{gap}})\approx  2$,~${i=1,\ldots, 6}$. In addition, we report that the discrete primal-dual error (\textit{cf}.\ \eqref{eq:discrete_primal_dual_error}) approximately coincides with the discrete primal-dual gap estimator (\textit{cf}.\ \eqref{eq:discrete_primal_dual_gap_estimator}).}
	 	\label{fig:apriori_exp1}
	 \end{figure}

    \begin{figure}[H]
	 	\centering
	 	\includegraphics[width=7.5cm]{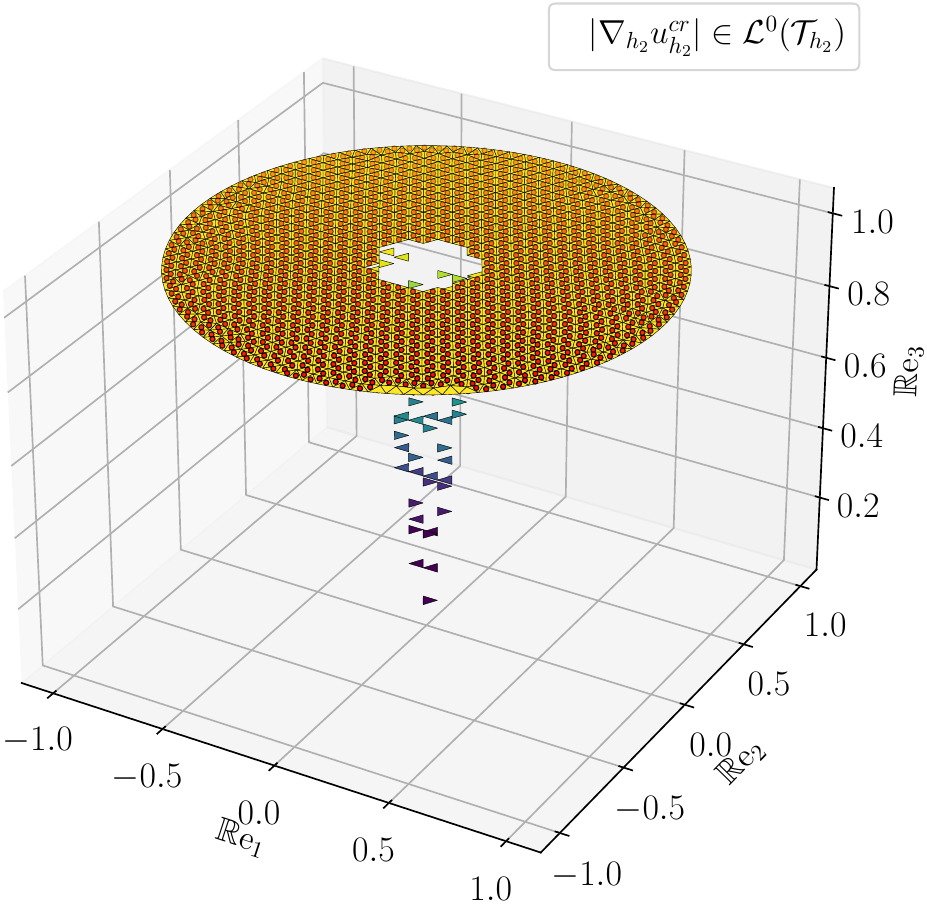}\includegraphics[width=7.5cm]{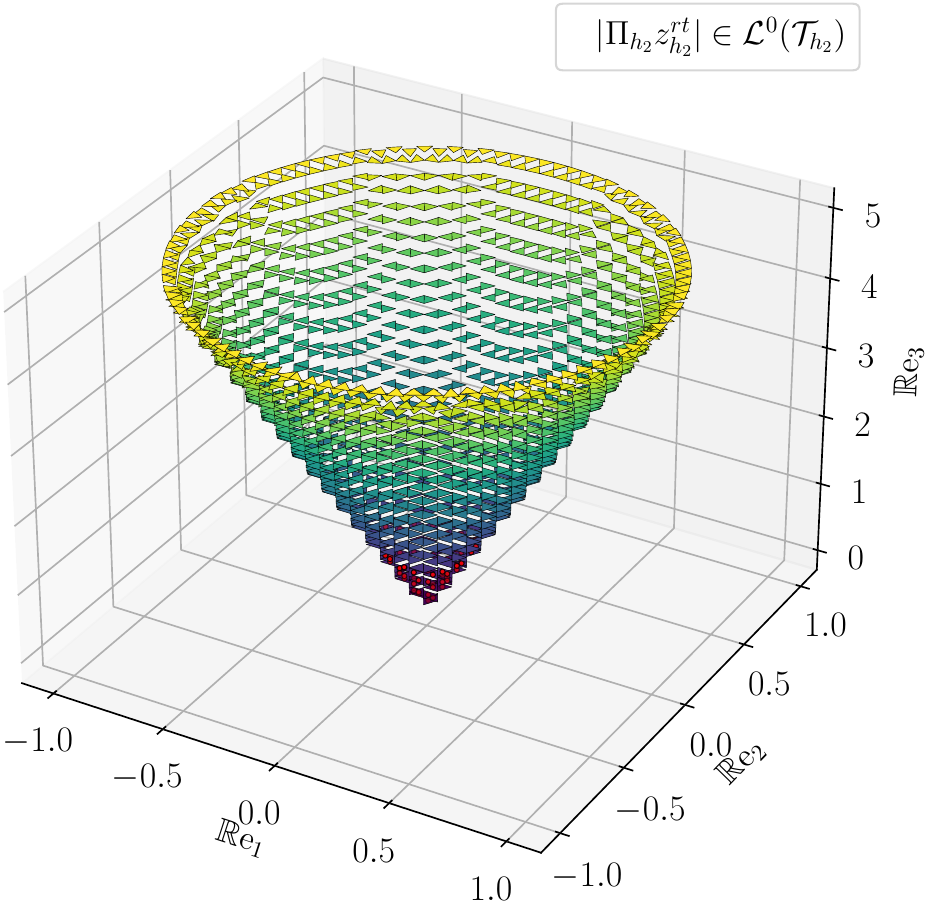}
	 	\caption{\textit{left}:  plot of  $\vert \nabla_{h_2}u_{h_2}^{cr}\vert\in \mathcal{L}^0(\mathcal{T}_{h_2})$, where red dots mark $T\in \mathcal{T}_{h_2}$ with $\vert \nabla_{h_2}u_{h_2}^{cr}\vert=1$~in~$T$; \textit{right}:
        plot of $\vert \Pi_{h_2} z_{h_2}^{rt}\vert\hspace{-0.1em}\in\hspace{-0.1em} \mathcal{L}^0(\mathcal{T}_{h_2})$, where red dots mark  $T\hspace{-0.1em}\in \hspace{-0.1em}\mathcal{T}_{h_2}$~with~${\vert \Pi_{h_2} z_{h_2}^{rt}\vert\hspace{-0.1em}<\hspace{-0.1em}1}$~in~$T$.~In~summary, we report that
        $\{\vert \nabla_{h_2} u_{h_2}^{cr}\vert=1\}=\{\vert \Pi_{h_2} z_{h_2}^{rt}\vert\ge 1\}$ and $\{\vert \nabla_{h_2} u_{h_2}^{cr}\vert<1\}=\{\vert \Pi_{h_2} z_{h_2}^{rt}\vert< 1\}$~as~predicted by the discrete convex optimality relation \eqref{eq:discrete_optimality}.}
	 \end{figure}

    \subsection{Numerical experiments concerning the a posteriori error analysis}\label{subsec:num_a_posteriori}
	
	\hspace{5.5mm}In this subsection, we review the theoretical findings of Section \ref{sec:aposteriori}.\enlargethispage{6mm}

    More precisely, given the experimental setup of Subsection \ref{subsec:exp_setup}, for every $i=1,\ldots,6$, we define the admissible approximation
     \begin{align*}
        \overline{u}_{h_i}^{cr}\coloneqq \frac{\Pi_{h_i}^{w1}u_{h_i}^{cr}}{\max\{1,\|\nabla\Pi_{h_i}^{w1}u_{h_i}^{cr}\|_{\infty,\Omega}\}}\in K\,,
    \end{align*}
    where $\Pi_{h_i}^{w1}\colon\mathcal{S}^{1,cr}_D(\mathcal{T}_{h_i})\to  W^{1,\infty}_D(\Omega)$ is a suitable quasi-interpolation operator, \textit{e.g.},
    \begin{align*}
        \Pi_{h_i}^{w1}\in \Big\{\Pi_{h_i}^{av,1},\Pi_{h_i}^{av,2},\Pi_{h_i}^{sz,1},\Pi_{h_i}^{sz,2},\Pi_{h_i}^{L^2,1},\Pi_{h_i}^{L^2,2}\Big\}\,,
    \end{align*}
     where, setting $\mathcal{S}^k_D(\mathcal{T}_{h_i})\coloneqq \mathcal{L}^k(\mathcal{T}_{h_i})\cap W^{1,\infty}_D(\Omega)$, for $k\in \{1,2\}$, we denote by
    \begin{itemize}[noitemsep,topsep=2pt,leftmargin=!,labelwidth=\widthof{(iii)},font=\itshape]
        \item[(i)] \hspace{-0.15em}$\Pi_{h_i}^{av,k}\hspace{-0.15em}\colon \hspace{-0.175em} \mathcal{S}^{1,cr}_D(\mathcal{T}_{h_i}) \hspace{-0.175em}\to \hspace{-0.175em} \mathcal{S}^k_D(\mathcal{T}_{h_i})$ \hspace{-0.2mm}the \hspace{-0.2mm}\textit{averaging \hspace{-0.2mm}quasi-interpolation \hspace{-0.2mm}operator}~\hspace{-0.2mm}from~\hspace{-0.2mm}\cite[\mbox{Subsec.}~\hspace{-0.2mm}22.4.1]{EG21I};
        \item[(ii)] \hspace{-0.15em}$\Pi_{h_i}^{sz,k}\hspace{-0.175em}\colon \hspace{-0.1em}\mathcal{S}^{1,cr}_D(\mathcal{T}_{h_i}) \hspace{-0.175em}\to \hspace{-0.175em}\mathcal{S}^k_D(\mathcal{T}_{h_i})$ \hspace{-0.1mm}the \hspace{-0.1mm}\textit{Scott--Zhang \hspace{-0.1mm}quasi-interpolation \hspace{-0.1mm}operator} \hspace{-0.1mm}from \hspace{-0.1mm}\cite[Appx.~\hspace{-0.1mm}A.3]{EG21I};
        \item[(iii)] \hspace{-0.15em}$\Pi_{h_i}^{L^2,k}\hspace{-0.15em}\colon \hspace{-0.175em}\mathcal{S}^{1,cr}_D(\mathcal{T}_{h_i}) \hspace{-0.175em}\to \hspace{-0.175em}\mathcal{S}^k_D(\mathcal{T}_{h_i})$ the \textit{(global) $L^2$-projection operator}  from \cite[Subsec.\ 22.5]{EG21I}.
    \end{itemize}
    
    Then, we compute the error quantities
	\begin{align}\label{eq:error_quantities}
		\left.\begin{aligned}
			e_i^{\textup{tot}}&\coloneqq 	\smash{\rho_{\textup{tot}}^2(\overline{u}_{h_i}^{cr},z_{h_i}^{rt})}\,,\\
			e_i^{\textup{gap}}&\coloneqq 	\smash{\eta_{\textup{gap}}^2(\overline{u}_{h_i}^{cr},z_{h_i}^{rt})}\,,
		\end{aligned}\quad\right\}\quad i=0,\ldots,6\,.
	\end{align}  

    In Figure \ref{fig:sub_optimality}, 
    for $C=10.0$, we report a reduced (compared to Subsection~\ref{subsec:num_a_priori}) convergence rate of about $\texttt{EOC}_i(e_i^{\textup{tot}})\approx\texttt{EOC}_i(e_i^{\textup{gap}})\approx  1$, $i=1,\ldots, 6$,
	\textit{i.e.}, an error decay of order 
	$\mathcal{O}(h_i)=  \mathcal{O}(N_i^{\textcolor{black}{-\frac{1}{2}}})$, ${i=1,\ldots, 6}$. Moreover, we observe that the \textit{a posteriori} error identity in Theorem~\ref{thm:prager_synge_identity} is approximately satisfied.

    \begin{figure}[H]
        \hspace*{-2mm}
        \centering
        \includegraphics[width=15cm]{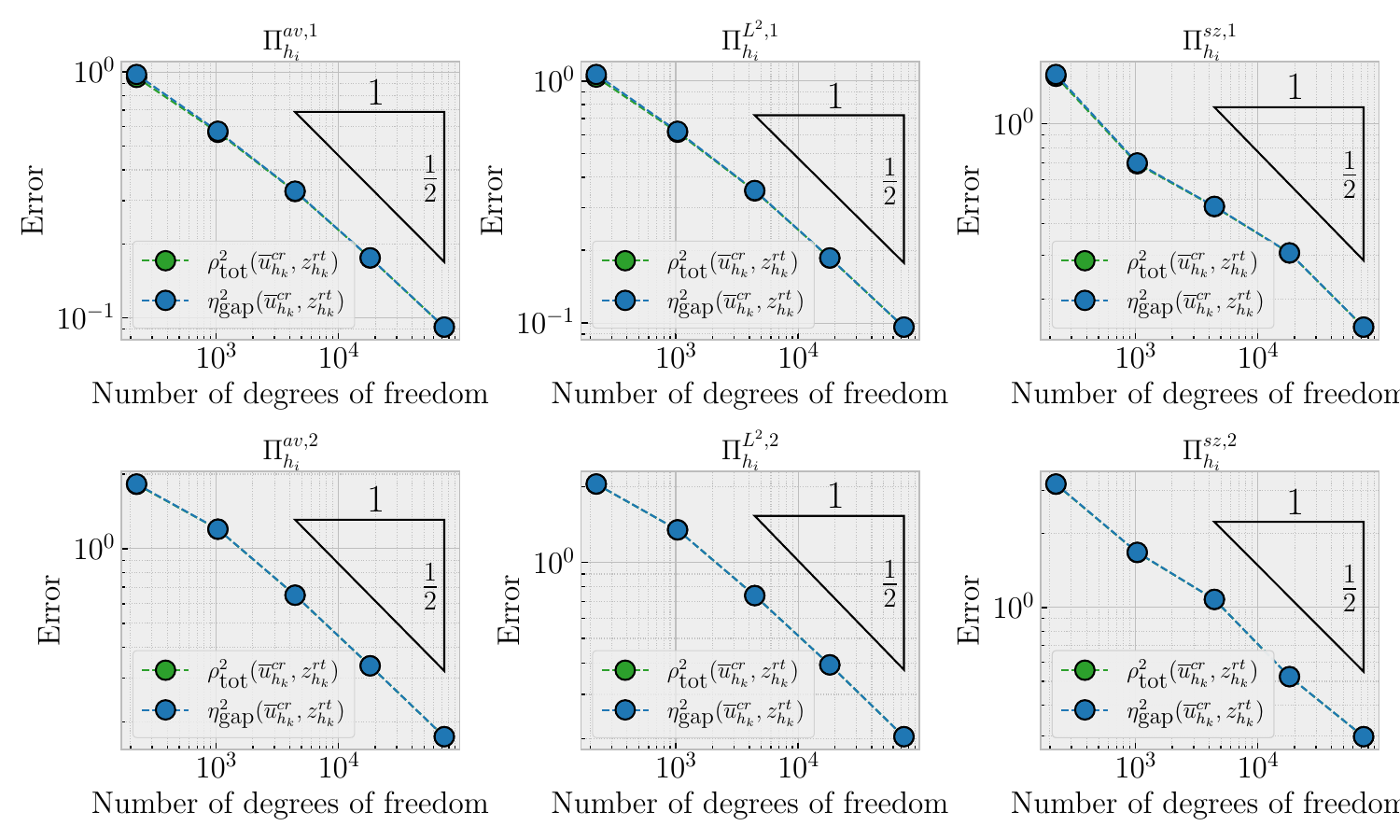}
        \caption{Logarithmic plots of the experimental convergence rates of the error~\mbox{quantities}~\eqref{eq:error_quantities}. For   $C=10.0$ and $\Pi_{h_i}^{w1}\in \{\Pi_{h_i}^{av,1},\Pi_{h_i}^{av,2},\Pi_{h_i}^{sz,1},\Pi_{h_i}^{sz,2},\Pi_{h_i}^{L^2,1},\Pi_{h_i}^{L^2,2}\}$, we report a reduced linear error decay rate, \textit{i.e.}, $\texttt{EOC}_i(e_i^{\textup{tot}})\approx\texttt{EOC}_i(e_i^{\textup{gap}})\approx  1$,~${i=1,\ldots, 6}$. In addition, we~report~that~the~primal-dual error (\textit{cf}.\ \eqref{def:primal_dual_total_error}) approximately coincides with the primal-dual gap estimator~(\textit{cf}.~\eqref{eq:primal-dual.1}).}
        \label{fig:sub_optimality}
    \end{figure}
 
    \section{Conclusion and Outlook}

    \hspace{5mm}A (Fenchel) duality theory for variational problems with gradient constraints was developed. On the basis of this (Fenchel) duality theory, an \textit{a posteriori} error identity was derived, which can be seen as a generalization of the Prager--{Synge} identity (\textit{cf}.\ \cite{PraSyn47})~for~the~Poisson~problem. 
    To make the \textit{a posteriori} error identity numerically practicable, it is necessary to approximate in a computationally inexpensive way the primal problem and the dual problem~at~the~same~time. To this end, exploiting orthogonality relations between the Crouzeix--Raviart element and the Raviart--Thomas element, 
     all (Fenchel) duality relations from the continuous level are transferred to a discrete level and reconstruction of a primal solution from a given dual solution is derived, using the so-called inverse generalized Marini formula. The inverse generalized Marini formula enabled us to approximate the primal problem and the dual problem at the same time using only the Raviart--Thomas element. In addition, the discrete (Fenchel) duality theory allowed us to derive an \textit{a posteriori} error identity also on the discrete level, which, eventually, turned out to be an \textit{a priori} error identity and enabled us to derive explicit error decay rates depending on the regularity of the dual solution and the obstacle function.
     Numerical experiments were carried out that confirm the optimality of the derived error decay rates and the validity of both the \textit{a priori} error identity and the \textit{a posteriori} error identity.
     Still open is the question of the accurate post-processing of the discrete primal solution to obtain an admissible approximation of the primal solution. Simple node-averaging and, subsequently, global scaling with the gradient length does not approximate the gradient in $(L^\infty(\Omega))^d$ with the desired accuracy. Instead a quasi-interpolation operator, which either preserves the gradient length or at least approximates the latter~with~the~needed~accuracy, should be employed. In the end, this would allow us to use the local refinement indicators~induced\linebreak by the primal-dual gap estimator for adaptive mesh-refinement. This will be the subject of future research.

	{\setlength{\bibsep}{0pt plus 0.0ex}\small
		

\begin{thebibliography}{10}

\bibitem{ArbogastChen95}
\bgroup\scshape{}T.~Arbogast\egroup{} and \bgroup\scshape{}Z.~Chen\egroup{}, On the implementation of mixed methods as nonconforming methods for second-order elliptic problems,  \emph{Math. Comput.} \textbf{64} no.~211 (1995), 943--972 (English). \doi{10.2307/2153478}.

\bibitem{ArnoldBrezzi85}
\bgroup\scshape{}D.~N. Arnold\egroup{} and \bgroup\scshape{}F.~Brezzi\egroup{}, Mixed and nonconforming finite element methods: {Implementation}, postprocessing and error estimates,  \emph{RAIRO, Mod{\'e}lisation Math. Anal. Num{\'e}r.} \textbf{19} (1985), 7--32 (English). \doi{10.1051/m2an/1985190100071}.

\bibitem{Bartels16}
\bgroup\scshape{}S.~Bartels\egroup{}, \emph{Numerical approximation of partial differential equations}, \emph{Texts Appl. Math.} \textbf{64}, Cham: Springer, 2016 (English). \doi{10.1007/978-3-319-32354-1}.

\bibitem{Bar21}
\bgroup\scshape{}S.~Bartels\egroup{}, Nonconforming discretizations of convex minimization problems and precise relations to mixed methods,  \emph{Comput. Math. Appl.} \textbf{93} (2021), 214--229 (English). \doi{10.1016/j.camwa.2021.04.014}.

\bibitem{BarGudKal24}
\bgroup\scshape{}S.~Bartels\egroup{}, \bgroup\scshape{}T.~Gudi\egroup{}, and \bgroup\scshape{}A.~Kaltenbach\egroup{}, $\textit{A priori}$ and $\textit{a posteriori}$ error identities for the scalar Signorini problem, 2024.  
\doi{10.48550/arXiv.2407.10912}

\bibitem{BartelsKaltenbachOverview24}
\bgroup\scshape{}S.~Bartels\egroup{} and \bgroup\scshape{}A.~Kaltenbach\egroup{}, Chapter seven - exact a posteriori error control for variational problems via convex duality and explicit flux reconstruction,  in \emph{Error Control, Adaptive Discretizations, and Applications, Part 1} (\bgroup\scshape{}F.~Chouly\egroup{}, \bgroup\scshape{}S.~P. Bordas\egroup{}, \bgroup\scshape{}R.~Becker\egroup{}, and \bgroup\scshape{}P.~Omnes\egroup{}, eds.), \emph{Advances in Applied Mechanics} \textbf{58}, Elsevier, 2024, pp.~281--361. \doi{10.1016/bs.aams.2024.04.001}.

\bibitem{bildhauer2009elastic}
\bgroup\scshape{}M.~Bildhauer\egroup{}, \bgroup\scshape{}M.~Fuchs\egroup{}, and \bgroup\scshape{}S.~Repin\egroup{}, The elastic--plastic torsion problem: A posteriori error estimates for approximate solutions,  \emph{Numerical Functional Analysis and Optimization} \textbf{30} no.~7-8 (2009), 653--664.

\bibitem{MR4593745}
\bgroup\scshape{}F.~Chouly\egroup{}, \bgroup\scshape{}T.~Gustafsson\egroup{}, and \bgroup\scshape{}P.~Hild\egroup{}, A {N}itsche method for the elastoplastic torsion problem,  \emph{ESAIM Math. Model. Numer. Anal.} \textbf{57} no.~3 (2023), 1731--1746.  \doi{10.1051/m2an/2023034}.

\bibitem{unknown}
\bgroup\scshape{}F.~Chouly\egroup{}, \bgroup\scshape{}T.~Gustafsson\egroup{}, and \bgroup\scshape{}P.~Hild\egroup{}, Finite element approximation of penalized elastoplastic torsion problem with nonconstant source term, 08 2024.

\bibitem{MR4432040}
\bgroup\scshape{}F.~Chouly\egroup{} and \bgroup\scshape{}P.~Hild\egroup{}, On a finite element approximation for the elastoplastic torsion problem,  \emph{Appl. Math. Lett.} \textbf{132} (2022), Paper No. 108191, 6.  \doi{10.1016/j.aml.2022.108191}.

\bibitem{CR73}
\bgroup\scshape{}M.~Crouzeix\egroup{} and \bgroup\scshape{}P.-A. Raviart\egroup{}, Conforming and nonconforming finite element methods for solving the stationary {Stokes} equations. {I},  \emph{Rev. Franc. Automat. Inform. Rech. Operat., R} \textbf{7} no.~3 (1974), 33--76 (English). \doi{10.1051/m2an/197307R300331}.

\bibitem{MR0464857}
\bgroup\scshape{}G.~Duvaut\egroup{} and \bgroup\scshape{}J.-L. Lions\egroup{}, \emph{Les in\'equations en m\'ecanique et en physique}, \emph{Travaux et Recherches Math\'ematiques} \textbf{No. 21}, Dunod, Paris, 1972.  

\bibitem{ET99}
\bgroup\scshape{}I.~Ekeland\egroup{} and \bgroup\scshape{}R.~T{\'e}mam\egroup{}, \emph{Convex analysis and variational problems.}, unabridged, corrected republication of the 1976 {English} original ed., \emph{Class. Appl. Math.} \textbf{28}, Philadelphia, PA: Society for Industrial {and} Applied Mathematics, 1999 (English).

\bibitem{EG21I}
\bgroup\scshape{}A.~Ern\egroup{} and \bgroup\scshape{}J.-L. Guermond\egroup{}, \emph{Finite elements {I}. {Approximation} and interpolation}, \emph{Texts Appl. Math.} \textbf{72}, Cham: Springer, 2020 (English). \doi{10.1007/978-3-030-56341-7}.

\bibitem{EG21II}
\bgroup\scshape{}A.~Ern\egroup{} and \bgroup\scshape{}J.-L. Guermond\egroup{}, \emph{Finite elements {II}. {Galerkin} approximation, elliptic and mixed {PDEs}}, \emph{Texts Appl. Math.} \textbf{73}, Cham: Springer, 2021 (English). \doi{10.1007/978-3-030-56923-5}.

\bibitem{Falk1977}
\bgroup\scshape{}R.~Falk\egroup{} and \bgroup\scshape{}B.~Mercier\egroup{}, Error estimates for elasto-plastic problems,  \emph{RAIRO. Analyse num{\'e}rique} \textbf{11} no.~2 (1977), 135--144.    \doi{http://eudml.org/doc/193292}.

\bibitem{MR0597520}
\bgroup\scshape{}R.~Glowinski\egroup{}, \emph{Lectures on numerical methods for nonlinear variational problems}, \emph{Tata Institute of Fundamental Research Lectures on Mathematics and Physics} \textbf{65}, Tata Institute of Fundamental Research, Bombay; Springer-Verlag, Berlin-New York, 1980, Notes by M. G. Vijayasundaram and M. Adimurthi. 

\bibitem{GLT1981}
\bgroup\scshape{}R.~Glowinski\egroup{}, \bgroup\scshape{}J.-L. Lions\egroup{}, and \bgroup\scshape{}R.~Tremolieres\egroup{}, \emph{Numerical analysis of variational inequalities. {Transl}. and rev. ed}, \emph{Stud. Math. Appl.} \textbf{8}, Elsevier, Amsterdam, 1981 (English).

\bibitem{Hun07}
\bgroup\scshape{}J.~D. Hunter\egroup{}, Matplotlib: A 2d graphics environment,  \emph{Computing in Science \& Engineering} \textbf{9} no.~3 (2007), 90--95. \doi{10.1109/MCSE.2007.55}.

\bibitem{MR1989924}
\bgroup\scshape{}G.~Idone\egroup{}, \bgroup\scshape{}A.~Maugeri\egroup{}, and \bgroup\scshape{}C.~Vitanza\egroup{}, Variational inequalities and the elastic-plastic torsion problem,  \emph{J. Optim. Theory Appl.} \textbf{117} no.~3 (2003), 489--501.    \doi{10.1023/A:1023941520452}.

\bibitem{matlab}
\bgroup\scshape{}T.~M. Inc.\egroup{}, Matlab version: R2024a, 2024. Available at \url{https://www.mathworks.com}.

\bibitem{MR0567696}
\bgroup\scshape{}D.~Kinderlehrer\egroup{} and \bgroup\scshape{}G.~Stampacchia\egroup{}, \emph{An introduction to variational inequalities and their applications}, \emph{Pure and Applied Mathematics} \textbf{88}, Academic Press, Inc. [Harcourt Brace Jovanovich, Publishers], New York-London, 1980. 

\bibitem{LW10}
\bgroup\scshape{}A.~Logg\egroup{} and \bgroup\scshape{}G.~N. Wells\egroup{}, D{OLFIN}: automated finite element computing,  \emph{ACM Trans. Math. Software} \textbf{37} no.~2 (2010), Art. 20, 28.  \doi{10.1145/1731022.1731030}.

\bibitem{Mar85}
\bgroup\scshape{}L.~D. Marini\egroup{}, An inexpensive method for the evaluation of the solution of the lowest order {R}aviart-{T}homas mixed method,  \emph{SIAM J. Numer. Anal.} \textbf{22} no.~3 (1985), 493--496.   \doi{10.1137/0722029}.

\bibitem{distmesh}
\bgroup\scshape{}P.-O. Persson\egroup{} and \bgroup\scshape{}G.~Strang\egroup{}, A {S}imple {M}esh {G}enerator in {MATLAB},  \emph{SIAM Review} \textbf{46} no.~2 (2004), 329--345. \doi{10.1137/S0036144503429121}.

\bibitem{PraSyn47}
\bgroup\scshape{}W.~Prager\egroup{} and \bgroup\scshape{}J.~L. Synge\egroup{}, Approximations in elasticity based on the concept of function space,  \emph{Quart. Appl. Math.} \textbf{5} (1947), 241--269. 
\doi{10.1090/qam/25902}.

\bibitem{RT77}
\bgroup\scshape{}P.~A. Raviart\egroup{} and \bgroup\scshape{}J.~M. Thomas\egroup{}, A mixed finite element method for 2nd order elliptic problems, Math. {Aspects} {Finite} {Elem}. {Meth}., {Proc}. {Conf}. {Rome} 1975, {Lect}. {Notes} {Math}. 606, 292-315 (1977)., 1977.

\bibitem{Rockafellar68}
\bgroup\scshape{}R.~T. Rockafellar\egroup{}, Integrals which are convex functionals,  \emph{Pac. J. Math.} \textbf{24} (1968), 525--539 (English). \doi{10.2140/pjm.1968.24.525}.

\bibitem{Rockafellar71}
\bgroup\scshape{}R.~T. Rockafellar\egroup{}, Integrals which are convex functionals. {II},  \emph{Pac. J. Math.} \textbf{39} (1971), 439--469 (English). \doi{10.2140/pjm.1971.39.439}.

\bibitem{MR1284980}
\bgroup\scshape{}S.~E. Shreve\egroup{} and \bgroup\scshape{}H.~M. Soner\egroup{}, Optimal investment and consumption with transaction costs,  \emph{Ann. Appl. Probab.} \textbf{4} no.~3 (1994), 609--692.

\bibitem{soner1991free}
\bgroup\scshape{}H.~M. Soner\egroup{}, \bgroup\scshape{}S.~E. Shreve\egroup{}, and \bgroup\scshape{}N.~El~Karoui\egroup{}, A free boundary problem related to singular stochastic control: the parabolic case,  \emph{Communications in partial differential equations} \textbf{16} no.~2-3 (1991), 373--424.

\bibitem{Toland20}
\bgroup\scshape{}J.~Toland\egroup{}, \emph{The dual of {{\(L _\infty (X,\mathcal{L}, \lambda)\)}}, finitely additive measures and weak convergence. {A} primer}, \emph{SpringerBriefs Math.}, Cham: Springer, 2020 (English). \doi{10.1007/978-3-030-34732-1}.

\end{thebibliography}
        \providecommand{\bysame}{\leavevmode\hbox to3em{\hrulefill}\thinspace}
\providecommand{\noopsort}[1]{}
\providecommand{\mr}[1]{\href{http://www.ams.org/mathscinet-getitem?mr=#1}{MR~#1}}
\providecommand{\zbl}[1]{\href{http://www.zentralblatt-math.org/zmath/en/search/?q=an:#1}{Zbl~#1}}
\providecommand{\jfm}[1]{\href{http://www.emis.de/cgi-bin/JFM-item?#1}{JFM~#1}}
\providecommand{\arxiv}[1]{\href{http://www.arxiv.org/abs/#1}{arXiv~#1}}
\providecommand{\doi}[1]{\url{https://doi.org/#1}}
\providecommand{\MR}{\relax\ifhmode\unskip\space\fi MR }
\providecommand{\MRhref}[2]{%
  \href{http://www.ams.org/mathscinet-getitem?mr=#1}{#2}
}
\providecommand{\href}[2]{#2}

	}
	
\end{document}